\documentclass[11pt]{article}
\usepackage[utf8]{inputenc}
\usepackage[margin=0.75in]{geometry}
\usepackage[title]{appendix}
\usepackage{amsthm}
\usepackage{url,color}
\usepackage{hyperref}
\usepackage{amsmath, amsfonts}
\usepackage{mathtools}
\usepackage{mathalfa}
\usepackage{amssymb}
\usepackage{mathrsfs}
\usepackage[makeroom]{cancel}
\usepackage{pifont}
\usepackage{upgreek}
\usepackage{times}
\usepackage{calligra}
\usepackage{graphicx}
\usepackage{caption}
\usepackage{amsfonts,amssymb, latexsym, amsmath, amsthm,mathrsfs, verbatim,geometry}
\usepackage{graphics}
\usepackage{slashed}
\usepackage{url,color}
\usepackage{enumerate}
\usepackage{tikz}
\usetikzlibrary{calc}
\usetikzlibrary{calc}
\usetikzlibrary{arrows.meta}
\usetikzlibrary{hobby}
\usepackage{tkz-euclide}
\usepackage[normalem]{ulem}
\usetkzobj{all}

\usepackage{hyperref}
\hypersetup{
    colorlinks,
    citecolor=red,
    filecolor=green,
    linkcolor=blue,
    urlcolor=black
}

\newcommand{\dell}{\partial}

\newcommand{\Ndell}[2]{\rdell^{#1}\angdell^{\underline{#2}}}
\newcommand{\Ndellonetwo}[2]{\rdell^{#1}\angdell_{12}^{#2}}
\newcommand{\comm}[2]{\left[#1,#2\right]}
\newcommand{\X}[1]{\mathcal{X}^{#1}}
\newcommand{\Y}[2]{\mathcal{Y}^{#1}(#2)}
\newcommand{\angdell}{\cancel{\partial}}
\newcommand{\rdell}{\Lambda}
\newcommand{\cov}{\mathcal{M}}
\newcommand{\inv}{\mathcal{A}}
\newcommand{\jac}{\mathcal{J}}
\newcommand{\ncov}{\mathscr{M}}
\newcommand{\ninv}{\mathscr{A}}
\newcommand{\njac}{\mathscr{J}}
\newcommand{\pjac}{\mathscr{J}^{-1/\alpha}}
\newcommand{\rem}{\mathcal{R}}

\newcommand{\spaceI}{\int_{\Omega}}
\newcommand{\twopartdef}[4]
   {
    \left\{
      \begin{array}{ll}
        #1 & \text{if } #2 \\
        #3 & \text{if } #4
      \end{array}
    \right.
   }
\newcommand{\threepartdef}[6]
{
	\left\{
		\begin{array}{lll}
			#1 & \mbox{if } #2 \\
			#3 & \mbox{if } #4 \\
			#5 & \mbox{if } #6
		\end{array}
	\right.
}
\DeclarePairedDelimiterX{\inp}[2]{\langle}{\rangle}{#1, #2}
\DeclareMathOperator{\tr}{Tr}
\DeclareMathOperator{\supp}{supp}
\DeclareMathOperator{\dive}{div}
\DeclareMathOperator{\curl}{curl}
\DeclareMathOperator{\Curl}{Curl}
\DeclareMathOperator{\grad}{\nabla}

\DeclareMathOperator{\ngrad}{\nabla_{\zeta}}
\DeclareMathOperator{\ndiv}{div_{\zeta}}
\DeclareMathOperator{\ncurl}{curl_{\zeta}}
\DeclareMathOperator{\nCurl}{Curl_{\zeta}}
\DeclarePairedDelimiter\ceil{\lceil}{\rceil}

\def\bcr{\begin{color}{red}}
\def\ec{\end{color}}
\def\be{\begin{equation}}
\def\ee{\end{equation}}
\def\bcb{\begin{color}{blue}}
\def\bcv{\begin{color}{violet}}

\newtheorem{definition}{Definition}[section]

\newtheorem{theorem}[definition]{Theorem}

\newtheorem{proposition}[definition]{Proposition}

\newtheorem{lemma}[definition]{Lemma}

\newtheorem{remark}[definition]{Remark}

\title{Global expanding solutions of compressible Euler equations with small initial densities}

\author{Shrish Parmeshwar\footnote{Department of Mathematics, King's College London, Strand, London WC2R 2LS, UK}, \ Mahir Had\v zi\'c\footnote{Department of Mathematics, King's College London, Strand, London WC2R 2LS, UK}, \ 
Juhi Jang\footnote{Department of Mathematics, University of Southern California, Los Angeles, USA and Korea Institute for Advanced Study, Seoul, Republic of Korea}}

\date{}

\begin{document}

\maketitle


\begin{abstract}
We prove the existence of a large class of global-in-time expanding solutions to vacuum free boundary compressible Euler flows without relying on the existence of an underlying finite-dimensional family of special affine solutions of the flow. 
\end{abstract}

\tableofcontents

\section{Introduction}

In this work we prove the existence of a wide class of global solutions to the free boundary isentropic compressible Euler equations expanding into the vacuum, under 
a suitable assumption of {\em smallness} on the initial gas density and an {\em outgoing} condition on the initial velocity profile which roughly speaking states $u_0(x)\sim x$. 
The key property of the Euler flow is the behaviour under scaling of the nonlinearity, 
which generates a stabilising effect under a suitable outgoing condition on the initial velocity profile. 
We work in a vacuum free boundary setting which translates this stabilising effect into the expansion of the gas support. 
Through the mass conservation this will, morally speaking, force the gas density to decay and hence disperse away as $t\to\infty$.

We work with the compressible Euler equations describing the motion of an isentropic and ideal gas. We consider the free-boundary formulation of the problem wherein 
a blob of gas supported on a moving domain $\Omega(t)\subset \mathbb R^3$ is surrounded by vacuum.
Further unknowns are the gas velocity vectorfield ${\bf u}$, the density $\rho$, and the pressure $p$. We assume the isentropic equation of state
\begin{align}
    p=P(\rho) = \rho^{\gamma}, \ \ \gamma>1. \label{equation-of-state}
\end{align}
The associated system of equations takes the form
\begin{subequations}
\label{E:EULER}
\begin{alignat}{2}
    \dell_{t}\rho+\nabla\cdot\left(\rho {\bf u}\right)&=0 && \  \text{in}\ \Omega(t),\label{new-euler-1}\\
    \rho\left(\dell_{t}+{\bf u}\cdot\nabla\right){\bf u}+\nabla\left(\rho^{\gamma}\right)&=0 && \  \text{in}\ \Omega(t),\label{new-euler-2}\\
    \rho&=0 && \ \text{on}\ \dell\Omega(t),\label{new-euler-3}\\
    \mathcal{V}\left(\dell\Omega(t)\right)&={\bf u}\cdot {\bf n} &&  \  \text{on}\ \dell\Omega(t),\label{new-euler-4}
\end{alignat}
\end{subequations}
equipped with initial conditions
\begin{equation}
(\rho,{\bf u})=(\rho_{0},u_{0})\ \ \ \ \text{on}\ \Omega\coloneqq\Omega(0). \label{new-euler-5}
\end{equation}
Here $\mathcal{V}\left(\dell\Omega(t)\right)$ denotes the {\em normal velocity} of $\dell\Omega(t)$, and ${\bf n}$ is the {\em outward unit normal} of $\dell\Omega(t)$. 


A crucial requirement for the well-posedness theory of the moving vacuum boundary Euler equations is the so-called {\em physical vacuum} condition \cite{LY2,JM1}. 
We first define the {\em speed of sound} $c$ through the relationship
\begin{align}
    c^{2}=\frac{d P}{d\rho}=\gamma\rho^{\gamma-1}.\label{speed-of-sound}
\end{align}
Then the {\em physical vacuum boundary condition} reads
\begin{align}
    -\infty <\frac{\dell c^{2}}{\dell n}\bigg\rvert_{\dell\Omega} < 0.\label{vacuum-boundary-condition}
\end{align}
If we set
\begin{align}
w(x): =\rho_{0}^{\gamma-1}(x),\label{w-definition}
\end{align}
then~\eqref{vacuum-boundary-condition} in particular implies that there exists a constant $C>0$ such that 
\begin{align}\label{E:VACUUMCOND}
\frac1C \text{dist}(x,\partial\Omega) \le w(x) \le C  \text{dist}(x,\partial\Omega)
\end{align}
in the vicinity of the initial vacuum boundary $\partial\Omega$. Quantity $w$ is proportional to the enthalpy of the system and it will
play an important role in our analysis.
For notational convenience, define
\begin{align}
\alpha=\frac{1}{\gamma-1},\label{alpha-definition}
\end{align}
so that $\rho_0= w^\alpha$.

There are several works on the expansion-into-vacuum for gases described by the compressible Euler flows. As a rule, the vast majority of such results rely on a reduction of the flow to a finite-dimensional dynamical system, by means of separation of variables arguments for the Lagrangian flow map.
Such reductions of compressible flows with the {\em affine} ansatz on the Lagrangian flow can be tracked back to the works of Ovsiannikov~\cite{Ov1956} and Dyson~\cite{Dyson1968}, where one may obtain solutions with different properties based on the choice of the equation of state. With the choice of the equation of state~\eqref{equation-of-state}, a finite dimensional class of special affine compactly supported expanding solutions of the vacuum free boundary compressible Euler flows was discovered by Sideris~\cite{Sideris,Sideris2014}. Nonlinear stability of such motions was shown by Had\v zi\'c and Jang~\cite{HaJa2} for the range of adiabatic exponents $1<\gamma\le\frac53$ and it was extended to the range $\gamma>\frac53$ by Shkoller and Sideris~\cite{ShSi2017}. 



In the absence of free boundaries, the stabilising effect of the expansion was already understood by Serre~\cite{Se1997} and generalised in the subsequent work by Grassin~\cite{Grassin98}. In the latter work a class of smooth expanding solutions to Euler equations with small initial densities was constructed for any $\gamma>1$. In 2003 Rozanova~\cite{Ro} also showed the existence of  a class of global solutions with expansion as a driving mechanism. The associated velocity profiles are essentially linear (or affine) - this amounts to an assumption on the initial velocity field that drives the expansion of the fluid, 
which in turn overcomes a possible focusing effect that can lead to shock formation.  

It is well-known that the presence of the free boundary causes severe difficulties in the analysis, and a central theme in the well-posedness theory is a delicate 
interplay between the energy estimates and the transport equation satisfied by the vorticity of the velocity vector field \cite{CoSh2012,JaMa2009,JM2012,JaMa2015}. Our analysis is performed entirely in 
Lagrangian coordinates and our solutions are unique in a suitable regularity class which includes the physical vacuum boundary condition~\eqref{vacuum-boundary-condition}.

%

Our main result, Theorem~\ref{main-theorem}, shows by contrast to~\cite{HaJa2,HaJa3,ShSi2017} that the global existence of expanding solutions 
does not crucially depend on the existence of the underlying  ODE-type affine motions. 
Instead, in addition to the stabilising effects of the expansion described in~\cite{HaJa2}, this work exploits an additional scaling structure of the problem, already observed in~\cite{HaJa3} for the Euler-Poisson system, which allows us 
to insert a small parameter in front of the fluid density. 
As a consequence we identify an open class of initial data, with small, but otherwise essentially arbitrary density profiles satisfying~\eqref{E:VACUUMCOND} 
which lead to the global existence for~\eqref{E:EULER}--\eqref{vacuum-boundary-condition}.





\section{Lagrangian Coordinates, Rescaling, and Derivatives}\label{lagrangian-coordinates-rescaling-derivatives}

The well-posedness theory for~\eqref{E:EULER} with the physical vacuum condition~\eqref{vacuum-boundary-condition} was developed independently in~\cite{JaMa2015} and~\cite{CoSh2012} for domains that are periodic in two directions. For the purpose of
this work we rely on the well-posedness framework developed in~\cite{JaMa2015} that has been suitably adapted to handle ball-like domains in~\cite{HaJa2}. 
At the heart of this approach are
the Lagrangian coordinates; they allow us to pull back the free boundary problem onto a fixed domain.

\subsection{Lagrangian Coordinates}\label{lagrangian-coordinates}

In order to address the movement of the vacuum free boundary, we shall reformulate the problem using the Lagrangian coordinates. To that end we introduce the flow map  
$\eta: [0,T]\times \Omega\rightarrow\mathbb{R}^{3}$ as a solution of the ordinary differential equation
\begin{align}
    \dell_{t}\eta(t, x)&={\bf u}(t, \eta(t, x)),\ \ \ t\in [0,T],\nonumber\\
    \eta(0,x)&=x,\label{flow-map}
\end{align}
for some $T>0$.
Let
\begin{align*}
    v(t,\cdot)&={\bf u}(t,\eta(t,\cdot)),\\
    f(t,\cdot)&=\rho(t,\eta(t,\cdot))),\\
    \cov&=D\eta,\\
    \inv&=\cov^{-1},\\
    \jac&=\det{\cov},\\
    a&=\jac\inv.
\end{align*}
The following differentiation identities are useful to note:
\begin{align}
    &\dell\inv^{k}_{i}=-\inv^{k}_{r}\dell \dell_{s}\eta^{r}\inv^{s}_{i},\label{inverse-differentiation}\\
    &\dell\jac=\jac\inv^{s}_{r}\dell\dell_{s}\eta^{r},\label{jacobian-differentiation}
\end{align}
where $\dell = \dell_t$ or $\dell = \dell_i$, $i=1,2,3$. Here and hereafter we use the Einstein summation convention. 
From these, and the definition of $a$, we obtain the Piola identity
\begin{align}
    \dell_{k}a^{k}_{i}=0.\label{piola-identity}
\end{align}
We now pull back~\eqref{E:EULER}--\eqref{new-euler-5} with respect to the Lagrangian coordinates, thus fixing the domain. 
The new system of equations is given by
\begin{align}
    \dell_{t}f+f\inv^{j}_{i}\dell_{j}v^{i}=0\ \ \ \ &\text{in}\ I\times \Omega,\label{lagrangian-euler-1}\\
    f\dell_{t}v^{i}+\inv^{k}_{i}\dell_{k}f^{\gamma}=0\ \ \ \ &\text{in}\ I\times \Omega,\label{lagrangian-euler-2}\\
    f=0\ \ \ \ &\text{on}\ I\times \dell\Omega,\label{lagrangian-euler-3}\\
    (f,v,\eta)=(\rho_{0},u_{0},x)\ \ \ \ &\text{in}\ \{t=0\}\times \Omega .\label{lagrangian-euler-4}
\end{align}
Note that we can use the identity (\ref{jacobian-differentiation}) with the differential operator $\dell_t$ to rewrite $(\ref{lagrangian-euler-1})$ as
\begin{align*}
    \dell_{t}f+f\jac^{-1}\dell_{t}\jac=0.
\end{align*}
This simplifies to
   $ \dell_{t}\left(\log{f\jac}\right)=0, $
which gives the formula
\begin{align}
    f\jac=\rho_{0}=w^{\alpha}.\label{f-jacobian-weight}
\end{align}
Multiplying $(\ref{lagrangian-euler-2})$ by $\jac$, and using $(\ref{piola-identity})$, we arrive at
\begin{align*}
    w^{\alpha}\dell_{t}v^{i}+a^{k}_{i}\dell_{k}\left(\rho^{\gamma}\jac^{-\gamma}\right)=w^{\alpha}\dell_{t}v^{i}+\dell_{k}\left(a^{k}_{i}\rho^{\gamma}\jac^{-\gamma}\right)=0.
\end{align*}
Using the identities $\rho^{\gamma}=w^{1+\alpha}$ and $a^{k}_{i}\jac^{-\gamma}=\inv^{k}_{i}\jac^{1-\gamma}=\inv^{k}_{i}\jac^{-1/\alpha}$,  we finally obtain 
\begin{subequations}
\label{E:EULERNEW}
\begin{alignat}{2}
    w^{\alpha}\dell_{tt}\eta^{i}+\dell_{k}\left(w^{1+\alpha}\inv^{k}_{i}\jac^{-1/\alpha}\right)&=0 &&
    \ \text{in} \  I\times \Omega,\label{new-lagrangian-euler-1}\\
    (v,\eta)&=(u_{0},x) && \ \text{in} \ \{t=0\}\times \Omega,\label{new-lagrangian-euler-3}\\
    w&=0 && \ \text{on}\ \dell\Omega.\label{new-lagrangian-euler-2}
\end{alignat}
\end{subequations}


\subsection{Rescaling and a new formulation}\label{rescaling}


As the expected mechanism for the global existence is the expansion of the support of $\Omega(t)$, we pass to a new set of variables:
\begin{align}
    &\tau=\log{(1+t)},\label{rescaled-time}\\
    &\zeta=e^{-\tau}\eta.\label{rescaled-flow}
\end{align}
We then define
\be
\nu=\dell_{\tau}\zeta, \ \  \mathscr{M}=D\zeta, \ \  \mathscr{A}=[D\zeta]^{-1}, \ \ \mathscr{J}=\det{\mathscr{M}}, \ \ \tilde{a}=\mathscr{J}\mathscr{A}.\label{nu-jacobian-inverse-definition}
\ee
Comparing with Lagrangian coordinates, we have similar differentiation formulae for the jacobian, and inverse of the gradient, given by
\begin{align}
    &\dell\ninv^{k}_{i}=-\ninv^{k}_{r}\dell \dell_{s}\zeta^{r}\ninv^{s}_{i},\label{inverse-differentiation-zeta}\\
    &\dell\njac=\njac\ninv^{s}_{r}\dell\dell_{s}\zeta^{r},\label{jacobian-differentiation-zeta}
\end{align}
giving rise to an analogous Piola identity:
\begin{align}
    \dell_{k}\tilde{a}^{k}_{i}=0. \label{piola-identity-zeta}
\end{align}
Now the chain rule gives us
\begin{align*}
    \dell_{t}=\frac{d\tau}{dt}\dell_{\tau}=\frac{1}{1+t}\dell_{\tau}=e^{-\tau}\dell_{\tau}.
\end{align*}
Applying this to the first term in (\ref{new-lagrangian-euler-1}), we get
\begin{align*}
    \dell_{t}\eta(t,x)=\frac{de^{\tau}}{dt}\zeta+e^{\tau}\frac{\dell\zeta}{\dell\tau}=e^{-\tau}e^{\tau}\zeta+e^{\tau}e^{-\tau}\zeta_{\tau}=\zeta(\tau,x)+\zeta_{\tau}(\tau,x),
\end{align*}
which implies 
\begin{align*}
    \dell_{tt}\eta=e^{-\tau}\zeta_{\tau}+e^{-\tau}\zeta_{\tau\tau}.
\end{align*}
Now, for the pressure term, we can see that $\njac=\det{D\zeta}=\det{\left(e^{-\tau}D\eta\right)}=e^{-3\tau}\jac$. Finally, $\cov^{k}_{i}=e^{\tau}\ncov^{k}_{i}$, and therefore $\inv^{k}_{i}=e^{-\tau}\ninv^{k}_{i}$. Combining these identities allows us to rewrite (\ref{new-lagrangian-euler-1}) as
\begin{align*}
    w^{\alpha}\left(e^{-\tau}\zeta^{i}_{\tau}+e^{-\tau}\zeta^{i}_{\tau\tau}\right)+e^{-\left(1+\frac{3}{\alpha}\right)\tau}\dell_{k}\left(w^{1+\alpha}\ninv^{k}_{i}\pjac\right)=0.
\end{align*}
Multiplying by the appropriate power of $e^{\tau}$ to get rid of exponential terms on the pressure term we get
\begin{align}
    w^{\alpha}e^{\beta\tau}\left(\zeta^{i}_{\tau\tau}+\zeta^{i}_{\tau}\right)+\dell_{k}\left(w^{1+\alpha}\ninv^{k}_{i}\pjac\right)=0,\label{rescaled-euler-w}
\end{align}
where for clarity, we define 
\begin{align}
\beta\coloneqq\frac{3}{\alpha}=3(\gamma-1).\label{beta-definition}
\end{align}
In this paper we assume
\begin{align}
\Omega(0)=\Omega=B_{1},\label{initial-domain-is-unit-ball}
\end{align}
the closed unit ball on $\mathbb{R}^{3}$. Assumption~\eqref{initial-domain-is-unit-ball} will simplify some of the technical steps in the proof, but we can easily treat more general domains that are small smooth deformations of domains of the form $AB_1(0)$, for $A\in \text{GL}^+(3)$ (i.e. small deformations of ellipsoids).

We also introduce an ansatz for the weight function $w$.
\begin{definition}\label{definition-of-W}
	Let $W:\Omega\rightarrow\mathbb{R}$ be a given nonnegative function such that 
	\begin{itemize}
	\item $W>0$ on $\text{int}(\Omega)$ and
	\begin{align}
	W\rvert_{\dell\Omega}\equiv0.\label{W-0-on-boundary}
	\end{align} 
	\item 
	There exists a positive constant $C>0$ such that  for any $x\in\Omega$
	\begin{align}
	\frac1 Cd(x,\dell\Omega)\leq W(x)\leq C d(x,\dell\Omega),\label{W-and-d-equivalence}
	\end{align}
 where $x\mapsto d(x,\partial\Omega)$ is the distance function to $\partial\Omega$.
 	\item The function given by
 		\begin{align}
 		x\mapsto\frac{W(x)}{d(x,\Omega)}\label{ratio-of-W-d-smooth}
 		\end{align}
 		is smooth on a neighbourhood of the boundary $\dell\Omega$.
 	
	\end{itemize}
	For any $\delta>0$ we consider the enthalpy profile 
	\begin{align*}
	w(x)=w_\delta(x)=\delta W(x).
	\end{align*}
	The set $\{w_\delta\}_{\delta\in\mathbb R_{>0}}$ forms a 1-parameter family of initial enthalpies, which generate a 1-parameter family of the corresponding initial densities via
	\[
	\rho_0^\delta(x) = w_\delta(x)^\alpha.
	\]
\end{definition}

\begin{remark}
A model example of a function $W$ satisfying the above assumptions is $W(r)=(1-r^2)_+$.
\end{remark}


\begin{remark}\label{smallness-of-delta} 
The parameter $\delta>0$ will be assumed small in our work and will be used as an effective measure of smallness for the initial gas density.
The freedom to make $\delta>0$ small will be crucial in our strategy, as it will also be used as an additional small factor in the closing our estimates. Without this smallness our strategy would fail. 
\end{remark}

From now on we drop the explicit $\delta$-dependence in $w_\delta$ and simply write $w$. Using such a choice of $w$ in~\eqref{rescaled-euler-w} and Definition~\ref{definition-of-W}, upon dividing~\eqref{rescaled-euler-w} by $\delta^{1+\alpha}$, we obtain
\begin{align}
    \frac{1}{\delta}W^{\alpha}e^{\beta\tau}\left(\zeta^{i}_{\tau\tau}+\zeta^{i}_{\tau}\right)+\dell_{k}\left(W^{1+\alpha}\ninv^{k}_{i}\pjac\right)=0.\label{rescaled-euler-W}
\end{align}


We shall look for solutions where $\zeta$ is a small perturbation of the identity map. To that end, we introduce   
\[
\theta\coloneqq\zeta-x,
\] 
and note that $\theta_{\tau}=\zeta_{\tau}$. Then~\eqref{rescaled-euler-W} reads
\begin{align}
    \frac{1}{\delta}W^{\alpha}e^{\beta\tau}\left(\theta^{i}_{\tau\tau}+\theta^{i}_{\tau}\right)+\dell_{k}\left(W^{1+\alpha}\ninv^{k}_{i}\pjac\right)=0.\label{rescaled-euler-W-perturbation-estimate-form-beta-leq-2}
\end{align}

Muliplying~\eqref{rescaled-euler-W-perturbation-estimate-form-beta-leq-2} by $\theta_i$ and integrating over $\Omega$, we arrive at 
\[
\frac1\delta\frac12\frac{d}{d\tau}\int_{\Omega}e^{\beta\tau} |\theta|^2 W^\alpha\,dx
+ \frac1\delta\left(1-\frac{\beta}{2}\right) \int_{\Omega} e^{\beta\tau} |\theta|^2 W^\alpha\,dx
+ \left(\left(W^{1+\alpha}\ninv^{k}_{i}\pjac\right),\theta^i\right)_{L^2(\Omega)}=0
\]
Therefore, to guarantee the non-negativity  of the second term on the left-hand side above, it appears  necessary to assume $1-\beta/2\ge0$, which is in turn equivalent to 
$\gamma\leq 5/3$. This apparent restriction is analogous to the one in~\cite{HaJa2}.
A way to go around this when $\beta>2$ has was introduced in~\cite{ShSi2017}: we multiply~\eqref{rescaled-euler-W-perturbation-estimate-form-beta-leq-2} by $e^{(2-\beta)\tau}$ 
and obtain
\begin{align}
    \frac{1}{\delta}W^{\alpha}e^{2\tau}\left(\theta^{i}_{\tau\tau}+\theta^{i}_{\tau}\right)+e^{(2-\beta)\tau}\dell_{k}\left(W^{1+\alpha}\ninv^{k}_{i}\pjac\right)=0.\label{rescaled-euler-W-perturbation-estimate-form-beta-g-2}
\end{align}
This removes the above mentioned issue with the potentially wrongly signed damping term, at the expense of a negative exponential in front of the pressure term. Nevertheless, we will be able to close our estimates, thereby allowing some of the spatial norms to grow as $\tau\to\infty$ when $\beta>2$.

To unify the two cases
we introduce some more notation
\begin{definition}\label{D:SIGMA12}
For any $\beta\in(0,\infty)$ we define
\begin{align}
    &\sigma_{1}(\beta)\coloneqq\twopartdef{\beta}{\beta\leq2}{2}{\beta>2},\label{sigma-1}\\
    &\sigma_{2}(\beta)\coloneqq\twopartdef{0}{\beta\leq2}{\beta-2}{\beta>2}.\label{sigma-2}
\end{align}
\end{definition}

\begin{remark}
Note that $\sigma_{1}(\beta)+\sigma_{2}(\beta)=\beta$.
\end{remark}

We shall drop the explicit $\beta$ dependence in $\sigma_1(\beta),\sigma_2(\beta)$ and write instead 
$\sigma_1, \sigma_2$ respectively. We recall that $\beta=\frac3\alpha = 3(\gamma-1)$.
We may therefore rewrite our system succinctly as
\begin{subequations}
\label{E:EULERNEW2}
\begin{alignat}{2}
    \frac{1}{\delta}W^{\alpha}e^{\sigma_{1}\tau}\left(\theta^{i}_{\tau\tau}+\theta^{i}_{\tau}\right)+e^{-\sigma_{2}\tau}\dell_{k}\left(W^{1+\alpha}\ninv^{k}_{i}\pjac\right)&=0 && \ \text{in}\ I \times  \Omega,\label{rescaled-euler-W-perturbation-estimate-form-1}\\
    (\nu,\theta)&=(u_{0}-x,0) && \ \text{in}\ \{\tau=0\}\times \Omega,\label{rescaled-euler-W-perturbation-estimate-form-2}
\end{alignat}
\end{subequations}
where the profile $W$ is given in Definition~\ref{definition-of-W}, $\delta>0$ is a constant, and $\sigma_i$, $i=1,2$, are given in Definition~\ref{D:SIGMA12}.

\section{Notation}\label{notation}

\subsection{General Notation}
For a function $F:\mathcal O\rightarrow\mathbb{R}$, some domain $\mathcal O$, the support of $F$ is denoted $\supp{F}$. For a real number $\lambda$, the ceiling function, denoted $\ceil*{\lambda}$, is the smallest integer $M$ such that $\lambda\leq M$. For two real numbers $A$ and $B$, we say $A\lesssim B$ if there exists a positive constant $C$ such that
\begin{align}
A\leq CB,\label{lesssim-def}
\end{align}
and for two real valued functions $f$ and $g$, we say $f\lesssim g$ if $f(x)\lesssim g(x)$ holds pointwise. For two real-valued non-negative functions $f,g: \mathcal O\rightarrow \mathbb{R}_{\geq0}$, some domain $\mathcal O$, we say $f\sim g$ if there exist positive constants $c_{1}$ and $c_{2}$ such that
\begin{align}
c_{1}g(x)\leq f(x)\leq c_{1}g(x),\label{two-equivalent-functions-def}
\end{align}
for all $x\in \mathcal O$. 

For a collection of rank $2$ tensors $M[i]$,  
$i=1,\dots,j$ by
\begin{align}
M[1]\dots M[j],\label{schematically}
\end{align}
we mean a particular element of the rank $2^{j}$ tensor $M[1]\otimes\dots\otimes M[j]$, that is, something of the form
\begin{align}
M[1]^{k_{1}}_{i_{1}}\dots M[j]^{k_{j}}_{i_{j}},\label{schematically-2}
\end{align}
for some $k_{1},i_{1},\dots,k_{j},i_{j}\in\{1,2,3\}$.
We say this is a schematic representation of the object. Note that this notation is ambiguous and will only be used when we are looking to bound such a quantity using the properties of $M[1],\dots,M[j]$ themselves, rather than any of their specific elements.

We also record the definition of the radial function on $\Omega=B_{1}$ the unit ball:
\begin{align}
r:B_{1}&\rightarrow\mathbb{R}_{\ge0}\nonumber\\
x&\mapsto r(x)\coloneqq |x|.\label{radial-function}
\end{align}
It is convenient to define shorthand for the distance function on $\Omega$. Define
\begin{align}
d_{\Omega}(x)=d(x,\dell\Omega).\label{distance-function-shorthand}
\end{align}
\subsection{Derivatives}\label{derivatives}

As we have seen above, rectangular derivatives will be denoted as $\dell_{i}$, for $i$ in $1,2,3$. In addition, we define various rectangular and $\zeta$ Lie derivatives that will be used throughout. The gradient, divergence, and curl on vector fields are given by
\begin{align}
&[\nabla F]^{i}_{j}=\dell_{j}F^{i},\label{gradient}\\
&\dive{F}=\dell_{i}F^{i},\label{divergence}\\
&[\curl{F}]^{i}=\varepsilon_{ijk}\dell_{j}F^{k},\label{curl}
\end{align}
for $i,j=1,2,3.$

The $\zeta$ versions are given by
\begin{align}
&[\ngrad{F}]^{i}_{j}=\ninv^{k}_{j}\dell_{k}F^{i},\label{zeta-gradient}\\
&\ndiv{F}=\ninv^{k}_{i}\dell_{k}F^{i},\label{zeta-divergence}\\
&[\ncurl{F}]^{i}=\varepsilon_{ijk}\ninv^{s}_{j}\dell_{s}F^{k}.\label{zeta-curl}
\end{align}
In addition, we also need the matrix $\zeta$ curl, given by
\begin{align}
[\nCurl{F}]^{i}_{j}=\ninv^{s}_{j}\dell_{s}F^{i}-\ninv^{s}_{i}\dell_{s}F^{j}.\label{zeta-Curl}
\end{align}

As stated in $(\ref{initial-domain-is-unit-ball})$, our initial domain will be the closed unit ball in $\mathbb{R}^{3}$, so $\Omega=B_{1}$. Therefore, there exists a natural choice of spherical coordinates $(r,\omega,\phi)$. An advantage of this choice of domain is that we can privilege the outward normal derivative, the direction in which the degeneracy of the problem occurs, due to the vacuum boundary condition.

Accordingly, in essence we use $\dell_{r}$ as the normal derivative, and $\dell_{\omega},\dell_{\phi}$ as the tangential derivatives. However we modify these derivatives by using linear combinations. These modifications allow for better commutation relations with the rectangular derivatives.

Let the angular derivatives $\angdell_{ij}$ and radial derivative $\rdell$ be given by
\begin{align}
\angdell_{ij} &:=x_{i}\dell_{j}-x_{j}\dell_{i},\label{angular-derivative-def}\\
\rdell &: = x_i\partial_i =r\dell_{r},\label{radial-derivative-def}
\end{align}
where the $x_{i}$ and $\dell_{j}$ are rectangular, and $i,j$ run through $1,2,3$.


\begin{remark}\label{degeneracy-modified-spherical-derivatives}
	The coefficients of the derivatives we have defined go to 0 at the origin, which means we can only use them to do estimates on a region separated from the origin. This can be dealt with using a partition of unity argument. Near the boundary we use these modified spherical derivatives, and on the interior, we are free to use rectangular derivatives as the degeneracy at the vacuum boundary is not an issue in this case.
\end{remark}

Now, for $m\in\mathbb{Z}_{\geq0}$, and $\underline{n}=(n_{1},n_{2},n_{3})\in\mathbb{Z}^{3}_{\geq0}$, we define
\begin{align}
\Ndell{m}{n}\coloneqq \rdell^{m}\angdell_{12}^{n_{1}}\angdell_{13}^{n_{2}}\angdell_{23}^{n_{3}}.\label{Ndell-def}
\end{align}
Although there are six non-zero $\angdell$ derivatives to consider, $\angdell_{ij}=-\angdell_{ji}$, so $(\ref{Ndell-def})$ covers all cases. For such an $\underline{n}\in\mathbb{Z}^{3}_{\geq0}$, $|\underline{n}|=n_{1}+n_{2}+n_{3}$.

Similarly for rectangular derivatives we define, for $\underline{k}=(k_{1},k_{2},k_{3})\in\mathbb{Z}_{\geq0}^{3}$,
\begin{align}
\nabla^{\underline{k}}=\dell_{1}^{k_{1}}\dell_{2}^{k_{2}}\dell_{3}^{k_{3}}.\label{Dk-def}
\end{align}

We have the commutation relations between the modified spherical  and rectangular derivatives, for $i,j,k,m\in\{1,2,3\}$, given by
\begin{align}
&\comm{\angdell_{ij}}{\rdell}=0,\label{commutator-ang-rad}\\
&\comm{\angdell_{ij}}{\angdell_{jk}}=\angdell_{ik},\label{commutator-ang-ang}\\
&\comm{\dell_{m}}{\rdell}=\dell_{m},\label{commutator-rectangular-rad}\\
&\comm{\dell_{m}}{\angdell_{ji}}=\delta_{mj}\dell_{i}-\delta_{mi}\dell_{j}.\label{commutator-rectangular-ang}
\end{align}
We also define commutators between the higher order differential operator defined in $(\ref{Ndell-def})$, and $\grad$:
\begin{align}
\left(\comm{\grad}{\Ndell{m}{n}}F\right)^{i}_{j}&=\dell_{j}\left(\Ndell{m}{n}F^{i}\right)-\Ndell{m}{n}\left(\dell_{j}F^{i}\right).\label{commutator-grad-Ndell}
\end{align}
We can do the same thing with $\ngrad$:
\begin{align}
\left(\comm{\ngrad}{\Ndell{m}{n}}F\right)^{i}_{j}&=\ninv^{k}_{j}\dell_{k}\left(\Ndell{m}{n}F^{i}\right)-\Ndell{m}{n}\left(\ninv^{k}_{j}\dell_{k}F^{i}\right),\label{commutator-ngrad-Ndell}\\
\left(\comm{\ngrad}{\nabla^{\underline{k}}}F\right)^{i}_{j}&=\ninv^{k}_{j}\dell_{k}\left(\nabla^{\underline{k}}F^{i}\right)-\nabla^{\underline{k}}\left(\ninv^{k}_{j}\dell_{k}F^{i}\right).\label{commutator-ngrad-Dk}
\end{align}
There is no corresponding definition to $(\ref{commutator-ngrad-Dk})$ for $\grad$, as $\grad$ and $\nabla^{\underline{k}}$ commute for all $\underline{k}\in\mathbb{Z}_{\geq0}^{3}$. Note that $(\ref{commutator-ngrad-Ndell})$ and $(\ref{commutator-ngrad-Dk})$ also define analogous objects for $\nCurl$ and $\ndiv$ as the former is $\ngrad-\ngrad^{\intercal}$, and the latter is $\tr{\ngrad}$.

We also use the following decomposition frequently:
\begin{align}
\dell_{i}=\frac{x_{j}}{r^{2}}\angdell_{ji}+\frac{x_{i}}{r^{2}}\rdell.\label{rectangular-as-ang-rad}
\end{align}

\subsection{Higher Order Energies and Local Well-Posedness}\label{higher-order-energies-and-lwp}

We now define the function spaces, and their associated norms, that we use to prove global-in-time solutions for~\eqref{E:EULERNEW2} with small initial data.

As noted in Remark~\ref{degeneracy-modified-spherical-derivatives}, our choice of derivatives requires a separation of the analysis in two parts; near the boundary of the ball, and near the origin. Our choice of function spaces will reflect this fact.
Let $0<r_{1}<r_0<1$ be given and define a smooth cutoff function $\psi$ on the unit ball such that
\begin{align}
    \psi=\twopartdef{1}{r\in[r_{0},1]}{0}{r\in[0,r_{1}]},\label{cutoff-function}
\end{align}
and such that $W/d_{\Omega}$ is smooth on $\supp{\psi}$, as in (\ref{ratio-of-W-d-smooth}). In addition define $\bar{\psi}$ by
\begin{align}
	\bar{\psi}=1-\psi.\label{cutoff-function-conjugate}
\end{align}

Before we define our energy spaces, recall the definition of $\alpha$ given in $(\ref{alpha-definition})$ and $W$ in Definition \ref{definition-of-W}, and that our initial domain $\Omega$ is given by the closed unit ball $B_{1}$ on $\mathbb{R}^{3}$.
\begin{definition}{\label{energy-function-spaces-definition}}
    Let $b\in\mathbb{Z}_{\geq0}$ and define the space $\X{b}$ by
    \begin{align*}
        \X{b}=\left\{W^{\frac{\alpha}{2}}F\in L^{2}(\Omega): \spaceI W^{\alpha+m}\psi\left|\Ndell{m}{n}F\right|^{2}+W^\alpha\bar{\psi}\left|\nabla^{\underline{k}}F\right|^{2}dx<\infty, 0\leq \max{(m+|\underline{n}|,|\underline{k}|)}\leq b\right\}.
    \end{align*}
    The norm of $\X{b}$ is given by
    \begin{align}
    \left\|F\right\|^{2}_{\X{b}}=\sum_{m+|\underline{n}|=0}^{b}\spaceI\psi W^{\alpha+m}\left|\Ndell{m}{n}F\right|^{2}dx+\sum_{|\underline{k}|=0}^{b}\spaceI\bar{\psi} W^{\alpha}\left|\nabla^{\underline{k}}F\right|^{2}dx.\label{energy-space-norm}
    \end{align}
    For $\mathscr{D}\in\{\grad,\ngrad,\dive,\ndiv,\Curl,\nCurl\}$ define the set $\Y{b}{\mathscr{D}}$ by
    \begin{align*}
    \Y{b}{\mathscr{D}}=   \Big\{ &W^{\frac{1+\alpha}{2}}\mathscr{D}{F}\in L^{2}(\Omega): \spaceI W^{1+\alpha+m}\pjac\psi\left|\mathscr{D}{\Ndell{m}{n}F}\right|^{2}+W^{1+\alpha}\bar{\psi}\left|\mathscr{D}{\nabla^{\underline{k}}F}\right|^{2}dx<\infty, \\
    &0\leq \max{(m+|\underline{n}|,|\underline{k}|)}\leq b\Big\},
    \end{align*}
    and associate to it the quantity
    \begin{align}
    \left\|F\right\|^{2}_{\Y{b}{\mathscr{D}}}=\sum_{m+|\underline{n}|=0}^{b}\spaceI\psi W^{1+\alpha+m}\pjac\left|\mathscr{D}{\Ndell{m}{n}F}\right|^{2}dx+\sum_{|\underline{k}|=0}^{b}\spaceI\bar{\psi} W^{1+\alpha}\pjac\left|\mathscr{D}{\nabla^{\underline{k}}F}\right|^{2}dx.\label{energy-space-norm-2}
    \end{align}
\end{definition}
\begin{remark}\label{power-of-W-on-interior-norm}
By the assumption~\eqref{W-and-d-equivalence}, $W$ is bounded uniformly from below and above by some constant depending on $r_1$ in all of the above integrals that are weighted by $\bar\psi$. In particular, it is strictly speaking superfluous to keep the corresponding powers of $W$ in such integrals, but it provides a notational unity in the derivation of various energy identities and does not create any issues. 
\end{remark}

\begin{remark}\label{X-banach-space-Y-not}
It is important to note that while $\X{b}$ is a Banach space with $\left\|\cdot\right\|_{\X{b}}$ as its norm, the same is not true for $\Y{b}{\mathscr{D}}$ for any of the possible choices of $\mathscr{D}$ given in Definition \ref{energy-function-spaces-definition}. Nevertheless, the non-negative quantity $\left\|\cdot\right\|_{\Y{b}{\mathscr{D}}}$ is crucial as it forms part of our higher order energy function.
\end{remark}

We now define our higher order energy function and $\nCurl$ energy functions. Recall $\delta$ from Definition \ref{definition-of-W}. 

\begin{definition}\label{energy-function-for-solutions} Let $(\nu,\theta)$ be the solution to~\eqref{E:EULERNEW2} on $[0,T]$ in the sense of Theorem $\ref{local-well-posedness-theorem}$, for some $N\geq2\ceil*{\alpha}+12$. We define, for all $\tau\in[0,T]$, the {\em total energy}
	\begin{align}
	\mathscr{S}_{N}(\tau)
	& = S_{N}(\nu,\theta,\tau) \notag \\
	& \coloneqq \sup_{0\leq\tau'\leq\tau}\left(\frac{1}{\delta}e^{\sigma_{1}\tau'}\left\|\nu(\tau')\right\|_{\X{N}}^{2}+\left\|\theta(\tau')\right\|_{\X{N}}^{2}+e^{-\sigma_{2}\tau'}\left\|\theta(\tau')\right\|_{\Y{N}{\ngrad}}^{2}+\frac{1}{\alpha}e^{-\sigma_{2}\tau'}\left\|\theta(\tau')\right\|_{\Y{N}{\ndiv}}^{2}\right),\label{energy-function-for-solutions-statement}
	\end{align}
	and the {\em curl energy}
	\begin{align}
	\mathscr{C}_{N}(\tau) 
	& = C_{N}(\nu,\theta,\tau)\notag \\
& \coloneqq \sup_{0\leq\tau'\leq\tau}e^{-\sigma_{2}\tau'}\left( \left\|\nu(\tau')\right\|_{\Y{N}{\nCurl}}^{2}+\left\|\theta(\tau')\right\|_{\Y{N}{\nCurl}}^{2}\right).
	\label{Curl-energy-function-for-solutions-statement}
	\end{align}
\end{definition}


The local well-posedness theory for \eqref{new-lagrangian-euler-1}--\eqref{new-lagrangian-euler-2} is given in~\cite{JaMa2015}. From this proof we can adapt an argument to show local well-posedness for~\eqref{E:EULERNEW2}.

\begin{theorem}[Local Well-Posedness of the Rescaled Free Boundary Euler System]\label{local-well-posedness-theorem}Let $N$ be an integer such that $N\geq2\ceil*{\alpha}+12$, with $\alpha$ defined in $(\ref{alpha-definition})$. Additionally, assume that $\nabla W\in\X{N}$. Let $u_{0}$ be such that
	\begin{align*}
	 \mathscr{S}_{N}(0)+\mathscr{C}_{N}(0)=S_{N}(u_{0}-x,0,0)+C_{N}(u_{0}-x,0,0)<\infty.
	\end{align*}
	Then, there exists a $T>0$ such that there exists a unique solution $(\nu,\theta)$ to~\eqref{E:EULERNEW2} on the interval $[0,T]$ such that
	\begin{align*}
	& \mathscr{S}_{N}(\tau) + \mathscr{C}_{N}(\tau) 
	\leq2(\mathscr{S}_{N}(0)+\mathscr{C}_{N}(0)+\sqrt\delta),\ \ \ \forall\tau\in[0,T],\\
	&(\nu(0),\theta(0))=(u_{0}-x,0).
	\end{align*}
	Moreover, the function $\tau\mapsto S_{N}(\tau)$ is continuous.
\end{theorem}


\section{Main Theorem and A Priori Assumptions}\label{main-theorem-a-priori-assumptions}


Now we state the main theorem. Recall the definition of $\alpha$ in $(\ref{alpha-definition})$, as well as $W$ and $\delta$ in Definition \ref{definition-of-W}.

\begin{theorem}[Global Existence of Expanding Solutions with Small Initial Data]{\label{main-theorem}} Let $\gamma\in(1,\infty)$, or equivalently, $\alpha\in(0,\infty)$. Let $N$ be an integer such that $N\geq 2\ceil*{\alpha}+12$. Assume $\nabla W\in\X{N}$. Then there exists sufficiently small $\delta, \varepsilon_{0}>0$ such that for all $0\leq\varepsilon\leq\varepsilon_{0}$, and $u_{0}$ with
\begin{align*}
    \mathscr{S}_{N}(0)+\mathscr{C}_{N}(0)=S_{N}(u_{0}-x,0,0)+C_{N}(u_{0}-x,0,0)\leq\varepsilon,
\end{align*}
there exists a global-in-time solution to~\eqref{E:EULERNEW2}, with $(u_{0}-x,0)$ as the initial conditions, and 
\begin{align}
\mathscr{S}_{N}(\tau)\leq C\left(\varepsilon+\sqrt\delta\right),\ \ \ \forall\tau\in[0,\infty),\label{global-energy-estimate}
\end{align}
some constant $C$. 
Finally, there exists a $\tau$ independent function $\theta_{\infty}:B_{1}\rightarrow\mathbb{R}^{3}$ such that
\begin{align}
\left\|\theta(\tau)-\theta_{\infty}\right\|_{\X{N}}\rightarrow0\ \ \ \tau\rightarrow\infty.\label{theta-infinity-statement}
\end{align}
\end{theorem}

\begin{remark}[Initial density profiles]
The above theorem in particular implies that there exists a global solution to~\eqref{new-lagrangian-euler-1}--\eqref{new-lagrangian-euler-3} with initial densities
of the form $\rho_0 = w_\delta^\alpha = \delta^\alpha W^\alpha$ and the initial domain $\Omega=B_1(0)$. Therefore, apart for the smallness parameter $\delta$ and the physical vacuum 
condition~\eqref{ratio-of-W-d-smooth}, the density profile is essentially arbitrary! This highlights the difference to the results relying on finite-dimensional reductions of the flow~\cite{Sideris,Ov1956,Dyson1968}.
\end{remark}

\begin{remark}[Initial velocity profiles]
Unlike the densities, there is more rigidity for the initial Eulerian velocity profile, as it takes the form of a nearly linear field
\[
u_0(x) = x + \partial_\tau\theta(0,x) = x + O(\sqrt{\varepsilon\delta}).
\]
Thus, to leading order $u_0(x)\sim x$ and this assumption encodes the expansive nature of our flow.
\end{remark}

\begin{remark}
Unwinding the change of variables~\eqref{rescaled-time}--\eqref{rescaled-flow}, Theorem~\ref{main-theorem} gives us the formulas
\begin{align}
\eta(t,x) &= (1+t) \left(x +  O(\sqrt{\varepsilon}+\delta^{\frac14}) \right), \notag \\
v(t,x) = \partial_t\eta(t,x) &=x +  O(\sqrt{\varepsilon\delta}+\delta^{\frac34}) 
\notag
\end{align} 
where we recall Definition~\ref{D:SIGMA12}.
\end{remark}

\begin{remark}
Observe that the expansion as a mechanism for global existence is a stable phenomenon, since our initial data form an open set in a suitable topology.
\end{remark}

We note that our perturbed equations \eqref{E:EULERNEW2} do not admit $\theta\equiv 0$ as a solution, whereas in \cite{HaJa2} $\theta\equiv 0$ of the perturbed equations is the solution and it represents the Sideris' affine motions with a correct choice of the density. Hence it is a priori not clear why one would expect global solutions near $\theta=0$. A key to success of achieving the global-in-time solutions without being close to the affine motions is the scaling structure of the Euler equations that grants a small parameter $\delta$. With sufficiently small $0<\delta\ll1$ as well as the stabilising effect of the coefficient $e^{\sigma_1\tau}$, one would expect that $\theta_\tau$ in \eqref{E:EULERNEW2} decays and $\theta$ stays small for large time if initial data are sufficiently small.

 In order to capitalise on this viewpoint and prove Theorem \ref{main-theorem}, we adapt the weighted energy methods developed in \cite{HaJa,HaJa2,HaJa3,JaMa2015}. We use suitable spatial vector fields to deal with the vacuum degeneracy and obtain the leading order energy by high order energy estimates and curl estimates.  We keep track of the $\delta$ dependence and critically use the smallness of $\delta$ to continue the solution for all forward in time. 

Our method gives a unified treatment for all $\gamma>1$. In order to treat all $\gamma>1$, we design $\gamma$-dependent time weights in the energy norms, which still exhibit the stabilising effect and allow some growth for spatial norms in the energy at the top order for $\gamma>\frac53$. In this regime ($\gamma>\frac53$) we control the lower order spatial norms, inspired by \cite{ShSi2017},  by making use of the fundamental theorem of calculus in time variable to obtain their boundedness (without any growth in time) by the given total energy.

\begin{remark}
It is not hard to see that the analogue of Theorem~\ref{main-theorem} holds for the Euler-Poisson system in both the gravitational and the plasma case. However, in that case, due to the nature of the nonlocal forcing term, the same restriction on the values of $\gamma$ as in~\cite{HaJa3} applies, i.e. we allow $\gamma = 1+\frac1n$, $n\in\mathbb N$, $n\ge2$ or $\gamma\in(1,\frac{14}{13})$.
\end{remark}

Here we also state a priori assumptions that will be used to prove the energy estimates. In turn we use the energy estimates to improve upon these assumptions, thereby closing the proof via a continuation argument. The assumptions are, for a solution $(\nu,\theta)$ on $[0,T]$:
\begin{align}
\mathscr{S}_{N}(\tau)\leq \frac{1}{3}\ \ \ \left\|\ninv-I\right\|_{L^{\infty}(\Omega)}\leq \frac{1}{3}\ \ \ \left\|\njac-1\right\|_{L^{\infty}(\Omega)}\leq \frac{1}{3}.\label{a-priori-assumptions}
\end{align}

For sufficiently small enough initial data, the assumptions $(\ref{a-priori-assumptions})$ are initially true. Then the local well-posedness theory for this system ensures these bounds will hold for at least some, possibly short period of time. These conditions also ensure the invertibility of $\nabla\zeta$ on the time interval of existence for the solution.

\section{Lower Order Estimates}\label{theta-estimates}

In this section we record several estimates of the lower order terms.
We start with a lemma that explains how $\left\|\theta\right\|_{\X{N}}$ is bounded in terms of the energy function defined in Definition~\ref{energy-function-for-solutions}.


\begin{lemma}\label{theta-estimates-lemma}
	Let $(\nu,\theta)$ be the solution to~\eqref{E:EULERNEW2} on $[0,T]$ in the sense of Theorem~\ref{local-well-posedness-theorem}, for some $N\geq2\ceil*{\alpha}+12$. Suppose the a priori assumptions $(\ref{a-priori-assumptions})$ hold. Then for all $\tau\in[0,T]$, we have
	\begin{align*}
		\left\|\theta\right\|_{\X{N}}^{2}\lesssim \delta \mathscr{S}_{N}(\tau).
	\end{align*}
\end{lemma}

\begin{remark}\label{theta-estimates-extra-smallness-remark}
	The $\delta$ in front of the energy function provides an extra source of smallness that we will use repeatedly in our estimates. 
\end{remark}

\begin{proof}The quantity $\left\|\theta\right\|_{\X{N}}$ is defined in $(\ref{energy-space-norm})$. Let $m\in\mathbb{Z}_{\geq0}$, and let $\underline{n}\in\mathbb{Z}_{\geq0}^{3}$ be such that $m+|\underline{n}|\leq N$.
	First we observe
	\begin{align}
	\sqrt{\psi}W^{\frac{\alpha+m}{2}}\Ndell{m}{n}\theta^{i}=\int_{0}^{\tau}\sqrt{\psi}W^{\frac{\alpha+m}{2}}\Ndell{m}{n}\nu^{i}d\tau'.\label{theta-ftc-at-zero}
	\end{align}
	Here we make use of $\theta(0,x)=0$ for all $x\in\Omega$.  Taking the square of the norm on both sides above gives
	\begin{align*}
	\psi W^{\alpha+m}\left|\Ndell{m}{n}\theta\right|^{2}= \left|\int_{0}^{\tau}\sqrt{\psi} W^{\frac{\alpha+m}{2}}\Ndell{m}{n}\nu\ d\tau'\right|^{2}.
	\end{align*}
	Now, integrating over $\Omega$ and using the generalised Minkowski's integral inequality gives us
	\begin{align*}
	\spaceI \psi W^{\alpha+m}\left|\Ndell{m}{n}\theta\right|^{2}dx&\lesssim\left(\int_{0}^{\tau}\left(\spaceI\psi W^{\alpha+m}\left|\Ndell{m}{n}\nu\right|^{2}dx\right)^{1/2}d\tau'\right)^{2}\nonumber\\
	&\lesssim\left(\int_{0}^{\tau}\sqrt{\delta}e^{-\sigma_{1}\tau'/2}\mathscr{S}_{N}(\tau')^{1/2}d\tau'\right)^{2}\nonumber\\
	&\lesssim\delta \mathscr{S}_{N}(\tau),
	\end{align*}
	where we have used $\sigma_1>0$.
\end{proof}

We also have a result that bounds $\left\|\theta\right\|_{\Y{M}{\ngrad}}$ in terms of $\left\|\theta\right\|_{\X{N}}$, both defined in $(\ref{energy-space-norm-2})$, for indices $M<N$.


\begin{lemma}\label{theta-Y-norm-theta-X-norm-relation-lemma}Let $(\nu,\theta)$ be the solution to~\eqref{E:EULERNEW2} on $[0,T]$ in the sense of Theorem~\ref{local-well-posedness-theorem}, for some $N\geq2\ceil*{\alpha}+12$. Suppose the a priori assumptions $(\ref{a-priori-assumptions})$ hold. Finally, let $M\in\mathbb{Z}_{\geq0}$ be such that $M<N$. Then for all $\tau\in[0,T]$, we have
	\begin{align*}
	\left\|\theta\right\|_{\Y{M}{\ngrad}}\lesssim\left\|\theta\right\|_{\X{N}}.
	\end{align*}
\end{lemma}


\begin{proof}
	Recall that
	\begin{align*}
	\left\|\theta\right\|_{\Y{M}{\ngrad}}^{2}=\sum_{m+|\underline{n}|=0}^{M}\spaceI\psi W^{1+\alpha+m}\pjac\left|\ngrad{\Ndell{m}{n}\theta}\right|^{2}dx+\sum_{|\underline{k}|=0}^{M}\spaceI\bar{\psi} W^{1+\alpha+m}\pjac\left|\ngrad{\nabla^{\underline{k}}\theta}\right|^{2}dx.
	\end{align*}
	By the definition of $\ngrad$ in $(\ref{zeta-gradient})$, we have the bound
	\begin{align*}
	&\sum_{m+|\underline{n}|=0}^{M}\spaceI\psi W^{1+\alpha+m}\pjac\left|\ngrad{\Ndell{m}{n}\theta}\right|^{2}dx+\sum_{|\underline{k}|=0}^{M}\spaceI\bar{\psi} W^{1+\alpha+m}\pjac\left|\ngrad{\nabla^{\underline{k}}\theta}\right|^{2}dx\\
	&\lesssim\sum_{m+|\underline{n}|=0}^{M}\spaceI\psi W^{1+\alpha+m}\pjac\left|\ninv\right|^{2}\left|\nabla\Ndell{m}{n}\theta\right|^{2}dx+\sum_{|\underline{k}|=0}^{M}\spaceI\bar{\psi} W^{1+\alpha+m}\pjac\left|\ninv\right|^{2}\left|\nabla\nabla^{\underline{k}}\theta\right|^{2}dx.
	\end{align*}
	For the second term on the second line above,
	\begin{align}
	\sum_{|\underline{k}|=0}^{M}\spaceI\bar{\psi} W^{1+\alpha+m}\pjac\left|\ninv\right|^{2}\left|\nabla\nabla^{\underline{k}}\theta\right|^{2}dx\lesssim\sum_{|\underline{k}|=0}^{M+1}\spaceI\bar{\psi} W^{\alpha+m}\left|\nabla^{\underline{k}}\theta\right|^{2}dx.\label{theta-Y-norm-theta-X-norm-relation-interior-estimate}
	\end{align}
	We bound the $\ninv$ and $\pjac$ terms, as well as $W$, in $L^{\infty}$ using the a priori assumptions in $(\ref{a-priori-assumptions})$, and the fact that $W\sim1$ on $\supp{\bar{\psi}}$. Using $(\ref{rectangular-as-ang-rad})$, we can write rectangular derivatives as a linear combination of radial and angular derivatives with coefficients smooth away from the origin. So, for the integrals on $\supp{\psi}$, we have
	\begin{align}
	&\sum_{m+|\underline{n}|=0}^{M}\spaceI\psi W^{1+\alpha+m}\pjac\left|\ninv\right|^{2}\left|\grad{\Ndell{m}{n}\theta}\right|^{2}dx\nonumber\\
	&\lesssim\sum_{m+|\underline{n}|=0}^{M}\spaceI\psi W^{1+\alpha+m}\pjac\left|\ninv\right|^{2}\left|\rdell\Ndell{m}{n}\theta\right|^{2}dx+\sum_{m+|\underline{n}|=0}^{M}\spaceI\psi W^{1+\alpha+m}\pjac\left|\ninv\right|^{2}\left|\angdell\Ndell{m}{n}\theta\right|^{2}dx\nonumber\\
	&\lesssim\sum_{m+|\underline{n}|=0}^{M+1}\spaceI\psi W^{\alpha+m}\left|\Ndell{m}{n}\theta\right|^{2}dx.\label{theta-Y-norm-theta-X-norm-relation-boundary-estimate}
	\end{align}
	We once again use $(\ref{a-priori-assumptions})$ to bound $\ninv$ and $\pjac$ in $L^{\infty}$. We also use that on $\supp{\psi}$, $W(x)\sim d(x,\dell\Omega)$ to bound $W$ in $L^{\infty}$.
	 Combining $(\ref{theta-Y-norm-theta-X-norm-relation-interior-estimate})$ and $(\ref{theta-Y-norm-theta-X-norm-relation-boundary-estimate})$, we have
	\begin{align*}
	\left\|\theta\right\|_{\Y{M}{\ngrad}}^{2}\lesssim\left\|\theta\right\|_{\X{M+1}}^{2}\lesssim\left\|\theta\right\|_{\X{N}}^{2},
	\end{align*}
	as $M+1\leq N$. This gives the result.
\end{proof}


Finally, we prove a lemma that shows how to estimate terms of the form $\Ndell{a}{b}\left(W\pjac\ninv\right)$.


\begin{lemma}\label{lower-order-estimates-lemma}Let $(\nu,\theta)$ be the solution to~\eqref{E:EULERNEW2} on $[0,T]$ in the sense of Theorem~\ref{local-well-posedness-theorem}, for some $N\geq2\ceil*{\alpha}+12$. Assume $\nabla W\in\X{N}$. Suppose the a priori assumptions $(\ref{a-priori-assumptions})$ hold. Finally, let $(m,\underline{n},\underline{k})\in\mathbb{Z}_{\geq0}\times\mathbb{Z}_{\geq0}^{3}\times\mathbb{Z}_{\geq0}^{3}$ be such that $1\leq\max{(m+|\underline{n}|,|\underline{k}|)}\leq N$. Then for all $\tau\in[0,T]$, we have
	\begin{align}
	\spaceI\psi W^{\alpha+m}\left|\Ndell{m}{n}\left(W\pjac\ninv\right)\right|^{2}dx+\spaceI\bar{\psi}W^{\alpha}\left|\nabla^{\underline{k}}\left(W\pjac\ninv\right)\right|^{2}dx\lesssim1+\delta\mathscr{S}_{N}(\tau)+ e^{\sigma_{2}\tau}\mathscr{S}_{N}(\tau),\label{lower-order-estimates-statement-1}
	\end{align}
	\begin{align}
	\spaceI\psi W^{1+\alpha+m}\left|\Ndell{m}{n}\left(\pjac\ninv\right)\right|^{2}dx+\spaceI\bar{\psi}W^{1+\alpha}\left|\nabla^{\underline{k}}\left(\pjac\ninv\right)\right|^{2}dx\lesssim\delta\mathscr{S}_{N}(\tau)+ e^{\sigma_{2}\tau}\mathscr{S}_{N}(\tau).\label{lower-order-estimates-statement-2}
	\end{align}
\end{lemma}


\begin{proof}Consider $\Ndell{m}{n}\left(W\pjac\ninv^{k}_{i}\right)$, some $i$ and $k$ in $\{1,2,3\}$. Using $(\ref{inverse-differentiation-zeta})$ and $(\ref{jacobian-differentiation-zeta})$ to differentiate $\ninv^{k}_{i}$ and $\pjac$ repeatedly, this term can be written as
\begin{align}
&-W\pjac\ninv^{k}_{r}\ninv^{s}_{i}\Ndell{m}{n}\left(\dell_{s}\theta^{r}\right)-\frac{1}{\alpha}W\pjac\ninv^{k}_{i}\ninv^{s}_{r}\Ndell{m}{n}\left(\dell_{s}\theta^{r}\right)+\pjac\ninv^{k}_{i}\Ndell{m}{n}W\nonumber\\
&+\sum_{p=1}^{m+|\underline{n}|}\sum_{\substack{a_{1}+\dots+a_{p}+c=m\\ \underline{b_{1}}+\dots+\underline{b_{p}}+\underline{d}=\underline{n}\\1\leq c+|\underline{d}|< m+|\underline{n}|}}\mathscr{L}(a_{1},\underline{b}_{1}\dots,a_{p},\underline{b}_{p},c,d)\pjac\underbrace{\ninv\dots\ninv}_{p+1}\Ndell{c}{d}W\rdell^{a_{1}}\angdell^{\underline{b}_{1}}\left(\nabla\theta\right)\dots\rdell^{a_{p}}\angdell^{\underline{b}_{p}}\left(\nabla\theta\right).\label{lower-order-term-expansion}
\end{align}
The terms 
\begin{align*}
\ninv\dots\ninv\Ndell{c}{d}W\rdell^{a_{1}}\angdell^{\underline{b}_{1}}\left(\nabla\theta\right)\dots\rdell^{a_{p}}\angdell^{\underline{b}_{p}}\left(\nabla\theta\right)
\end{align*}
are written schematically in the sense of (\ref{schematically}) and (\ref{schematically-2}). As we are looking to bound this sum using the higher order energy function, the specific structure of the indices for each factor above is not important, only the derivative distribution across the product. The terms $\mathscr{L}(a_{1},\underline{b}_{1}\dots,a_{p},\underline{b}_{p},c,d)$ are the constant coefficients of expansion, counting the multiplicity of each term that appears in the sum. Note that if $m+|\underline{n}|=1$ then all of these coefficients are $0$.

To estimate the first term in $(\ref{lower-order-term-expansion})$, we use Lemma \ref{higher-order-commutator-identities} to write
\begin{align*}
-W\pjac\ninv^{k}_{r}\ninv^{s}_{i}\Ndell{m}{n}\left(\dell_{s}\theta^{r}\right)=-W\pjac\ninv^{k}_{r}\left[\ngrad{\Ndell{m}{n}\theta}\right]^{r}_{i}-\sum_{j=1}^{m+|\underline{n}|}\sum_{\substack{e+|\underline{f}|=j\\e\leq m+1}}W\mathcal{K}_{s,j,e,\underline{f}}\pjac\ninv^{k}_{r}\ninv^{s}_{i}\Ndell{e}{f}\theta^{r},
\end{align*}
where $\mathcal{K}_{s,j,e,\underline{f}}$ are some functions, smooth on $\supp{\psi}$. Thus we have the estimate
\begin{align}
\spaceI\psi W^{\alpha+m}\left|W\pjac\ninv^{k}_{r}\ninv^{s}_{i}\Ndell{m}{n}\left(\dell_{s}\theta^{r}\right)\right|^{2}dx&\lesssim\spaceI\psi W^{1+\alpha+m}\njac^{-1/\alpha}\left|\ngrad{\Ndell{m}{n}\theta}\right|^{2}dx\nonumber\\
&+\sum_{j=1}^{m+|\underline{n}|}\sum_{\substack{e+|\underline{f}|=j\\e\leq m+1}}\spaceI\psi W^{1+\alpha+m}\left|\Ndell{e}{f}\theta\right|^{2}dx,\label{lower-order-estimate-1}
\end{align}
where we have used $(\ref{a-priori-assumptions})$ to bound $\pjac$ and $\ninv$ terms in $L^{\infty}$. We have also bounded $\mathcal{K}_{s,j,e,\underline{f}}$ in $L^{\infty}$. Then we have
\begin{align}
\spaceI\psi W^{\alpha+m}\left|W\pjac\ninv^{k}_{r}\ninv^{s}_{i}\Ndell{m}{n}\left(\dell_{s}\theta^{r}\right)\right|^{2}dx\lesssim\left\|\theta\right\|_{\Y{N}{\ngrad}}^{2}+\left\|\theta\right\|_{\X{N}}^{2}\lesssim e^{\sigma_{2}\tau}\mathscr{S}_{N}(\tau)+\delta\mathscr{S}_{N}(\tau).\label{lower-order-estimate-2}
\end{align}
Analogously for the second term in $(\ref{lower-order-term-expansion})$,
\begin{align}
\spaceI\psi W^{\alpha+m}\left|W\pjac\ninv^{k}_{i}\ninv^{s}_{r}\Ndell{m}{n}\left(\dell_{s}\theta^{r}\right)\right|^{2}dx\lesssim\left\|\theta\right\|_{\Y{N}{\ndiv}}^{2}+\left\|\theta\right\|_{\X{N}}^{2}\lesssim e^{\sigma_{2}\tau}\mathscr{S}_{N}(\tau)+\delta\mathscr{S}_{N}(\tau).\label{lower-order-estimate-3}
\end{align}
For the third term in $(\ref{lower-order-term-expansion})$, if $m>0$, we write $\Ndell{m}{n}W=\Ndell{m-1}{n}\left(x\cdot\nabla W\right)$, as $\rdell=x\cdot\nabla$. This implies
\begin{align}
\Ndell{m}{n}W=\sum^{m+|\underline{n}|-1}_{\substack{e+|\underline{f}|=0\\e\leq m-1}}\mathcal{Q}_{e,\underline{f}}\cdot\Ndell{e}{f}\left(\nabla W\right).\label{Ndell-W-as-Ndell-nabla-W}
\end{align}
The coefficients $\mathcal{Q}_{e,\underline{f}}$ come from differentiating $x$, and are smooth on $\supp{\psi}$. Therefore, we have the bound
\begin{align}
\spaceI\psi W^{\alpha+m}\left|\pjac\ninv^{k}_{i}\Ndell{m}{n}W\right|^{2}dx\lesssim\left\|\nabla W\right\|_{\X{N}}^{2}\lesssim1.\label{lower-order-estimate-4}
\end{align}
We have bounded $\pjac$ and $\ninv^{k}_{i}$ in $L^{\infty}$ using $(\ref{a-priori-assumptions})$, and the last inequality is due to the assumption that $\nabla W\in\X{N}$. If $m=0$, then we can use $\angdell_{ij}=x_{i}\dell_{j}-x_{j}\dell_{i}$ to instead obtain
\begin{align}
\angdell^{\underline{n}}W=\sum^{|\underline{n}|-1}_{|\underline{f}|=0}\mathcal{Q}_{\underline{f}}\cdot\angdell^{\underline{f}}\left(\nabla W\right),\label{Ndell-W-as-Ndell-nabla-W-2}
\end{align}
and get the same bound as $(\ref{lower-order-estimate-4})$. Finally we have the sum of lower order terms on the second line of $(\ref{lower-order-term-expansion})$. For each term in the sum, we have the bound
\begin{align}
&\spaceI\psi W^{\alpha+m}\left|\mathscr{L}(a_{1},\underline{b}_{1},\dots,a_{p},\underline{b}_{p},c,d)\right|^{2}\njac^{-2/\alpha}\left|\ninv\right|^{2(p+1)}\left|\Ndell{c}{d}W\right|^{2}\left|\rdell^{a_{1}}\angdell^{\underline{b}_{1}}\left(\nabla\theta\right)\right|^{2}\dots\left|\rdell^{a_{p}}\angdell^{\underline{b}_{p}}\left(\nabla\theta\right)\right|^{2}dx\nonumber\\
&\lesssim\spaceI\psi W^{\alpha+m}\njac^{-2/\alpha}\left|\Ndell{c}{d}W\right|^{2}\left|\rdell^{a_{1}}\angdell^{\underline{b}_{1}}\left(\nabla\theta\right)\right|^{2}\dots\left|\rdell^{a_{p}}\angdell^{\underline{b}_{p}}\left(\nabla\theta\right)\right|^{2}dx,\label{lower-order-estimate-5}
\end{align}
with the bound coming from the fact that $\mathscr{L}$ are constant coefficients and $\ninv$ is bounded in $L^{\infty}$. The estimate for the right hand side of $(\ref{lower-order-estimate-5})$ is comprised of two cases.\\

\noindent \textbf{Case I: $c>0$.}
 Since $c>0$, we write
\begin{align}
\Ndell{c}{d}W=\sum_{\substack{l+|\underline{q}|=c+|\underline{d}|-1\\l\leq c-1}}\mathcal{Q}_{l,\underline{q}}\cdot\Ndell{l}{q}\left(\nabla W\right)\label{Ndell-W-as-Ndell-nabla-W-3}
\end{align}
similarly to $(\ref{Ndell-W-as-Ndell-nabla-W})$. Once again, we can use Lemma \ref{higher-order-commutator-identities} and $(\ref{rectangular-as-ang-rad})$ to write
\begin{align}
\rdell^{a_{i}}\angdell^{\underline{b}_{i}}\left(\nabla\theta\right)=\sum_{j_{i}=1}^{a_{i}+|\underline{b}_{i}|}\sum_{\substack{e_{i}+|\underline{f}_{i}|=j_{i}\\e_{i}\leq a_{i}}}\mathcal{K}_{j_{i},e_{i},\underline{f}_{i}}\left[\nabla\rdell^{e_{i}}\angdell^{\underline{f}_{i}}\theta\right],\label{higher-order-commutator-identity-applied-to-lower-order-Ndell-theta}
\end{align}
with $\mathcal{K}$ smooth on $\supp{\psi}$. Thus we have the bound
\begin{align}
&\spaceI\psi W^{\alpha+m}\njac^{-2/\alpha}\left|\Ndell{c}{d}W\right|^{2}\left|\rdell^{a_{1}}\angdell^{\underline{b}_{1}}\left(\nabla\theta\right)\right|^{2}\dots\left|\rdell^{a_{p}}\angdell^{\underline{b}_{p}}\left(\nabla\theta\right)\right|^{2}dx\nonumber\\
&\lesssim\sum_{i=1}^{p}\sum_{j_{i}=1}^{a_{i}+|\underline{b}_{i}|}\sum_{\substack{e_{i}+|\underline{f}_{i}|=j_{i}\\e_{i}\leq a_{i}}}\sum_{\substack{l+|\underline{q}|=c+|\underline{d}|-1\\l\leq c-1}}\notag \\
& \qquad \spaceI\psi W^{\alpha+m}\left|\mathcal{K}_{j_{1},e_{1},\underline{f}_{1}}\right|^{2}\dots\left|\mathcal{K}_{j_{p},e_{p},\underline{f}_{p}}\right|^{2}\left|\mathcal{Q}_{l,\underline{q}}\right|^{2}\left|\Ndell{l}{q}\left(\nabla W\right)\right|^{2}\left|\nabla\rdell^{e_{1}}\angdell^{\underline{f}_{1}}\theta\right|^{2}\dots\left|\nabla\rdell^{e_{p}}\angdell^{\underline{f}_{p}}\theta\right|^{2}dx.\label{lower-order-estimate-6}
\end{align}
Choose a term from the right hand side above and, reordering if necessary, assume $e_{1}+|\underline{f}_{1}|\leq\dots\leq e_{p}+|\underline{f}_{p}|$. Suppose for now that $p\geq2$. Then we write
\begin{align}
W^{\alpha+m}\left|\Ndell{l}{q}\left(\nabla W\right)\right|^{2}\left|\nabla\rdell^{e_{1}}\angdell^{\underline{f}_{1}}\theta\right|^{2}\dots\left|\nabla\rdell^{e_{p}}\angdell^{\underline{f}_{p}}\theta\right|^{2}=W^{e_{1}}\left|\nabla\rdell^{e_{1}}\angdell^{\underline{f}_{1}}\theta\right|^{2}\dots W^{g+\alpha+l+e_{p}}\left|\Ndell{l}{q}\left(\nabla W\right)\right|^{2}\left|\nabla\rdell^{e_{p}}\angdell^{\underline{f}_{p}}\theta\right|^{2},\label{lower-order-estimate-distribution-of-W}
\end{align}
where $g=m-l-\sum_{i}e_{i}$. Note that $g\geq1$ because $l+e_{1}+\dots+e_{p}\leq c-1+a_{1}+\dots+a_{p}=m-1$. If $l+|\underline{q}|\geq e_{p}+|\underline{f}_{p}|$, then we must have $e_{p}+|\underline{f}_{p}|\leq (m+|\underline{n}|)/2$, and $e_{i}+|\underline{f}_{i}|<(m+|\underline{n}|)/2$ for $i=1,\dots,p-1$. Hence we have $\ceil*{\alpha}+6+e_{p}+|\underline{f}_{p}|\leq N$, and $\ceil*{\alpha}+6+e_{p}+|\underline{f}_{p}|< N$ for $i=1,\dots,p-1$.

So for $i=1,\dots,p$, we bound $W^{e_{i}}\left|\nabla\rdell^{e_{i}}\angdell^{\underline{f}_{i}}\right|^{2}$ in $L^{\infty}$, then use $(\ref{L-infinity-energy-space-bound-weights-statement-2})$ in Lemma \ref{L-infinity-energy-space-bound-weights}, as well as Lemma \ref{theta-Y-norm-theta-X-norm-relation-lemma} for $i=1,\dots,p-1$. This gives us
\begin{align}
&\spaceI\psi W^{\alpha+m}\left|\mathcal{K}_{j_{1},e_{1},\underline{f}_{1}}\right|^{2}\dots\left|\mathcal{K}_{j_{p},e_{p},\underline{f}_{p}}\right|^{2}\left|\mathcal{Q}_{l,\underline{q}}\right|^{2}\left|\Ndell{l}{q}\left(\nabla W\right)\right|^{2}\left|\nabla\rdell^{e_{1}}\angdell^{\underline{f}_{1}}\theta\right|^{2}\dots\left|\nabla\rdell^{e_{p}}\angdell^{\underline{f}_{p}}\theta\right|^{2}dx\nonumber\\
&\lesssim\left\|\theta\right\|_{\X{N}}^{2(p-1)}\left\|\theta\right\|_{\Y{N}{\ngrad}}^{2}\spaceI\psi W^{g+\alpha+l}\left|\Ndell{l}{q}\left(\nabla W\right)\right|^{2}dx\lesssim\delta e^{\sigma_{2}\tau}\mathscr{S}_{N}(\tau).\label{lower-order-estimate-7}
\end{align}
The first inequality above comes from the $L^{\infty}$ embeddings mentioned above, as well as bounding $\mathcal{K}_{j_{i},e_{i},\underline{f}_{i}}$ and $\mathcal{Q}_{l,\underline{q}}$ in $L^{\infty}$. The second inequality is because $W^{g}$ is bounded in $L^{\infty}$, and $\nabla W\in\X{N}$. Finally, we have bounded any extra powers of $\delta$ and $\mathscr{S}_{N}(\tau)$ by a constant. This is as $\delta$ is itself a constant, and $\mathscr{S}_{N}(\tau)$ can be bounded using the a priori assumptions $(\ref{a-priori-assumptions})$. 

If $e_{p}+|\underline{f}_{p}|\geq l+|\underline{q}|$, then we instead write
\begin{align}
W^{\alpha+m}\left|\Ndell{l}{q}\left(\nabla W\right)\right|^{2}\left|\nabla\rdell^{e_{1}}\angdell^{\underline{f}_{1}}\theta\right|^{2}\dots\left|\nabla\rdell^{e_{p}}\angdell^{\underline{f}_{p}}\theta\right|^{2}=W^{l}\left|\Ndell{l}{q}\left(\nabla W\right)\right|^{2}W^{e_{1}}\left|\nabla\rdell^{e_{1}}\angdell^{\underline{f}_{1}}\theta\right|^{2}\dots W^{g+\alpha+e_{p}}\left|\nabla\rdell^{e_{p}}\angdell^{\underline{f}_{p}}\theta\right|^{2}.\label{lower-order-estimate-distribution-of-W-2}
\end{align}
If $p=2$, and $e_{1}+|\underline{f}_{1}|=e_{2}+|\underline{f}_{2}|$, then we have
\begin{align}
&\spaceI\psi \left|\mathcal{K}_{j_{1},e_{1},\underline{f}_{1}}\right|^{2}\left|\mathcal{K}_{j_{2},e_{2},\underline{f}_{2}}\right|^{2}\left|\mathcal{Q}_{l,\underline{q}}\right|^{2}W^{l}\left|\Ndell{l}{q}\left(\nabla W\right)\right|^{2}W^{e_{1}}\left|\nabla\rdell^{e_{1}}\angdell^{\underline{f}_{1}}\theta\right|^{2}W^{g+\alpha+e_{2}}\left|\nabla\rdell^{e_{2}}\angdell^{\underline{f}_{2}}\theta\right|^{2}dx\nonumber\\
&\lesssim\left\|\nabla W\right\|_{\X{N}}^{2}\left\|\theta\right\|_{\Y{N}{\ngrad}}^{2}\spaceI\psi W^{1+\alpha+e_{2}}\left|\nabla\rdell^{e_{2}}\angdell^{\underline{f}_{2}}\theta\right|^{2}dx\lesssim\left\|\nabla W\right\|_{\X{N}}^{2}\left\|\theta\right\|_{\Y{N}{\ngrad}}^{2}\left\|\theta\right\|_{\Y{\ceil*{\alpha}+6+e_{2}+|\underline{f}_{2}|}{\ngrad}}^{2}.\label{lower-order-estimate-8}
\end{align}
The first inequality is once again due to $L^{\infty}$ bounds followed by Lemma \ref{L-infinity-energy-space-bound-weights}, and the second inequality follows by the definition of $\Y{b}{\ngrad}$ given in $(\ref{energy-space-norm-2})$. Now, since $\ceil*{\alpha}+6+e_{2}+|\underline{f}_{2}|<N$, we can apply Lemma \ref{theta-Y-norm-theta-X-norm-relation-lemma} to obtain
\begin{align}
&\spaceI\psi \left|\mathcal{K}_{j_{1},e_{1},\underline{f}_{1}}\right|^{2}\left|\mathcal{K}_{j_{2},e_{2},\underline{f}_{2}}\right|^{2}\left|\mathcal{Q}_{l,\underline{q}}\right|^{2}W^{l}\left|\Ndell{l}{q}\left(\nabla W\right)\right|^{2}W^{e_{1}}\left|\nabla\rdell^{e_{1}}\angdell^{\underline{f}_{1}}\theta\right|^{2}W^{r+\alpha+e_{2}}\left|\nabla\rdell^{e_{p}}\angdell^{\underline{f}_{p}}\theta\right|^{2}dx\nonumber\\
&\lesssim\left\|\nabla W\right\|_{\X{N}}^{2}\left\|\theta\right\|_{\Y{N}{\ngrad}}^{2}\left\|\theta\right\|_{\X{N}}^{2}\lesssim\delta e^{\sigma_{2}\tau}\mathscr{S}_{N}(\tau).\label{lower-order-estimate-9}
\end{align}
Otherwise, $e_{i}+|\underline{f}_{i}|<N/2$ for $i=1,\dots,p-1$, so we bound analogously to $(\ref{lower-order-estimate-7})$:
\begin{align}
&\spaceI\psi W^{\alpha+m}\left|\mathcal{K}_{j_{1},e_{1},\underline{f}_{1}}\right|^{2}\dots\left|\mathcal{K}_{j_{p},e_{p},\underline{f}_{p}}\right|^{2}\left|\mathcal{Q}_{l,\underline{q}}\right|^{2}\left|\Ndell{l}{q}\left(\nabla W\right)\right|^{2}\left|\nabla\rdell^{e_{1}}\angdell^{\underline{f}_{1}}\theta\right|^{2}\dots\left|\nabla\rdell^{e_{p}}\angdell^{\underline{f}_{p}}\theta\right|^{2}dx\nonumber\\
&\lesssim\left\|\nabla W\right\|_{\X{N}}^{2}\left\|\theta\right\|_{\X{N}}^{2(p-1)}\spaceI\psi W^{1+\alpha+e_{p}}\left|\nabla\rdell^{e_{p}}\angdell^{\underline{f}_{p}}\theta\right|^{2}dx\lesssim\delta e^{\sigma_{2}\tau}\mathscr{S}_{N}(\tau).\label{lower-order-estimate-10}
\end{align}
Combining bounds $(\ref{lower-order-estimate-5})$, $(\ref{lower-order-estimate-6})$, $(\ref{lower-order-estimate-7})$, $(\ref{lower-order-estimate-8})$, $(\ref{lower-order-estimate-9})$, and $(\ref{lower-order-estimate-10})$, we have
\begin{align}
\spaceI\psi W^{\alpha+m}\left|\mathscr{L}(a_{1},\underline{b}_{1},\dots,a_{p},\underline{b}_{p},c,d)\right|^{2}\njac^{-2/\alpha}\left|\ninv\right|^{2(p+1)}\left|\Ndell{c}{d}W\right|^{2}\left|\rdell^{a_{1}}\angdell^{\underline{b}_{1}}\left(\nabla\theta\right)\right|^{2}\dots\left|\rdell^{a_{p}}\angdell^{\underline{b}_{p}}\left(\nabla\theta\right)\right|^{2}dx\lesssim\delta e^{\sigma_{2}\tau}\mathscr{S}_{N}(\tau)\label{lower-order-estimate-11}
\end{align}
when $c>0$ and $p\geq2$.

If $p=1$ then we look to bound
\begin{align}
\spaceI\psi W^{\alpha+m}\left|\rdell^{l}\angdell^{\underline{q}}\left(\nabla W\right)\right|^{2}\left|\nabla\rdell^{e_{1}}\angdell^{\underline{f}_{1}}\theta\right|^{2}dx.\label{lower-order-estimate-11a}
\end{align}
If $l+|\underline{q}|\geq e_{1}+|\underline{f}_{1}|$, we bound $W^{e_{1}}\left|\nabla\rdell^{e_{1}}\angdell^{\underline{f}_{1}}\theta\right|^{2}$ in $L^{\infty}$, then use $(\ref{L-infinity-energy-space-bound-weights-statement-2})$ in Lemma \ref{L-infinity-energy-space-bound-weights}. If $l+|\underline{q}|\leq e_{1}+|\underline{f}_{1}|$, we instead bound $W^{l}\left|\rdell^{l}\angdell^{\underline{q}}\left(\nabla W\right)\right|^{2}$ in $L^{\infty}$, and use (\ref{L-infinity-energy-space-bound-weights-statement-1}) in Lemma \ref{L-infinity-energy-space-bound-weights}. In both cases we have the bound
\begin{align}
\spaceI\psi W^{\alpha+m}\left|\rdell^{l}\angdell^{\underline{q}}\left(\nabla W\right)\right|^{2}\left|\nabla\rdell^{e_{1}}\angdell^{\underline{f}_{1}}\theta\right|^{2}dx\lesssim e^{\sigma_{2}\tau}\mathscr{S}_{N}(\tau).\label{lower-order-estimate-11b}
\end{align}

\noindent \textbf{Case II: $c=0$.}
We have the bound 
\begin{align}
&\spaceI\psi W^{\alpha+m}\njac^{-2/\alpha}\left|\Ndell{c}{d}W\right|^{2}\left|\rdell^{a_{1}}\angdell^{\underline{b}_{1}}\left(\nabla\theta\right)\right|^{2}\dots\left|\rdell^{a_{p}}\angdell^{\underline{b}_{p}}\left(\nabla\theta\right)\right|^{2}dx\nonumber\\
&\lesssim\sum_{i=1}^{p}\sum_{j_{i}=1}^{a_{i}+|\underline{b}_{i}|}\sum_{\substack{e_{i}+|\underline{f}_{i}|=j_{i}\\e_{i}\leq a_{i}}}\spaceI\psi W^{\alpha+m}\left|\mathcal{K}_{j_{1},e_{1},\underline{f}_{1}}\right|^{2}\dots\left|\mathcal{K}_{j_{p},e_{p},\underline{f}_{p}}\right|^{2}\left|\angdell^{\underline{d}} W\right|^{2}\left|\nabla\rdell^{e_{1}}\angdell^{\underline{f}_{1}}\theta\right|^{2}\dots\left|\nabla\rdell^{e_{p}}\angdell^{\underline{f}_{p}}\theta\right|^{2}dx.\label{lower-order-estimate-12}
\end{align}
Once again we choose a term from the right hand side and assume $e_{1}+|\underline{f}_{1}|\leq\dots\leq e_{p}+|\underline{f}_{p}|$, with $p\geq2$.\\

\noindent \textbf{Case IIa: $|\underline{d}|\leq e_{p}+|\underline{f}_{p}|$.}
If $|\underline{d}|\leq e_{p}+|\underline{f}_{p}|$, then we have $\angdell^{\underline{d}}W=d_{\Omega}\angdell^{\underline{d}}(W/d_{\Omega})$, where we recall from $(\ref{distance-function-shorthand})$ that $d_{\Omega}(x)=d(x,\dell\Omega)=1-r(x)$ is a function depending on the radial direction only.

From Definition \ref{definition-of-W}, we know that $W/d_{\Omega}$ is smooth on $\supp{\psi}$, so
\begin{align}
\left\|\angdell^{\underline{d}}\left(\frac{W}{d_{\Omega}}\right)\right\|_{L^{\infty}(\supp{\psi})}\lesssim1.\label{W-over-d-estimate-1}
\end{align}
Therefore we have the bound
\begin{align}
&\spaceI\psi W^{\alpha+m}\left|\mathcal{K}_{j_{1},e_{1},\underline{f}_{1}}\right|^{2}\dots\left|\mathcal{K}_{j_{p},e_{p},\underline{f}_{p}}\right|^{2}\left|\angdell^{\underline{d}} W\right|^{2}\left|\nabla\rdell^{e_{1}}\angdell^{\underline{f}_{1}}\theta\right|^{2}\dots\left|\nabla\rdell^{e_{p}}\angdell^{\underline{f}_{p}}\theta\right|^{2}dx\nonumber\\
&\lesssim\spaceI\psi W^{2+\alpha+m}\left|\nabla\rdell^{e_{1}}\angdell^{\underline{f}_{1}}\theta\right|^{2}\dots\left|\nabla\rdell^{e_{p}}\angdell^{\underline{f}_{p}}\theta\right|^{2}dx.\label{lower-order-estimate-13}
\end{align}
Here we have used $L^{\infty}$ on the $\mathcal{K}_{j_{i},e_{i},\underline{f}_{i}}$, as well as $L^{\infty}$ and (\ref{W-over-d-estimate-1}) on $\angdell^{\underline{d}}(W/d_{\Omega})$. Moreover, from Definition \ref{definition-of-W}, we know $\left|d(x,\dell\Omega)\right|^{2}W^{\alpha+m}\sim W^{2+\alpha+m}$.

The integral on the right hand side of $(\ref{lower-order-estimate-13})$ can be estimated as in the $c>0$ case. This gives us
\begin{align}
\spaceI\psi W^{2+\alpha+m}\left|\nabla\rdell^{e_{1}}\angdell^{\underline{f}_{1}}\theta\right|^{2}\dots\left|\nabla\rdell^{e_{p}}\angdell^{\underline{f}_{p}}\theta\right|^{2}dx\lesssim\delta e^{\sigma_{2}\tau}\mathscr{S}_{N}(\tau).\label{lower-order-estimate-14}
\end{align}

\noindent \textbf{Case IIb: $|\underline{d}|\geq e_{p}+|\underline{f}_{p}|$.} If $|\underline{d}|\geq e_{p}+|\underline{f}_{p}|$, then write
\begin{align*}
\angdell^{\underline{d}}W=\sum^{|\underline{d}|-1}_{|\underline{l}|=0}\mathcal{Q}_{\underline{l}}\cdot\angdell^{\underline{l}}\left(\nabla W\right)
\end{align*}
as in $(\ref{Ndell-W-as-Ndell-nabla-W-2})$, with $\mathcal{Q}_{\underline{l}}$ smooth on $\supp{\psi}$, and we have the bound
\begin{align}
&\spaceI\psi W^{\alpha+m}\left|\mathcal{K}_{j_{1},e_{1},\underline{f}_{1}}\right|^{2}\dots\left|\mathcal{K}_{j_{p},e_{p},\underline{f}_{p}}\right|^{2}\left|\angdell^{\underline{d}} W\right|^{2}\left|\nabla\rdell^{e_{1}}\angdell^{\underline{f}_{1}}\theta\right|^{2}\dots\left|\nabla\rdell^{e_{p}}\angdell^{\underline{f}_{p}}\theta\right|^{2}dx\nonumber\\
&\lesssim\sum^{|\underline{d}|-1}_{|\underline{l}|=0}\spaceI\psi W^{e_{1}}\left|\nabla\rdell^{e_{1}}\angdell^{\underline{f}_{1}}\theta\right|^{2}\dots W^{e_{p}}\left|\nabla\rdell^{e_{p}}\angdell^{\underline{f}_{p}}\theta\right|^{2}W^{\alpha}\left|\angdell^{\underline{l}}\left(\nabla W\right)\right|^{2}dx\lesssim\delta e^{\sigma_{2}\tau}\mathscr{S}_{N}(\tau).\label{lower-order-estimate-15}
\end{align}
As $e_{i}+|\underline{f}_{i}|\leq (m+|\underline{n}|)/2$, $i=1,\dots,p$, we can use $L^{\infty}$ bounds, then Lemmas \ref{L-infinity-energy-space-bound-weights} and \ref{theta-Y-norm-theta-X-norm-relation-lemma} on the $W^{e_{i}}\left|\rdell^{e_{i}}\angdell^{\underline{f}_{i}}\theta\right|^{2}$ terms. This, along with the fact that $\nabla W\in\X{N}$, gives the second inequality in $(\ref{lower-order-estimate-15})$.

For the case $p=1$, the bound follows similarly to (\ref{lower-order-estimate-11b}), and we have
\begin{align}
\spaceI\psi W^{\alpha+m}\left|\angdell^{\underline{d}}W\right|^{2}\left|\nabla\rdell^{e_{1}}\angdell^{\underline{f}_{1}}\theta\right|^{2}dx\lesssim e^{\sigma_{2}\tau}\mathscr{S}_{N}(\tau).\label{lower-order-estimate-16}
\end{align}
Combining bounds $(\ref{lower-order-estimate-2})$, $(\ref{lower-order-estimate-3})$, $(\ref{lower-order-estimate-4})$, $(\ref{lower-order-estimate-11})$,
$(\ref{lower-order-estimate-11b})$, $(\ref{lower-order-estimate-14})$, $(\ref{lower-order-estimate-15})$, and $(\ref{lower-order-estimate-16})$, along with analogous estimates on $\supp{\bar{\psi}}$, gives $(\ref{lower-order-estimates-statement-1})$. The inequality in $(\ref{lower-order-estimates-statement-2})$ follows similarly.
\end{proof}

\section{Curl Estimates}\label{curl-estimates}

As we will see in Section $\ref{energy-estimates}$ when we derive the natural energy function, the top order $\nCurl$ terms appear with a bad sign. Closing these estimates therefore requires separate bounds on the $\nCurl$ terms. This comes from $(\ref{rescaled-euler-w})$ which has a specific structure we can exploit, encapsulated in the following lemma.

\begin{lemma}\label{curl-form-of-equation-lemma}
	Let $(\nu,\theta)$ be the solution to~\eqref{E:EULERNEW2} on $[0,T]$, for some $N\geq2\ceil*{\alpha}+12$. Then for all $\tau\in[0,T]$, we have
	\begin{align*}
		\nCurl{\nu}=e^{-\tau}\nCurl{\nu}(0)+e^{-\tau}\int_{0}^{\tau}e^{\tau'}\comm{\dell_{\tau}}{\nCurl}\nu(\tau')d\tau'.
	\end{align*}
\end{lemma}

\begin{proof}
	To prove this, first write $(\ref{rescaled-euler-w})$ as
	\begin{align}
	\left(\theta_{\tau\tau}+\theta_{\tau}\right)+(1+\alpha)\ngrad{\left(e^{-\beta\tau}w\pjac\right)}=0.\label{curl-equation}
	\end{align}
	The equivalence of $(\ref{rescaled-euler-w})$ and $(\ref{curl-equation})$ comes from the fact that $\zeta_{\tau}=\theta_{\tau}$, an application of the Piola identity, $(\ref{piola-identity-zeta})$, and finally dividing by $e^{\beta\tau}$.
	
 As $\nCurl$ annihilates $\ngrad$, we apply $\nCurl$ to $(\ref{curl-equation})$ to obtain
	\begin{align*}
	\nCurl{\dell_{\tau}\nu}+\nCurl{\nu}=0.
	\end{align*}
	Rewrite this as
	\begin{align*}
	\dell_{\tau}\left(\nCurl{\nu}\right)+\nCurl{\nu}-\comm{\dell_{\tau}}{\nCurl}\nu=0,
	\end{align*}
	where
	\begin{align}
	\left(\comm{\dell}{\nCurl}F\right)^{i}_{j}=\dell\left(\ninv^{s}_{j}\right)\dell_{s}F^{i}-\dell\left(\ninv^{s}_{i}\right)\dell_{s}F^{j}\label{commutator-Curl-dell-tau}
	\end{align}
	for $\dell\in\{\dell_{1},\dell_{2},\dell_{3},\dell_{\tau}\}$.
	
 Multiplying by $e^{\tau}$, we have
	\begin{align*}	
	\dell_{\tau}\left(e^{\tau}\nCurl{\nu}\right)=e^{\tau}\comm{\dell_{\tau}}{\nCurl}\nu.
	\end{align*}
	Integrating in $\tau$, we arrive at
	\begin{align}
	\nCurl{\nu}=e^{-\tau}\nCurl{\nu}(0)+e^{-\tau}\int_{0}^{\tau}e^{\tau'}\comm{\dell_{\tau}}{\nCurl}\nu(\tau')d\tau'.\label{curl-equation-estimate-form-nu}
	\end{align}
\end{proof}

\begin{remark}
	We recall here that as $\zeta=x$ at $\tau=0$, we have $\nCurl\nu(0)=\Curl\nu(0)$.
\end{remark}

Before we move on to estimates, we first define some functions that will play a role in providing sufficient bounds for $\nCurl{\nu}$. Recall the definitions of $\sigma_{1}$ and $\sigma_{2}$ given in $(\ref{sigma-1})$ and $(\ref{sigma-2})$.

\begin{definition}\label{definition-of-G-curl-estimates}
	For $i=1,\dots,6$, define $G_{i}:(0,\infty)\times(0,\infty)\rightarrow\mathbb{R}$ by
	\begin{align}
	&G_{1}(\beta,\tau)=\threepartdef{\tau e^{-2\tau}}{2\sigma_{1}(\beta)>2+\sigma_{2}(\beta)}{\tau^{2} e^{-2\tau}}{2\sigma_{1}(\beta)=2+\sigma_{2}(\beta)}{\tau e^{(\sigma_{2}-2\sigma_{1})\tau}}{2\sigma_{1}(\beta)<2+\sigma_{2}(\beta)},\label{multiple-integrable-factors-G-1}\\
	&G_{2}(\beta,\tau)=\twopartdef{\tau^{2} e^{-2\tau}}{\sigma_{1}(\beta)=2+\sigma_{2}(\beta)}{\tau e^{(\sigma_{2}-\sigma_{1})\tau}}{\sigma_{1}(\beta)<2+\sigma_{2}(\beta)},\label{multiple-integrable-factors-G-2}\\
	&G_{3}(\beta,\tau)=\threepartdef{\tau e^{-2\tau}}{2\beta>2+\sigma_{2}(\beta)}{\tau^{2} e^{-2\tau}}{2\beta=2+\sigma_{2}(\beta)}{\tau e^{(\sigma_{2}-2\beta)\tau}}{2\beta<2+\sigma_{2}(\beta)},\label{multiple-integrable-factors-G-3}\\
	&G_{4}(\beta,\tau)=\threepartdef{\tau e^{-2\tau}}{\sigma_{1}(\beta)>1}{\tau^{2} e^{-2\tau}}{\sigma_{1}(\beta)=1}{\tau e^{-2\sigma_{1}\tau}}{\sigma_{1}(\beta)<1},\label{multiple-integrable-factors-G-4}\\
	&G_{5}(\beta,\tau)=\twopartdef{\tau^{2} e^{-2\tau}}{\sigma_{1}(\beta)=2}{\tau e^{-\sigma_{1}\tau}}{\sigma_{1}(\beta)<2},\label{multiple-integrable-factors-G-5}\\
	&G_{6}(\beta,\tau)=\threepartdef{\tau e^{-2\tau}}{\beta>1}{\tau^{2} e^{-2\tau}}{\beta=1}{\tau e^{-2\beta\tau}}{\beta<1}.\label{multiple-integrable-factors-G-6}
	\end{align}
\end{definition}

The first proposition will lead to estimates on $\nCurl{\nu}$.
\begin{proposition}\label{remainder-Curl-estimates-nu}
	Let $(\nu,\theta)$ be the solution to~\eqref{E:EULERNEW2} on $[0,T]$ in the sense of Theorem $\ref{local-well-posedness-theorem}$, for some $N\geq2\ceil*{\alpha}+12$. Suppose the a priori assumptions $(\ref{a-priori-assumptions})$ hold. Then for all $\tau\in[0,T]$, we have
		\begin{align}
			\sum_{m+|\underline{n}|=0}^{N}\spaceI\psi W^{1+\alpha+m}\pjac\left|\comm{\nCurl}{\Ndell{m}{n}}\nu\right|^{2}dx&+\sum_{|\underline{k}|=0}^{N}\spaceI\bar{\psi} W^{1+\alpha}\pjac\left|\comm{\nCurl}{\nabla^{\underline{k}}}\nu\right|^{2}dx\nonumber\\
			&\lesssim\delta e^{-\sigma_{1}\tau}\mathscr{S}_{N}(\tau)+\delta e^{(\sigma_{2}-\sigma_{1})\tau}\mathscr{S}_{N}(\tau)^{2},\label{Curl-Ndell-commutator-estimate-nu}
		\end{align}
		\begin{align}
			\sum_{m+|\underline{n}|=0}^{N}\spaceI\psi W^{1+\alpha+m}\pjac e^{-2\tau}\left|\Ndell{m}{n}\nCurl{\nu}(0)\right|^{2}dx&+\sum_{|\underline{k}|=0}^{N}\spaceI\bar{\psi} W^{1+\alpha}\pjac e^{-2\tau}\left|\nabla^{\underline{k}}\nCurl{\nu}(0)\right|^{2}dx\nonumber\\
			&\lesssim e^{-2\tau}\left(\mathscr{C}_{N}(0)+\delta \mathscr{S}_{N}(0)\right),\label{Curl-Ndell-commutator-estimate-nu-at-0}
		\end{align}
		\begin{align}
			&\sum_{m+|\underline{n}|=0}^{N}\spaceI\psi W^{1+\alpha+m}\pjac e^{-2\tau}\left|\int_{0}^{\tau}e^{\tau'}\Ndell{m}{n}\comm{\dell_{\tau}}{\nCurl}\nu(\tau')d\tau'\right|^{2}dx\nonumber\\
			&+\sum_{|\underline{k}|=0}^{N}\spaceI\bar{\psi} W^{1+\alpha}\pjac e^{-2\tau}\left|\int_{0}^{\tau}e^{\tau'}\nabla^{\underline{k}}\comm{\dell_{\tau}}{\nCurl}\nu(\tau')d\tau'\right|^{2}dx\nonumber\\
			&\lesssim
			(\delta e^{(\sigma_{2}-\sigma_{1})\tau}+\delta e^{-\sigma_{1}\tau})\mathscr{S}_{N}(\tau)^{2}+\sum_{i=1}^{6}\delta G_{i}(\beta,\tau)\mathscr{S}_{N}(\tau)^{2},\label{Curl-dell-tau-commutator-estimate-nu}
		\end{align}
		with $G_{i}(\beta,\tau)$, $i=1,\dots,6$, defined in $(\ref{multiple-integrable-factors-G-1})-(\ref{multiple-integrable-factors-G-6})$.
\end{proposition}

\begin{proof}We prove the three statements in the proposition separately.\\

	\noindent\textbf{Proof of $(\ref{Curl-Ndell-commutator-estimate-nu})$:} If $m+|\underline{n}|$ or $|\underline{k}|$ are $0$, then the commutators in the integrands on the left hand side of (\ref{Curl-Ndell-commutator-estimate-nu}) are $0$. Therefore we can assume in this proof that both of these quantities are $\geq1$.
	Componentwise, we write 
	\begin{align*}
	\left(\comm{\nCurl}{\Ndell{m}{n}}\nu\right)^{i}_{j}&=\Ndell{m}{n}\left(\ninv^{s}_{j}\dell_{s}\nu^{i}-\ninv^{s}_{i}\dell_{s}\nu^{j}\right)-\left(\ninv^{s}_{j}\dell_{s}\Ndell{m}{n}\nu^{i}-\ninv^{s}_{i}\dell_{s}\Ndell{m}{n}\nu^{j}\right)\nonumber\\
	&=\underbrace{\left[\Ndell{m}{n}\left(\ninv^{s}_{j}\dell_{s}\nu^{i}\right)-\ninv^{s}_{j}\dell_{s}\Ndell{m}{n}\nu^{i}\right]}_{\mathcal{A}}-\underbrace{\left[\Ndell{m}{n}\left(\ninv^{s}_{i}\dell_{s}\nu^{j}\right)-\ninv^{s}_{i}\dell_{s}\Ndell{m}{n}\nu^{j}\right]}_{\mathcal{B}},
	\end{align*}
	recalling the definition of $\comm{\ngrad}{\Ndell{m}{n}}$ given in $(\ref{commutator-ngrad-Ndell})$.
	
	 It is sufficient to concentrate on bounding $\mathcal{A}$. The strategy for bounding $\mathcal{B}$ is identical, due to the fact that the distribution of derivatives is exactly the same. First we write
	\begin{align}
	&\Ndell{m}{n}\left(\ninv^{s}_{j}\dell_{s}\nu^{i}\right)-\ninv^{s}_{j}\dell_{s}\Ndell{m}{n}\nu^{i}=\sum_{\substack{a+c=m\\ \underline{b}+\underline{d}=\underline{n}}}\mathscr{L}(a,\underline{b},c,\underline{d})\Ndell{a}{b}\left(\ninv^{s}_{j}\right)\Ndell{c}{d}\left(\dell_{s}\nu^{i}\right)-\ninv^{s}_{j}\dell_{s}\Ndell{m}{n}\nu^{i}\nonumber\\
	&=\sum_{\substack{a+c=m\\ \underline{b}+\underline{d}=\underline{n}\\1\leq a+|\underline{b}|<m+|\underline{n}|}}\mathscr{L}(a,\underline{b},c,\underline{d})\Ndell{a}{b}\left(\ninv^{s}_{j}\right)\Ndell{c}{d}\left(\dell_{s}\nu^{i}\right)+\Ndell{m}{n}\left(\ninv^{s}_{j}\right)\dell_{s}\nu^{i}+\ninv^{s}_{j}\comm{\Ndell{m}{n}}{\dell_{s}}\nu^{i},\label{nu-Ndell-curl-commutator-term-expansion}
	\end{align}
	where $\mathscr{L}$ are the coefficients of expansion. For $\Ndell{m}{n}\left(\ninv^{s}_{j}\right)\dell_{s}\nu^{i}$, we have the bound
	\begin{align*}
	\spaceI\psi W^{1+\alpha+m}\pjac\left|\Ndell{m}{n}\ninv^{s}_{j}\right|^{2}\left|\dell_{s}\nu^{i}\right|^{2}dx\lesssim\spaceI\psi W^{1+\alpha+m}\pjac\left|\Ndell{m}{n}\ninv\right|^{2}\left|\nabla\nu\right|^{2}dx. 
	\end{align*}
	For the $\nabla\nu$ term, we first use $(\ref{rectangular-as-ang-rad})$ to convert $\nabla$ in to a sum of radial and angular derivatives, after which Lemma \ref{L-infinity-energy-space-bound-no-weights} gives
	\begin{align}
	\left\|\nabla\nu\right\|^{2}_{L^{\infty}\left(\supp{\psi}\right)}\lesssim\left\|\nu\right\|_{\X{N}}^{2}\lesssim\delta e^{-\sigma_{1}\tau}\mathscr{S}_{N}(\tau).\label{Dnu-estimate}
	\end{align}
	Then we have
	\begin{align}
	\spaceI\psi W^{1+\alpha+m}\pjac\left|\Ndell{m}{n}\ninv\right|^{2}\left|\nabla\nu\right|^{2}dx\lesssim\delta e^{-\sigma_{1}\tau}\mathscr{S}_{N}(\tau)\underbrace{\spaceI\psi W^{1+\alpha+m}\pjac\left|\Ndell{m}{n}\ninv\right|^{2}dx}_{\mathcal{I}}\lesssim\delta e^{(\sigma_{2}-\sigma_{1})\tau}\mathscr{S}_{N}(\tau)^{2},\label{Ndell-ninv-estimate}
	\end{align}
	where we employ an analogous argument to the one used in Lemma \ref{lower-order-estimates-lemma} to bound $\mathcal{I}$ and go from the second inequality to the third.
	
 The strategy for bounding the sum of lower order terms in $(\ref{nu-Ndell-curl-commutator-term-expansion})$ is also adapted from Lemma \ref{lower-order-estimates-lemma}. 
	\begin{align}
	\spaceI\psi W^{1+\alpha+m}\pjac\left|\sum_{\substack{a+c=m\\ \underline{b}+\underline{d}=\underline{n}\\1\leq a+|\underline{b}|<m+|\underline{n}|}}\mathscr{L}(a,\underline{b},c,\underline{d})\Ndell{a}{b}\left(\ninv^{s}_{j}\right)\Ndell{c}{d}\left(\dell_{s}\nu^{i}\right)\right|^{2}dx\lesssim\delta e^{(\sigma_{2}-\sigma_{1})\tau}\mathscr{S}_{N}(\tau)^{2}+\delta e^{-\sigma_{1}\tau}\mathscr{S}_{N}(\tau)^{2}.\label{lower-order-sum-nu-Ndell-Curl-commutator-estimate}
	\end{align}
	The extra decay comes from the fact that we are bounding $\nabla\nu$ terms as well as $\nabla\theta$ terms.
	
	 The remaining term from $(\ref{nu-Ndell-curl-commutator-term-expansion})$ is $\ninv^{s}_{j}\comm{\Ndell{m}{n}}{\dell_{s}}\nu^{i}$. We apply Lemma \ref{higher-order-commutator-identities} to expand the commutator, and  obtain
	\begin{align}
	\sum_{s=1}^{3}\spaceI\psi W^{1+\alpha+m}\pjac\left|\ninv\right|\left|\comm{\Ndell{m}{n}}{\dell_{s}}\nu\right|^{2}dx\lesssim\delta e^{-\sigma_{1}\tau}\mathscr{S}_{N}(\tau).\label{lower-order-nu-Ndell-Curl-commutator-estimate}
	\end{align}
	Combining $(\ref{Ndell-ninv-estimate})$, $(\ref{lower-order-sum-nu-Ndell-Curl-commutator-estimate})$, and $(\ref{lower-order-nu-Ndell-Curl-commutator-estimate})$, along with analogous estimates on $\supp{\bar{\psi}}$, gives $(\ref{Curl-Ndell-commutator-estimate-nu})$.\\
	
	\noindent\textbf{Proof of (\ref{Curl-Ndell-commutator-estimate-nu-at-0}).}
	Similarly to the proof of $(\ref{Curl-Ndell-commutator-estimate-nu})$, we make use of the fact that
	\begin{align*}
	\Ndell{m}{n}\nCurl{\nu}(0)=\nCurl{\Ndell{m}{n}\nu}(0)+\comm{\Ndell{m}{n}}{\nCurl}\nu(0).
	\end{align*}
	The lower order commutator term is controlled by the estimates above for $(\ref{Curl-Ndell-commutator-estimate-nu})$. The top order $\nCurl{\Ndell{m}{n}\theta}$ term requires $\mathscr{C}_{N}$ to control it. Therefore, we have that 
	\begin{align}
	\spaceI\psi W^{1+\alpha+m}\pjac e^{-2\tau}\left|\Ndell{m}{n}\nCurl{\nu}(0)\right|^{2}dx\lesssim e^{-2\tau}\left(\mathscr{C}_{N}(0)+\delta \mathscr{S}_{N}(0)\right).\label{nu-Ndell-curl-commutator-estimate-2}
	\end{align}
	Once again, estimates on $\supp{\bar{\psi}}$ are analogous, so we have (\ref{Curl-Ndell-commutator-estimate-nu-at-0}).\\

\noindent \textbf{Proof of (\ref{Curl-dell-tau-commutator-estimate-nu}).} Finally we look at
\begin{align*}
    \spaceI\psi W^{1+\alpha+m}\pjac e^{-2\tau}\left|\int_{0}^{\tau}e^{\tau'}\Ndell{m}{n}\comm{\dell_{\tau}}{\nCurl}\nu(\tau')d\tau'\right|^{2}dx.
\end{align*}
Recall that
\begin{align*}
    \left(\Ndell{m}{n}\comm{\dell_{\tau}}{\nCurl}\nu\right)^{i}_{j}=\Ndell{m}{n}\left(\dell_{\tau}\left(\ninv^{s}_{j}\right)\dell_{s}\nu^{i}\right)-\Ndell{m}{n}\left(\dell_{\tau}\left(\ninv^{s}_{i}\right)\dell_{s}\nu^{j}\right).
\end{align*}
On the right hand side, both terms have the same distribution of derivatives, as in (\ref{Curl-Ndell-commutator-estimate-nu}), so we focus on $\Ndell{m}{n}\left(\left(\dell_{\tau}\ninv^{s}_{j}\right)\dell_{s}\nu^{i}\right)$. First, using $(\ref{inverse-differentiation-zeta})$, we write this term as
\begin{align*}
    \Ndell{m}{n}\left(\dell_{\tau}\left(\ninv^{s}_{j}\right)\dell_{s}\nu^{i}\right)=-\Ndell{m}{n}\left(\ninv^{k}_{j}\ninv^{s}_{l}\dell_{k}\nu^{l}\dell_{s}\nu^{i}\right).
\end{align*}
When $m+|\underline{n}|=0$ we have the bound
	\begin{align}
	&\spaceI\psi W^{1+\alpha+m}\pjac e^{-2\tau}\left(\int_{0}^{\tau}e^{\tau'}\left|\ninv^{k}_{j}\ninv^{s}_{l}\dell_{k}\nu^{l}\dell_{k}\nu^{s}\right|d\tau'\right)^{2}dx\nonumber\\
	&\lesssim\tau e^{-2\tau}\spaceI\psi W^{1+\alpha+m}\pjac\left(\int_{0}^{\tau}e^{2\tau'}\left|\ninv\right|^{4}\left|\nabla\nu\right|^{4}d\tau'\right)dx\nonumber\\
	&\lesssim\delta^{2}\tau e^{-2\tau}\mathscr{S}_{N}(\tau)^{2}\left(\int_{0}^{\tau}e^{(2-2\sigma_{1})\tau'}d\tau'\right)\lesssim\delta G_{4}(\beta,\tau)\mathscr{S}_{N}(\tau)^{2}.\label{nu-Ndell-tau-commutator-estimate-zero-order}
	\end{align}
	The first inequality comes from Cauchy-Schwarz inequality on the $\tau$ integral. The second comes from bounding $\ninv$, $\njac$, and $\nabla\nu$ in $L^{\infty}$, using the a priori assumptions (\ref{a-priori-assumptions}) for the first two, and (\ref{Dnu-estimate}) on the last, as well as noting the integral of $\psi W^{1+\alpha+m}$ over $\Omega$ is finite. The last inequality is by definition of $G_{4}$.

Thus from now we can assume $m+|\underline{n}|\geq1$. We then expand the expression:
\begin{align}
    \Ndell{m}{n}\left(\ninv^{k}_{j}\ninv^{s}_{l}\dell_{k}\nu^{l}\dell_{s}\nu^{i}\right)&=\Ndell{m}{n}\left(\ninv^{k}_{j}\right)\ninv^{s}_{l}\dell_{k}\nu^{l}\dell_{s}\nu^{i}+\ninv^{k}_{j}\Ndell{m}{n}\left(\ninv^{s}_{l}\right)\dell_{k}\nu^{l}\dell_{s}\nu^{i}\nonumber\\
    &+\ninv^{k}_{j}\ninv^{s}_{l}\Ndell{m}{n}\left(\dell_{k}\nu^{l}\right)\dell_{s}\nu^{i}+\ninv^{k}_{j}\ninv^{s}_{l}\dell_{k}\nu^{l}\Ndell{m}{n}\left(\dell_{s}\nu^{i}\right)\nonumber\\
    &+\sum_{\substack{a+c+e+g=m\\ \underline{b}+\underline{d}+\underline{f}+\underline{h}\\1\leq g+|\underline{h}|<m+|\underline{n}|}}\mathscr{L}(a,\underline{b},c,\underline{d},e,\underline{f},g,\underline{h})\Ndell{a}{b}\left(\ninv^{k}_{j}\right)\Ndell{c}{d}\left(\ninv^{s}_{l}\right)\Ndell{e}{f}\left(\dell_{k}\nu^{l}\right)\Ndell{g}{h}\left(\dell_{s}\nu^{i}\right).\label{nu-Ndell-tau-commutator-expansion-1}
\end{align}
The first four terms on the right hand side are top order. We look at the first and third of these terms, as the estimation strategies for the other two are analogous. We are left to consider
\begin{align*}
    \underbrace{\Ndell{m}{n}\left(\ninv^{k}_{j}\right)\ninv^{s}_{l}\dell_{k}\nu^{l}\dell_{s}\nu^{i}}_{\mathcal{A}}+\underbrace{\ninv^{k}_{j}\ninv^{s}_{l}\Ndell{m}{n}\left(\dell_{k}\nu^{l}\right)\dell_{s}\nu^{i}}_{\mathcal{B}}.
\end{align*}
The term $\mathcal{A}$ gives
\begin{align}
	&\spaceI\psi W^{1+\alpha+m}\pjac e^{-2\tau}\left(\int_{0}^{\tau}e^{\tau'}\left|\Ndell{m}{n}\left(\ninv^{k}_{j}\right)\ninv^{s}_{l}\dell_{k}\nu^{l}\dell_{s}\nu^{i}\right|d\tau'\right)^{2}dx\nonumber\\
    &\lesssim\spaceI\psi W^{1+\alpha+m}\pjac e^{-2\tau}\left(\int_{0}^{\tau}e^{\tau'}\left|\ninv\right|\left|\Ndell{m}{n}\ninv\right|\left|\nabla\nu\right|^{2}d\tau'\right)^{2}dx\nonumber\\
    &\lesssim \tau e^{-2\tau}\spaceI\psi W^{1+\alpha+m}\pjac\left(\sup_{0\leq\tau'\leq\tau}e^{-\sigma_{2}\tau'}\left|\Ndell{m}{n}\ninv\right|^{2}\right)\left(\int_{0}^{\tau}e^{(2+\sigma_{2})\tau'}\left|\nabla\nu\right|^{4}d\tau'\right)dx\nonumber\\
    &\lesssim \delta^{2}\tau e^{-2\tau}\mathscr{S}_{N}(\tau)^{2}\left(\int_{0}^{\tau}e^{(2+\sigma_{2}-2\sigma_{1})\tau'}d\tau'\right)\underbrace{\left(\sup_{0\leq\tau'\leq\tau}e^{-\sigma_{2}\tau'}\spaceI\psi W^{1+\alpha+m}\pjac\left|\Ndell{m}{n}\ninv\right|^{2}dx\right)}_{\mathcal{A}_{1}}.\label{nu-Ndell-tau-commutator-estimate-1a}
\end{align}
To go from the first line to the second, we employ Cauchy-Schwarz on the $\tau$ integral, along with a priori assumptions from $(\ref{a-priori-assumptions})$ for the $\ninv$ terms. From the second to third line we use (\ref{Dnu-estimate}). To go from the third to fourth line, we make use of the a priori assumptions $(\ref{a-priori-assumptions})$, to estimate $\pjac$ (which is of order $1$) by its  supremum over $[0,\tau]$.
Then we use an argument in the style of~Lemma \ref{lower-order-estimates-lemma} to handle $\mathcal{A}_{1}$, and  we obtain
\begin{align}
\left(\sup_{0\leq\tau'\leq\tau}e^{-\sigma_{2}\tau'}\spaceI\psi W^{1+\alpha+m}\pjac\left|\Ndell{m}{n}\ninv\right|^{2}dx\right)\lesssim\left(\sup_{0\leq\tau'\leq\tau}e^{-\sigma_{2}\tau'}\delta e^{\sigma_{2}\tau'}\mathscr{S}_{N}(\tau')\right)\lesssim\mathscr{S}_{N}(\tau),\label{nu-Ndell-tau-commutator-estimate-1b}
\end{align}
where we have modified the bound from Lemma \ref{lower-order-estimates-lemma}, as $e^{\sigma_{2}\tau'}\geq1$. Combining $(\ref{nu-Ndell-tau-commutator-estimate-1a})$ and $(\ref{nu-Ndell-tau-commutator-estimate-1b})$ we get
\begin{align}
\spaceI\psi W^{1+\alpha+m}\pjac e^{-2\tau}\left(\int_{0}^{\tau}e^{\tau'}\left|\Ndell{m}{n}\left(\ninv^{k}_{j}\right)\ninv^{s}_{l}\dell_{k}\nu^{l}\dell_{s}\nu^{i}\right|d\tau'\right)^{2}dx\lesssim\delta G_{1}(\beta,\tau)\mathscr{S}_{N}(\tau)^{2},\label{nu-Ndell-tau-commutator-estimate-1c}
\end{align}
with $G_{1}$ defined in $(\ref{multiple-integrable-factors-G-1})$, and any extra powers of $\delta$ and $\mathscr{S}_{N}$ bounded by a constant.

For the third term on the right hand side in $(\ref{nu-Ndell-tau-commutator-expansion-1})$, $\mathcal{B}$, there are too many derivatives on $\dell_{k}\nu^{l}$, so we write
\begin{align*}
    \ninv^{k}_{j}\ninv^{s}_{l}\Ndell{m}{n}\left(\dell_{k}\nu^{l}\right)\dell_{s}\nu^{i}&=\dell_{\tau}\left(\ninv^{k}_{j}\ninv^{s}_{l}\Ndell{m}{n}\left(\dell_{k}\theta^{l}\right)\dell_{s}\nu^{i}\right)\nonumber\\
    &-\dell_{\tau}\left(\ninv^{k}_{j}\ninv^{s}_{l}\right)\Ndell{m}{n}\left(\dell_{k}\theta^{l}\right)\dell_{s}\nu^{i}-\ninv^{k}_{j}\ninv^{s}_{l}\Ndell{m}{n}\left(\dell_{k}\theta^{l}\right)\dell_{s}\nu_{\tau}^{i},
\end{align*}
and commuting $\dell_{k}$ with $\Ndell{m}{n}$, we get
\begin{align}
    \ninv^{k}_{j}\ninv^{s}_{l}\Ndell{m}{n}\left(\dell_{k}\nu^{l}\right)\dell_{s}\nu^{i}&=\dell_{\tau}\left(\ninv^{k}_{j}\ninv^{s}_{l}\dell_{k}\Ndell{m}{n}\theta^{l}\dell_{s}\nu^{i}\right)-\dell_{\tau}\left(\ninv^{k}_{j}\ninv^{s}_{l}\right)\dell_{k}\Ndell{m}{n}\theta^{l}\dell_{s}\nu^{i}\nonumber\\
    &-\ninv^{k}_{j}\ninv^{s}_{l}\dell_{k}\Ndell{m}{n}\theta^{l}\dell_{s}\nu_{\tau}^{i}+\dell_{\tau}\left(\ninv^{k}_{j}\ninv^{s}_{l}\comm{\Ndell{m}{n}}{\dell_{k}}\theta^{l}\dell_{s}\nu^{i}\right)\nonumber\\
    &-\dell_{\tau}\left(\ninv^{k}_{j}\ninv^{s}_{l}\right)\comm{\Ndell{m}{n}}{\dell_{k}}\theta^{l}\dell_{s}\nu^{i}-\ninv^{k}_{j}\ninv^{s}_{l}\comm{\Ndell{m}{n}}{\dell_{k}}\theta^{l}\dell_{s}\nu_{\tau}^{i}.\label{nu-Ndell-tau-commutator-expansion-2}
\end{align}
The first term we can write as $\dell_{\tau}\left(\ninv^{s}_{l}\left[\ngrad{\Ndell{m}{n}\theta}\right]^{l}_{j}\dell_{s}\nu^{i}\right)$, and we estimate
\begin{align}
	&\spaceI\psi W^{1+\alpha+m}\pjac e^{-2\tau}\left|\int_{0}^{\tau}e^{\tau'}\dell_{\tau}\left(\ninv^{s}_{l}\left[\ngrad{\Ndell{m}{n}\theta}\right]^{l}_{j}\dell_{s}\nu^{i}\right)d\tau'\right|^{2}\nonumber\\
    &=\spaceI\psi W^{1+\alpha+m}\pjac e^{-2\tau}\left|\left[e^{\tau'}\ninv^{s}_{l}\left[\ngrad{\Ndell{m}{n}\theta}\right]^{l}_{j}\dell_{s}\nu^{i}\right]^{\tau}_{0}-\int_{0}^{\tau}e^{\tau'}\ninv^{s}_{l}\left[\ngrad{\Ndell{m}{n}\theta}\right]^{l}_{j}\dell_{s}\nu^{i}(\tau')d\tau'\right|^{2}dx\nonumber\\
    &\lesssim\spaceI\psi W^{1+\alpha+m}\pjac\left(\left|\ninv\right|^{2}\left|\ngrad{\Ndell{m}{n}\theta}\right|^{2}\left|\nabla\nu\right|^{2}+e^{-2\tau}\left|\ngrad{\Ndell{m}{n}\theta}(0)\right|^{2}\left|\nabla\nu(0)\right|^{2}\right)dx\nonumber\\
    &+\spaceI\psi W^{1+\alpha+m}\pjac e^{-2\tau}\left(\int_{0}^{\tau}e^{\tau'}\left|\ninv\right|\left|\ngrad{\Ndell{m}{n}\theta}\right|\left|\nabla\nu\right|d\tau'\right)^{2}dx\nonumber\\
    &\lesssim \delta e^{(\sigma_{2}-\sigma_{1})\tau}\mathscr{S}_{N}(\tau)^{2}
    +\delta G_{2}(\beta,\tau)\mathscr{S}_{N}(\tau)^{2},\label{nu-Ndell-tau-commutator-estimate-2}
\end{align}
where $G_{2}$ is defined in $(\ref{multiple-integrable-factors-G-2})$. Note that the $\tau=0$ term does not contribute to the estimate as $\theta=0$ initially.

For the second term on the right hand side of $(\ref{nu-Ndell-tau-commutator-expansion-2})$, we first employ (\ref{jacobian-differentiation-zeta}) which gives
\begin{align*}
    \dell_{\tau}\left(\ninv^{k}_{j}\ninv^{s}_{l}\right)\dell_{k}\Ndell{m}{n}\theta^{l}\dell_{s}\nu^{i}=-\ninv^{s}_{r}\ninv^{p}_{l}\left[\ngrad{\Ndell{m}{n}\theta}\right]^{l}_{j}\dell_{s}\nu^{i}\dell_{p}\nu^{r}-\ninv^{s}_{l}\ninv^{p}_{j}\left[\ngrad{\Ndell{m}{n}\theta}\right]^{l}_{r}\dell_{s}\nu^{i}\dell_{p}\nu^{r},
\end{align*}
and we can control both of these terms by the following estimate:
\begin{align}
    \spaceI\psi W^{1+\alpha+m}\pjac e^{-2\tau}\left|\int_{0}^{\tau}e^{\tau'}\left|\ninv\right|^{2}\left|\ngrad{\Ndell{m}{n}\theta}\right|\left|\nabla\nu\right|^{2}d\tau'\right|^{2}dx\lesssim\delta^{2}G_{1}(\beta,\tau)\mathscr{S}_{N}(\tau)^{3}\lesssim\delta G_{1}(\beta,\tau)\mathscr{S}_{N}(\tau)^{2}.\label{nu-Ndell-tau-commutator-estimate-3}
\end{align}
For the third term, $\dell_{s}\nu_{\tau}^{i}$ has too many $\tau$ derivatives to be in our energy space. We use (\ref{rescaled-euler-W}) to relate $\nu_{\tau}$ to lower order terms. We rearrange equation $(\ref{rescaled-euler-W})$ to obtain
\begin{align*}
    \nu_{\tau}^{i}=-\nu^{i}-\delta e^{-\beta\tau}\left(W\dell_{k}\left(\ninv^{k}_{i}\pjac\right)+(1+\alpha)\ninv^{k}_{i}\pjac\right),
\end{align*}
and write the third term on the right hand side of (\ref{nu-Ndell-tau-commutator-expansion-2}) as
\begin{align*}
    -\ninv^{s}_{l}\left[\ngrad{\Ndell{m}{n}\theta}\right]^{l}_{j}\dell_{s}\nu_{\tau}^{i}=\ninv^{s}_{l}\left[\ngrad{\Ndell{m}{n}\theta}\right]^{l}_{j}\dell_{s}\nu^{i}&+\delta e^{-\beta\tau}\ninv^{s}_{l}\left[\ngrad{\Ndell{m}{n}\theta}\right]^{l}_{j}\dell_{s}\left(W\dell_{k}\left(\ninv^{k}_{i}\pjac\right)\right)\nonumber\\
    &+\delta(1+\alpha)e^{-\beta\tau}\ninv^{s}_{l}\left[\ngrad{\Ndell{m}{n}\theta}\right]^{l}_{j}\dell_{s}\left(\ninv^{k}_{i}\pjac\right).
\end{align*}
These can all be estimated using methods shown previously, as the terms that come from rewriting $\nu_{\tau}^{i}$ are of low order. We have
\begin{align}
    \spaceI\psi W^{1+\alpha+m}\pjac e^{-2\tau}\left|\int_{0}^{\tau}e^{\tau'}\ninv^{s}_{l}\left[\ngrad{\Ndell{m}{n}\theta}\right]^{l}_{j}\dell_{s}\nu_{\tau}^{i}d\tau'\right|^{2}dx\lesssim\delta G_{2}(\beta,\tau)\mathscr{S}_{N}(\tau)^{2}+\delta G_{3}(\beta,\tau)\mathscr{S}_{N}(\tau)^{2},\label{nu-Ndell-tau-commutator-estimate-4}
\end{align}
where $G_{3}$ is defined in $(\ref{multiple-integrable-factors-G-3})$.

For the last three terms on the right hand side of (\ref{nu-Ndell-tau-commutator-expansion-2}) we first use Lemma \ref{higher-order-commutator-identities} to expand the commutator terms, which can then all be controlled by $\left\|\theta\right\|_{\X{N}}$. Employing Lemma $\ref{theta-estimates-lemma}$, we have the bounds
\begin{align}
	&\spaceI\psi W^{1+\alpha+m}\pjac e^{-2\tau}\left|\int_{0}^{\tau}e^{\tau'}\dell_{\tau}\left(\ninv^{k}_{j}\ninv^{s}_{l}\comm{\Ndell{m}{n}}{\dell_{k}}\theta^{l}\dell_{s}\nu^{i}\right)d\tau'\right|^{2}dx\lesssim \delta e^{-\sigma_{1}\tau}\mathscr{S}_{N}(\tau)^{2}
	+\delta G_{5}(\beta,\tau)\mathscr{S}_{N}(\tau)^{2},\nonumber\\
	&\spaceI\psi W^{1+\alpha+m}\pjac e^{-2\tau}\left|\int_{0}^{\tau}e^{\tau'}\dell_{\tau}\left(\ninv^{k}_{j}\ninv^{s}_{l}\right)\comm{\Ndell{m}{n}}{\dell_{k}}\theta^{l}\dell_{s}\nu^{i}d\tau'\right|^{2}dx\lesssim \delta G_{4}(\beta,\tau)\mathscr{S}_{N}(\tau)^{2},\nonumber\\
	&\spaceI\psi W^{1+\alpha+m}\pjac e^{-2\tau}\left|\int_{0}^{\tau}e^{\tau'}\ninv^{k}_{j}\ninv^{s}_{l}\comm{\Ndell{m}{n}}{\dell_{k}}\theta^{l}\dell_{s}\nu_{\tau}^{i}d\tau'\right|^{2}dx\lesssim \delta G_{5}(\beta,\tau)\mathscr{S}_{N}(\tau)^{2}+\delta G_{6}(\beta,\tau)\mathscr{S}_{N}(\tau)^{2},\label{nu-Ndell-tau-commutator-estimate-5}
\end{align}
with $G_{5}$ and $G_{6}$ defined in $(\ref{multiple-integrable-factors-G-5})$ and $(\ref{multiple-integrable-factors-G-6})$.

Finally, we bound the last term on the right hand side of $(\ref{nu-Ndell-tau-commutator-expansion-1})$, the sum of lower order terms. This term is bounded analogously to $(\ref{nu-Ndell-tau-commutator-estimate-1c})$, and once again uses an argument adapted from Lemma \ref{lower-order-estimates-lemma}. If we denote the sum $\mathcal{L}$, we have the bound
\begin{align}
\spaceI\psi W^{1+\alpha+m}\pjac e^{-2\tau}\left|\int_{0}^{\tau}e^{\tau'}\mathcal{L}\ d\tau'\right|^{2}dx\lesssim\delta G_{1}(\beta,\tau)\mathscr{S}_{N}(\tau)^{2}+\delta G_{4}(\beta,\tau)\mathscr{S}_{N}(\tau)^{2}.\label{lower-order-sum-Curl-tau-commutator-estimate}
\end{align}
Bounds $(\ref{nu-Ndell-tau-commutator-estimate-1c})$, and $(\ref{nu-Ndell-tau-commutator-estimate-2})-(\ref{lower-order-sum-Curl-tau-commutator-estimate})$, along with analogous estimates on $\supp{\bar{\psi}}$ give $(\ref{Curl-dell-tau-commutator-estimate-nu})$, which completes the proof of Proposition $\ref{remainder-Curl-estimates-nu}$.
\end{proof}

Proposition \ref{remainder-Curl-estimates-nu} is used to prove the main estimate for $\nCurl{\nu}$.

\begin{theorem}\label{Curl-estimates-theorem-nu}
	Let $(\nu,\theta)$ be the solution to~\eqref{E:EULERNEW2} on $[0,T]$ in the sense of Theorem $\ref{local-well-posedness-theorem}$, for some $N\geq2\ceil*{\alpha}+12$. Suppose the a priori assumptions $(\ref{a-priori-assumptions})$ hold. Then for all $\tau\in[0,T]$, we have
	\begin{align}
	\left\|\nu\right\|_{\Y{N}{\nCurl}}^{2}\lesssim & e^{-2\tau}\left(\mathscr{C}_{N}(0)+\delta \mathscr{S}_{N}(0)\right)+\delta e^{-\sigma_{1}\tau}\mathscr{S}_{N}(\tau)+(\delta e^{(\sigma_{2}-\sigma_{1})\tau} \notag \\
&	+\delta^{2}e^{-\sigma_{1}\tau})\mathscr{S}_{N}(\tau)^{2}+\sum_{i=1}^{6}\delta G_{i}(\beta,\tau)\mathscr{S}_{N}(\tau)^{2},\label{main-Curl-nu-estimate}
	\end{align}
	for $G_{i}$ defined in $\ref{definition-of-G-curl-estimates}$, $i=1,\dots,6$.
\end{theorem}

\begin{proof}
	Act on (\ref{curl-equation-estimate-form-nu}) with $\Ndell{m}{n}$ to get
	\begin{align*}
	\nCurl{\Ndell{m}{n}\nu}=\comm{\nCurl}{\Ndell{m}{n}}\nu+e^{-\tau}\Ndell{m}{n}\nCurl{\nu}(0)+e^{-\tau}\int_{0}^{\tau}e^{\tau'}\Ndell{m}{n}\comm{\dell_{\tau}}{\nCurl}\nu(\tau')d\tau'.
	\end{align*}
	From here we take the modulus and then square both sides, then use Young's inequality:
	\begin{align*}
	&\left|\nCurl{\Ndell{m}{n}\nu}\right|^{2}=\left|\comm{\nCurl}{\Ndell{m}{n}}\nu+e^{-\tau}\Ndell{m}{n}\nCurl{\nu}(0)+e^{-\tau}\int_{0}^{\tau}e^{\tau'}\Ndell{m}{n}\comm{\dell_{\tau}}{\nCurl}\nu(\tau')d\tau'\right|^{2}\\
	&\lesssim\left|\comm{\nCurl}{\Ndell{m}{n}}\nu\right|^{2}+\left|e^{-\tau}\Ndell{m}{n}\nCurl{\nu}(0)\right|^{2}+\left|e^{-\tau}\int_{0}^{\tau}e^{\tau'}\Ndell{m}{n}\comm{\dell_{\tau}}{\nCurl}\nu(\tau')d\tau'\right|^{2}.
	\end{align*}
	Multiply by $\psi W^{1+\alpha+m}\pjac$ and integrate over $\Omega$. Finally sum over $0\leq m+|\underline{n}|\leq N$ to get
	\begin{align}
	&\sum_{m+|\underline{n}|=0}^{N}\spaceI\psi W^{1+\alpha+m}\pjac\left|\nCurl{\Ndell{m}{n}\nu}\right|^{2}dx \notag \\
	& \lesssim\sum_{m+|\underline{n}|=0}^{N}\spaceI\psi W^{1+\alpha+m}\pjac\left|\comm{\nCurl}{\Ndell{m}{n}}\nu\right|^{2}dx\nonumber\\
	&\qquad +\sum_{m+|\underline{n}|=0}^{N}\spaceI\psi W^{1+\alpha+m}\pjac e^{-2\tau}\left|\Ndell{m}{n}\nCurl{\nu}(0)\right|^{2}dx\nonumber\\
	&\qquad +\sum_{m+|\underline{n}|=0}^{N}\spaceI\psi W^{1+\alpha+m}\pjac e^{-2\tau}\left|\int_{0}^{\tau}e^{\tau'}\Ndell{m}{n}\comm{\dell_{\tau}}{\nCurl}\nu(\tau')d\tau'\right|^{2}dx.\label{curl-equation-nu-terms-to-estimate}
	\end{align}
	The result is then obtained by using Proposition \ref{remainder-Curl-estimates-nu}, along with analogous estimates for $\supp{\bar{\psi}}$.
\end{proof}
We turn our attention to estimates for $\nCurl{\theta}$. First we need to define more functions as in Definition \ref{definition-of-G-curl-estimates} that play a role in proving sufficient bounds for $\nCurl{\theta}$.

\begin{definition}\label{definition-of-H-G-dash-curl-estimates}For $i=1,\dots,6$, define $\tilde{G}_{i}:(0,\infty)\times(0,\infty)\rightarrow\mathbb{R}$ by
	\begin{align}
	\tilde{G}_{i}(\beta,\tau)=\int_{0}^{\tau}e^{\frac{\sigma_{1}\tau'}{2}}G_{i}(\beta,\tau')d\tau',\label{multiple-integrable-factors-G-dash-i}
	\end{align}
	with $G_{1},\dots,G_{6}$ as in Definition~\ref{definition-of-G-curl-estimates}. Furthermore we define $H_{1}, H_{2}:(0,\infty)\times(0,\infty)\rightarrow\mathbb{R}$ by
	\begin{align}
	&H_{1}(\beta,\tau)=\threepartdef{e^{(\sigma_{2}-\sigma_{1})\tau}}{\sigma_{1}(\beta)<\sigma_{2}(\beta)}{\tau^{2}}{\sigma_{1}(\beta)=\sigma_{2}(\beta)}{1}{\sigma_{1}(\beta)>\sigma_{2}(\beta)},\label{multiple-integrable-factors-H-1}\\
	&H_{2}(\beta,\tau)=\threepartdef{e^{(\sigma_{2}-\frac{\sigma_{1}}{2})\tau}}{\sigma_{1}(\beta)<2\sigma_{2}(\beta)}{\tau}{\sigma_{1}(\beta)=2\sigma_{2}(\beta)}{1}{\sigma_{1}(\beta)>2\sigma_{2}(\beta)}.\label{multiple-integrable-factors-H-2}
	\end{align}
\end{definition}

The estimates for $\nCurl{\theta}$ come directly from Theorem~\ref{Curl-estimates-theorem-nu} along with an application of the Fundamental Theorem of Calculus.

\begin{theorem}\label{Curl-estimates-theorem-theta}
	Let $(\nu,\theta)$ be the solution to~\eqref{E:EULERNEW2} on $[0,T]$ in the sense of Theorem $\ref{local-well-posedness-theorem}$. Let $N$ be an integer such that $N\geq2\ceil*{\alpha}+12$. Suppose the a priori assumptions $(\ref{a-priori-assumptions})$ hold. Then for all $\tau\in[0,T]$, we have
	\begin{align}
	\left\|\theta\right\|_{\Y{N}{\nCurl}}^{2}\lesssim & \left(\mathscr{C}_{N}(0)+\delta \mathscr{S}_{N}(0)\right)+\delta \mathscr{S}_{N}(\tau)+\delta \mathscr{S}_{N}(\tau)^{2}+\delta (H_{1}(\beta,\tau)+H_{2}(\beta,\tau))\mathscr{S}_{N}(\tau)^{2} \notag \\
	&+\sum_{i=1}^{6}\delta\tilde{G}_{i}(\beta,\tau)\mathscr{S}_{N}(\tau)^{2}, \label{main-Curl-theta-estimate-beta-greater-than-2}
	\end{align}
	with $\tilde{G}_{i}$, $i=1,\dots,6$, $H_{1}$, and $H_{2}$ defined in $\ref{definition-of-H-G-dash-curl-estimates}$.
\end{theorem}

\begin{proof}
	Define the tensor $\mathbb{K}(m,\underline{n})$, which is given in components by
	\begin{align}
	\mathbb{K}(m,\underline{n})^{i}_{j}=-\ninv^{s}_{j}\dell_{s}\nu^{l}\left[\ngrad{\Ndell{m}{n}\theta}\right]^{i}_{l}. \label{Curl-with-A-tau-differentiated}
	\end{align}
	Using the Fundamental Theorem of Calculus we write
	\begin{align*}
	\nCurl{\Ndell{m}{n}\theta}=\int_{0}^{\tau}\dell_{\tau}\left(\nCurl{\Ndell{m}{n}\theta}\right)d\tau',
	\end{align*}
	because $\theta$ and all its spatial derivatives are $0$ at $\tau=0$. In component form this gives
	\begin{align*}
	\left[\dell_{\tau}\left(\nCurl{\Ndell{m}{n}\theta}\right)\right]^{i}_{j}&=\dell_{\tau}\left(\ninv^{k}_{j}\right)\dell_{k}\left(\Ndell{m}{n}\theta^{i}\right)-\dell_{\tau}\left(\ninv^{k}_{i}\right)\dell_{k}\left(\Ndell{m}{n}\theta^{j}\right)\\
	&+\ninv^{k}_{j}\dell_{k}\left(\Ndell{m}{n}\nu^{i}\right)-\ninv^{k}_{i}\dell_{k}\left(\Ndell{m}{n}\nu^{j}\right).
	\end{align*}
	This means
	\begin{align}
	\dell_{\tau}\left(\nCurl{\Ndell{m}{n}\theta}\right)=\mathbb{K}(m,\underline{n})-\mathbb{K}(m,\underline{n})^{\intercal}+\nCurl{\Ndell{m}{n}\nu},\label{decomposition-dell-tau-Curl-Ndell-theta}
	\end{align}
	because
	\begin{align*}
	\dell_{\tau}\left(\ninv^{k}_{j}\right)\dell_{k}\left(\Ndell{m}{n}\theta^{i}\right)=-\ninv^{s}_{j}\ninv^{k}_{l}\dell_{s}\nu^{l}\dell_{k}\left(\Ndell{m}{n}\theta^{i}\right)=-\ninv^{s}_{j}\dell_{s}\nu^{l}\left[\ngrad{\Ndell{m}{n}\theta}\right]^{i}_{l}.
	\end{align*}
	Now we look to bound the integral
	\begin{align*}
	\spaceI\psi W^{1+\alpha+m}\pjac\left|\nCurl{\Ndell{m}{n}\theta}\right|^{2}dx&=\spaceI\psi W^{1+\alpha+m}\pjac\left|\int_{0}^{\tau}\dell_{\tau}\left(\nCurl{\Ndell{m}{n}\theta}\right)d\tau'\right|^{2}dx.
	\end{align*}
	Using the decomposition (\ref{decomposition-dell-tau-Curl-Ndell-theta}) we obtain
	\begin{align*}
	\spaceI\psi W^{1+\alpha+m}\pjac\left|\nCurl{\Ndell{m}{n}\theta}\right|^{2}dx&\lesssim \spaceI\psi W^{1+\alpha+m}\pjac\left|\int_{0}^{\tau}\nCurl{\Ndell{m}{n}\nu}\ d\tau'\right|^{2}dx\nonumber\\
	&+\spaceI\psi W^{1+\alpha+m}\pjac\left|\int_{0}^{\tau}\mathbb{K}(m,\underline{n})\ d\tau'\right|^{2}dx\nonumber\\
	&+\spaceI\psi W^{1+\alpha+m}\pjac\left|\int_{0}^{\tau}\mathbb{K}(m,\underline{n})^{\intercal}\ d\tau'\right|^{2}dx.
	\end{align*}
	Due to $(\ref{Curl-with-A-tau-differentiated})$, we have the bound
	\begin{align*}
	\spaceI\psi W^{1+\alpha+m}\pjac\left|\int_{0}^{\tau}\mathbb{K}(m,\underline{n})\ d\tau'\right|^{2}dx\lesssim\spaceI\psi W^{1+\alpha+m}\pjac\left|\int_{0}^{\tau}\left|\ninv\right|\left|\nabla\nu\right| \left|\ngrad{\Ndell{m}{n}\theta}\right|d\tau'\right|^{2}dx.
	\end{align*}
	The $\nabla\nu$ term is dealt with as in $(\ref{Dnu-estimate})$, and we bound $\ninv$ in $L^{\infty}$ due to the a priori assumptions $(\ref{a-priori-assumptions})$. This leaves us with the bound
	\begin{align}
	\spaceI\psi W^{1+\alpha+m}\pjac\left|\int_{0}^{\tau}\mathbb{K}(m,\underline{n})\ d\tau'\right|^{2}dx\lesssim\delta H_{1}(\beta,\tau)\mathscr{S}_{N}(\tau)^{2},\label{K-m-n-estimate}
	\end{align}
	with $H_{1}$ defined in $(\ref{multiple-integrable-factors-H-1})$. The bound for $\mathbb{K}(m,\underline{n})^{\intercal}$ is completely analogous.
	
	It is left to bound the $\nCurl{\Ndell{m}{n}\nu}$ term. We have
	\begin{align*}
	&\spaceI\psi W^{1+\alpha+m}\pjac\left|\int_{0}^{\tau}\nCurl{\Ndell{m}{n}\nu}\ d\tau'\right|^{2}dx\\
	&\lesssim\left(\int_{0}^{\tau}e^{-\frac{\sigma_{1}}{2}\tau'}d\tau'\right)\left(\int_{0}^{\tau}e^{\frac{\sigma_{1}}{2}\tau'}\spaceI\psi W^{1+\alpha+m}\pjac\left|\nCurl{\Ndell{m}{n}\nu}\right|^{2}dxd\tau'\right)\\
	&\lesssim\int_{0}^{\tau}e^{\frac{\sigma_{1}}{2}\tau'}\left\|\nu\right\|_{\Y{N}{\nCurl}}^{2}d\tau'.
	\end{align*}
	Here we use Cauchy-Schwarz in $\tau$ on $\nCurl{\Ndell{m}{n}\nu}=e^{-\frac{\sigma_{1}}{4}\tau}\left(e^{\frac{\sigma_{1}}{4}\tau}\nCurl{\Ndell{m}{n}\nu}\right)$. Both $\psi$ and $W^{1+\alpha+m}$ are $\tau$ independent, and $\pjac\sim1$ due to the a priori assumptions in $(\ref{a-priori-assumptions})$, so all three terms can be absorbed in to the $\tau$ integral on the third line above. Then, using (\ref{main-Curl-nu-estimate}), we have
	\begin{align}
	\spaceI\psi W^{1+\alpha+m}\pjac\left|\int_{0}^{\tau}\nCurl{\Ndell{m}{n}\nu}\ d\tau'\right|^{2}dx&\lesssim\left(\mathscr{C}_{N}(0)+\delta \mathscr{S}_{N}(0)\right)\left(\int_{0}^{\tau}e^{-(2-\frac{\sigma_{1}}{2})\tau'}d\tau'\right)\nonumber\\
	&+\int_{0}^{\tau}e^{-\frac{\sigma_{1}}{2}\tau'}\left(\delta \mathscr{S}_{N}(\tau')+\delta \mathscr{S}_{N}(\tau')^{2}\right)d\tau'+\int_{0}^{\tau}\delta e^{(\sigma_{2}-\frac{\sigma_{1}}{2})\tau'}\mathscr{S}_{N}(\tau')^{2}d\tau'\nonumber\\
	&+\sum_{i=1}^{6}\int_{0}^{\tau}e^{\frac{\sigma_{1}}{2}\tau'}G_{i}(\beta,\tau')\delta \mathscr{S}_{N}(\tau')^{2}d\tau',\label{Curl-theta-FTC-curl-nu-estimate-1}
	\end{align}
	where we have bounded any extra powers of $\delta$ and $\mathscr{S}_{N}(\tau)$ by a constant.
	
	The definition of $\sigma_{1}(\beta)$ in $(\ref{sigma-1})$ means that $4-\sigma_{1}(\beta)\geq2$ for all $\beta\in(0,\infty)$. This implies that
	\begin{align}
	\int_{0}^{\tau}e^{-(2-\frac{\sigma_{1}}{2})\tau'}d\tau'\lesssim1.\label{Curl-theta-FTC-curl-nu-estimate-2}
	\end{align}
	Next, we use that $\mathscr{S}_{N}(\tau)$ is increasing to get
	\begin{align}
	&\int_{0}^{\tau}e^{-\frac{\sigma_{1}}{2}\tau'}\left(\delta \mathscr{S}_{N}(\tau')+\delta \mathscr{S}_{N}(\tau')^{2}\right)d\tau'+\int_{0}^{\tau}\delta e^{(\sigma_{2}-\frac{\sigma_{1}}{2})\tau'}\mathscr{S}_{N}(\tau')^{2}d\tau'+\sum_{i=1}^{6}\int_{0}^{\tau}e^{\frac{\sigma_{1}}{2}\tau'}G_{i}(\beta,\tau')\delta\mathscr{S}_{N}(\tau')^{2}d\tau'\nonumber\\
	&\lesssim\delta \mathscr{S}_{N}(\tau)+\delta \mathscr{S}_{N}(\tau)^{2}+\delta H_{2}(\beta,\tau)\mathscr{S}_{N}(\tau)^{2}+\sum_{i=1}^{6}\delta\tilde{G}_{i}(\beta,\tau)\mathscr{S}_{N}(\tau)^{2},\label{Curl-theta-FTC-curl-nu-estimate-3}
	\end{align}
	with $H_{2}$ defined in $(\ref{multiple-integrable-factors-H-2})$, and $\tilde{G}_{i}$, $i=1,\dots,6$, defined in $(\ref{multiple-integrable-factors-G-dash-i})$.
	
	Combining $(\ref{Curl-theta-FTC-curl-nu-estimate-1})-(\ref{Curl-theta-FTC-curl-nu-estimate-3})$ gives
	\begin{align*}
	\spaceI\psi W^{1+\alpha+m}\pjac\left|\int_{0}^{\tau}\nCurl{\Ndell{m}{n}\nu}\ d\tau'\right|^{2}dx&\lesssim\left(\mathscr{C}_{N}(0)+\delta \mathscr{S}_{N}(0)^{2}\right)+\delta \mathscr{S}_{N}(\tau)+\delta \mathscr{S}_{N}(\tau)^{2}\nonumber\\
	&+\delta H_{2}(\beta,\tau)\mathscr{S}_{N}(\tau)^{2}+\sum_{i=1}^{6}\delta\tilde{G}_{i}(\beta,\tau)\mathscr{S}_{N}(\tau)^{2}.
	\end{align*}
	Combine $(\ref{K-m-n-estimate})$ and $(\ref{Curl-theta-FTC-curl-nu-estimate-3})$, and sum over $(m,\underline{n})$. These, along with analogous estimates on $\supp{\bar{\psi}}$, give the result.
\end{proof}

We finish this section with a lemma that shows the functions $\tilde{G}_{i}$, $i=1,\dots,6$, $H_{1}$, and $H_{2}$ have sufficient pointwise boundedness, and integrability as functions of $\tau$ for our final estimates.

\begin{lemma}\label{G-dash-and-H-bounded-and-integrable-lemma} The functions $e^{-\sigma_{2}\tau}\tilde{G}_{i}(\beta,\tau)$, $i=1,\dots,6$, $e^{-\sigma_{2}\tau}H_{1}(\beta,\tau)$, and $e^{-\sigma_{2}\tau}H_{2}(\beta,\tau)$, when considered as functions of $\tau$, are bounded for all $\beta>0$, and integrable for all $\beta>2$.
\end{lemma}

\begin{proof}From $(\ref{multiple-integrable-factors-G-dash-i})$ we know that
	\begin{align*}
	\tilde{G}_{1}(\beta,\tau)=\threepartdef{\int_{0}^{\tau}\tau' e^{\left(\frac{\sigma_{1}}{2}-2\right)\tau'}d\tau'}{2\sigma_{1}(\beta)>2+\sigma_{2}(\beta)}{\int_{0}^{\tau}(\tau')^{2} e^{\left(\frac{\sigma_{1}}{2}-2\right)\tau'}d\tau'}{2\sigma_{1}(\beta)=2+\sigma_{2}(\beta)}{\int_{0}^{\tau}\tau' e^{\left(\sigma_{2}-\frac{3\sigma_{1}}{2}\right)\tau'}d\tau'}{2\sigma_{1}(\beta)<2+\sigma_{2}(\beta)}.
	\end{align*}
	For the case where $2\sigma_{1}(\beta)\geq2+\sigma_{2}(\beta)$, we once again recall the definition of $\sigma_{1}(\beta)$ and note that $4-\sigma_{1}(\beta)\geq2$ for all $\beta>0$. Hence on this region, $e^{-\sigma_{2}\tau}\tilde{G}_{1}(\beta,\tau)\lesssim e^{-\sigma_{2}\tau}$.
	
	For the case where $2\sigma_{1}(\beta)<2+\sigma_{2}(\beta)$, we have
	\begin{align*}
	e^{-\sigma_{2}\tau}\tilde{G}_{1}(\beta,\tau)\lesssim\threepartdef{e^{-\sigma_{2}\tau}}{3\sigma_{1}(\beta)>2\sigma_{2}(\beta)}{\tau^{2}e^{-\sigma_{2}\tau}}{3\sigma_{1}(\beta)=2\sigma_{2}(\beta)}{\tau e^{-\frac{3\sigma_{1}\tau}{2}}}{3\sigma_{1}(\beta)<2\sigma_{2}(\beta)},
	\end{align*}
	which satisfies pointwise boundedness for all $\beta>0$, and integrability for $\beta>2$. The arguments for $\tilde{G}_{2},\dots,\tilde{G}_{6}$, $H_{1}$, and $H_{2}$ are analogous.
\end{proof}

\section{Energy Estimates}\label{energy-estimates}

In this section we derive the natural higher order energy function coming from (\ref{rescaled-euler-W-perturbation-estimate-form-1}). While, in spirit, the process we go through for both the zero order case and higher order case is the same, there are technical differences. Therefore, for clarity, we separate the analysis.

To begin with, we define damping functionals which play a role in obtaining adequate energy estimates.

\begin{definition}\label{damping-functional-definition} Let $(\nu,\theta)$ be solutions to~\eqref{E:EULERNEW2} on $[0,T]$ in the sense of Theorem~\ref{local-well-posedness-theorem}. On $[0,T]$ define $\mathbb{D}^{\psi}_{(m,\underline{n})}$, the damping functional on $\supp{\psi}$, at level $(m,\underline{n})$ by
	\begin{align}
	\mathbb{D}^{\psi}_{(m,\underline{n})}(\tau)&=\frac{1}{\delta}\left(2-\sigma_{1}\right)\spaceI e^{\sigma_{1}\tau}\psi W^{\alpha+m}\left|\Ndell{m}{n}\nu\right|^{2}dx\nonumber\\
	&+\sigma_{2}\spaceI e^{-\sigma_{2}\tau}\psi W^{1+\alpha+m}\pjac\left(\left|\ngrad{\Ndell{m}{n}\theta}\right|^{2}+\frac{1}{\alpha}\left|\ndiv{\Ndell{m}{n}\theta}\right|^{2}\right)dx.\label{damping-functional-level-m-n}
	\end{align}
Define $\mathbb{D}^{\bar{\psi}}_{\underline{k}}$, the damping functional on $\supp{\bar{\psi}}$ at level $\underline{k}$ by
	\begin{align}
	\mathbb{D}^{\bar{\psi}}_{\underline{k}}(\tau)&=\frac{1}{\delta}\left(2-\sigma_{1}\right)\spaceI e^{\sigma_{1}\tau}\bar{\psi} W^{\alpha}\left|\nabla^{\underline{k}}\nu\right|^{2}dx\nonumber\\
	&+\sigma_{2}\spaceI e^{-\sigma_{2}\tau}\bar{\psi} W^{1+\alpha}\pjac\left(\left|\ngrad{\nabla^{\underline{k}}\theta}\right|^{2}+\frac{1}{\alpha}\left|\ndiv{\nabla^{\underline{k}}\theta}\right|^{2}\right)dx.\label{interior-functional-level-m-n}
	\end{align}
\end{definition}

\begin{remark}\label{damping-functional-is-nonnegative-remark}
	Note that due to definitions of $\sigma_{1}$ and $\sigma_{2}$ in $(\ref{sigma-1})-(\ref{sigma-2})$, the damping functionals are always non-negative, hence the nomenclature.
\end{remark}

\begin{remark}\label{damping-function-gamma-values-remark}
	Both $\mathbb{D}^{\psi}_{(m,\underline{n})}$ and $\mathbb{D}^{\bar{\psi}}_{\underline{k}}$ vanish when $\gamma=5/3$, or equivalently, when $\beta=2$. All we require is that these terms have the right sign; they are not otherwise used in our analysis.
\end{remark}

\subsection{Zero Order Energy Estimates}\label{zero-order-energy-section}

One of the main differences for the zero order case is the fact that we are not applying $\rdell$, $\angdell$ or $\dell$ derivatives to (\ref{rescaled-euler-W-perturbation-estimate-form-1}) before constructing an appropriate energy identity. Therefore, it is unnecessary to deal with the cases near the boundary and on the interior separately, as we do not use the cutoff functions $\psi$ and $\bar{\psi}$. Moreover, note that at zero order, the relevant damping term is given by $\mathbb{D}_{0}\coloneqq\mathbb{D}^{\psi}_{(0,\underline{0})}+\mathbb{D}^{\bar{\psi}}_{\underline{0}}$.

The energy estimates at zero order rely on the following energy identity.

\begin{theorem}[Zero Order Energy Identity]\label{zero-order-energy-identity-theorem} Let $(\nu,\theta)$ be solutions to~\eqref{E:EULERNEW2} on $[0,T]$ in the sense of Theorem~\ref{local-well-posedness-theorem}, for some $N\geq2\ceil*{\alpha}+12$. Suppose the a priori assumptions $(\ref{a-priori-assumptions})$ hold. Then for all $\tau\in[0,T]$,
	\begin{align}
	&\frac{d}{d\tau}\left(\frac{1}{2\delta}\spaceI e^{\sigma_{1}\tau} W^{\alpha}\left|\nu\right|^{2}dx+\spaceI e^{-\sigma_{2}\tau} W^{1+\alpha} \pjac\left(\frac{1}{2}|\ngrad{\theta}|^{2}+\frac{1}{2\alpha}|\ndiv{\theta}|^{2}\right)dx\right)+\frac{1}{2}\mathbb{D}_{0}\nonumber\\
	&-\frac{d}{d\tau}\left(\frac{1}{4}\spaceI e^{-\sigma_{2}\tau}W^{1+\alpha} \pjac|\nCurl{\theta}|^{2}dx\right)-\frac{\sigma_{2}}{4}\spaceI e^{-\sigma_{2}\tau}W^{1+\alpha} \pjac|\nCurl{\theta}|^{2}dx=\mathcal{R}(0),\label{zero-order-energy-identity-statement}
	\end{align}
	where
	\begin{align}
	\mathcal{R}(0)=&-(1+\alpha)\spaceI e^{-\sigma_{2}\tau}W^{\alpha}\dell_{i}(W)\nu^{i}dx-\spaceI e^{-\sigma_{2}\tau}W^{1+\alpha}\pjac\ninv^{k}_{j}\ninv^{l}_{p}[\nabla\theta]^{p}_{i}[\nabla\theta]^{j}_{l}\dell_{k}\nu^{i}dx\nonumber\\
	&+\spaceI e^{-\sigma_{2}\tau}W^{1+\alpha}\left(\pjac-1\right)\dell_{i}\nu^{i}dx-\frac{1}{2\alpha}\spaceI e^{-\sigma_{2}\tau} W^{1+\alpha}\njac^{-(1+\frac{1}{\alpha})}\dell_{\tau}\njac|\ngrad{\theta}|^{2}dx\nonumber\\
	&+\spaceI e^{-\sigma_{2}\tau} W^{1+\alpha}\pjac[\ngrad{\theta}]^{i}_{j}\dell_{\tau}\left(\ninv^{k}_{j}\right)\dell_{k}\theta^{i}dx-\frac{1}{4\alpha}\spaceI e^{-\sigma_{2}\tau} W^{1+\alpha}\njac^{-\left(1+\frac{1}{\alpha}\right)}\dell_{\tau}\njac\left|\nCurl{\theta}\right|^{2}dx\nonumber\\
	&+\sum_{i>j}\spaceI e^{-\sigma_{2}\tau} W^{1+\alpha}\pjac\left[\dell_{\tau}\ninv^{k}_{i}\dell_{k}\left(\theta^{j}\right)-\dell_{\tau}\ninv^{k}_{j}\dell_{k}\left(\theta^{i}\right)\right]\left[\nCurl{\theta}\right]^{j}_{i}dx\nonumber\\
	&-\frac{1}{2\alpha^{2}}\spaceI e^{-\sigma_{2}\tau}W^{1+\alpha}\njac^{-(1+\frac{1}{\alpha})}\dell_{\tau}\njac|\ndiv{\theta}|^{2}dx+\frac{1}{2\alpha}\spaceI e^{-\sigma_{2}\tau}W^{1+\alpha}\pjac\dell_{\tau}\left(|\ndiv{\theta}|^{2}\right)dx.\label{zero-order-remainder-statement}
	\end{align}
\end{theorem}

\begin{proof} Consider (\ref{rescaled-euler-W-perturbation-estimate-form-1}) in a modified form:
\begin{align}
    \frac{1}{\delta}e^{\sigma_{1}\tau}W^{\alpha}\left(\theta^{i}_{\tau\tau}+\theta^{i}_{\tau}\right)+e^{-\sigma_{2}\tau}\dell_{k}\left(W^{1+\alpha}\left(\ninv^{k}_{i}\pjac-\delta^{k}_{i}\right)\right)=-(1+\alpha)e^{-\sigma_{2}\tau}W^{\alpha}\dell_{i}W.\label{rescaled-euler-W-perturbation-estimate-form-1-modified-zero-order}
\end{align}
We multiply this expression by $\theta^{i}_{\tau}$, integrate over $\Omega$, and use integration by parts to 
\begin{align*}
    &\frac{1}{2}\frac{d}{d\tau}\left(\spaceI\frac{1}{\delta} e^{\sigma_{1}\tau}W^{\alpha}\left|\nu\right|^{2}dx\right)+\frac{1}{2}\left(2-\sigma_{1}\right)\spaceI\frac{1}{\delta} e^{\sigma_{1}\tau} W^{\alpha}\left|\nu\right|^{2}dx-\spaceI e^{-\sigma_{2}\tau} W^{1+\alpha}\left(\ninv^{k}_{i}\pjac-\delta^{k}_{i}\right)\dell_{k}\theta^{i}_{\tau}dx\nonumber\\
    &=-(1+\alpha)\spaceI e^{-\sigma_{2}\tau}W^{\alpha}\dell_{i}(W)\theta^{i}_{\tau}dx,
\end{align*}
where for the third term on the left hand side, we have used the divergence theorem, and noted that $W$ is $0$ on the boundary of the domain.

Now,
\begin{align*}
    \ninv^{k}_{i}\pjac-\delta^{k}_{i}=\left(\ninv^{k}_{i}-\delta^{k}_{i}\right)\pjac+\delta^{k}_{i}\left(\pjac-1\right).
\end{align*}
So we can write
\begin{align*}
    &\frac{1}{2}\frac{d}{d\tau}\left(\spaceI\frac{1}{\delta} e^{\sigma_{1}\tau}W^{\alpha}\left|\nu\right|^{2}dx\right)+\frac{1}{2}\left(2-\sigma_{1}\right)\spaceI\frac{1}{\delta}e^{\sigma_{1}\tau} W^{\alpha}\left|\nu\right|^{2}dx-\spaceI e^{-\sigma_{2}\tau}W^{1+\alpha} \left(\ninv^{k}_{i}-\delta^{k}_{i}\right)\pjac\dell_{k}\theta^{i}_{\tau}dx\nonumber\\
    &=-(1+\alpha)\spaceI e^{-\sigma_{2}\tau}W^{\alpha}\dell_{i}(W)\nu^{i}dx+\spaceI e^{-\sigma_{2}\tau}W^{1+\alpha}\left(\pjac-1\right)\dell_{i}\nu^{i}dx.
\end{align*}
Note that $\left(\ninv^{k}_{i}-\delta^{k}_{i}\right)=\ninv^{k}_{j}\left(\delta^{j}_{i}-\ncov^{j}_{i}\right)=\ninv^{k}_{j}\left([\nabla x]^{j}_{i}-[\nabla\zeta]^{j}_{i}\right)$.

Recall that $\theta=\zeta-x$, so
\begin{align*}
    \left(\ninv^{k}_{i}-\delta^{k}_{i}\right)=-\ninv^{k}_{j}[\nabla\theta]^{j}_{i}.
\end{align*}
We can further rewrite this as
\begin{align*}
    \left(\ninv^{k}_{i}-\delta^{k}_{i}\right)&=-\ninv^{k}_{j}[\nabla\theta]^{i}_{j}-\ninv^{k}_{j}[\Curl{\theta}]^{j}_{i}\nonumber\\
    &=-\ninv^{k}_{j}[\ngrad{\theta}]^{i}_{j}-\ninv^{k}_{j}\ninv^{l}_{p}[\nabla\theta]^{p}_{j}[\nabla\theta]^{i}_{l}-\ninv^{k}_{j}[\Curl{\theta}]^{j}_{i}.
\end{align*}
Next, note that
\begin{align*}
    [\Curl{\theta}]^{j}_{i}=[\nCurl{\theta}]^{j}_{i}+\ninv^{l}_{p}[\nabla\theta]^{p}_{i}[\nabla\theta]^{j}_{l}-\ninv^{l}_{p}[\nabla\theta]^{p}_{j}[\nabla\theta]^{i}_{l}.
\end{align*}
Hence we can finally write
\begin{align*}
     \left(\ninv^{k}_{i}-\delta^{k}_{i}\right)=-\ninv^{k}_{j}[\ngrad{\theta}]^{i}_{j}-\ninv^{k}_{j}[\nCurl{\theta}]^{j}_{i}-\ninv^{k}_{j}\ninv^{l}_{p}[\nabla\theta]^{p}_{i}[\nabla\theta]^{j}_{l}.
\end{align*}
So we obtain
\begin{align*}
    &\frac{1}{2}\frac{d}{d\tau}\left(\spaceI\frac{1}{\delta} e^{\sigma_{1}\tau}W^{\alpha}\left|\nu\right|^{2}dx\right)+\frac{1}{2}\left(2-\sigma_{1}\right)\spaceI\frac{1}{\delta}e^{\sigma_{1}\tau} W^{\alpha}\left|\nu\right|^{2}dx+\spaceI e^{-\sigma_{2}\tau}W^{1+\alpha} \pjac[\ngrad{\theta}]^{i}_{j}\ninv^{k}_{j}\dell_{k}\theta^{i}_{\tau}dx\\
    &+\spaceI e^{-\sigma_{2}\tau}W^{1+\alpha} \pjac[\nCurl{\theta}]^{j}_{i}\ninv^{k}_{j}\dell_{k}\theta^{i}_{\tau}dx=-(1+\alpha)\spaceI e^{-\sigma_{2}\tau}W^{\alpha}\dell_{i}(W)\nu^{i}dx\\
    &+\spaceI e^{-\sigma_{2}\tau}W^{1+\alpha}\left(\pjac-1\right)\dell_{i}\nu^{i}dx-\spaceI e^{-\sigma_{2}\tau}W^{1+\alpha}\pjac\ninv^{k}_{j}\ninv^{l}_{p}[\nabla\theta]^{p}_{i}[\nabla\theta]^{j}_{l}\dell_{k}\nu^{i}dx.
\end{align*}
Now, note that
\begin{align*}
    e^{-\sigma_{2}\tau}W^{1+\alpha} \pjac[\ngrad{\theta}]^{i}_{j}\ninv^{k}_{j}\dell_{k}\theta^{i}_{\tau}&=\frac{1}{2}\dell_{\tau}\left(e^{-\sigma_{2}\tau} W^{1+\alpha}\pjac|\ngrad{\theta}|^{2}\right)+\frac{\sigma_{2}}{2}e^{-\sigma_{2}\tau}W^{1+\alpha}\pjac|\ngrad{\theta}|^{2}\nonumber\\
    &+\frac{1}{2\alpha}e^{-\sigma_{2}\tau} W^{1+\alpha}\njac^{-(1+\frac{1}{\alpha})}\dell_{\tau}\njac|\ngrad{\theta}|^{2}-e^{-\sigma_{2}\tau} W^{1+\alpha}\pjac[\ngrad{\theta}]^{i}_{j}\dell_{\tau}\left(\ninv^{k}_{j}\right)\dell_{k}\theta^{i}.
\end{align*}
For the $\nCurl$ term we use an antisymmetrisation argument that gives
\begin{align*}
    &e^{-\sigma_{2}\tau} W^{1+\alpha}\pjac\ninv^{k}_{j}\dell_{k}\theta^{i}_{\tau}\left[\nCurl{\theta}\right]^{j}_{i}=-\frac{1}{4}\dell_{\tau}\left(e^{-\sigma_{2}\tau} W^{1+\alpha}\pjac\left|\nCurl{\theta}\right|^{2}\right)-\frac{\sigma_{2}}{4}e^{-\sigma_{2}\tau}W^{1+\alpha}\pjac\left|\nCurl{\theta}\right|^{2}\nonumber\\
    &-\frac{1}{4\alpha}e^{-\sigma_{2}\tau}W^{1+\alpha}\njac^{-\left(1+\frac{1}{\alpha}\right)}\dell_{\tau}\njac\left|\nCurl{\theta}\right|^{2}+\sum_{i>j}e^{-\sigma_{2}\tau} W^{1+\alpha}\pjac\left[\dell_{\tau}\ninv^{k}_{i}\dell_{k}\left(\theta^{j}\right)-\dell_{\tau}\ninv^{k}_{j}\dell_{k}\left(\theta^{i}\right)\right]\left[\nCurl{\theta}\right]^{j}_{i}.
\end{align*}
We also insert corresponding $\ndiv$ terms, as all higher order energy terms include such expressions:
\begin{align*}
    &\frac{1}{2\alpha}\dell_{\tau}\left(e^{-\sigma_{2}\tau} W^{1+\alpha}\pjac|\ndiv{\theta}|^{2}\right)+\frac{\sigma_{2}}{2\alpha}e^{-\sigma_{2}\tau} W^{1+\alpha}\pjac|\ndiv{\theta}|^{2}\\
    =&-\frac{1}{2\alpha^{2}}e^{-\sigma_{2}\tau} W^{1+\alpha}\njac^{-(1+\frac{1}{\alpha})}\dell_{\tau}\njac|\ndiv{\theta}|^{2}+\frac{1}{2\alpha}e^{-\sigma_{2}\tau} W^{1+\alpha}\pjac\dell_{\tau}\left(|\ndiv{\theta}|^{2}\right).
\end{align*}
Taking all of this in to account gives us
\begin{align}
    &\frac{d}{d\tau}\left(\frac{1}{2\delta}\spaceI e^{\sigma_{1}\tau}W^{\alpha}\left|\nu\right|^{2}dx+\spaceI e^{-\sigma_{2}\tau} W^{1+\alpha} \pjac\left(\frac{1}{2}|\ngrad{\theta}|^{2}+\frac{1}{2\alpha}|\ndiv{\theta}|^{2}\right)dx\right)+\frac{1}{2}\mathbb{D}_{0}\nonumber\\
    &-\frac{d}{d\tau}\left(\frac{1}{4}\spaceI e^{-\sigma_{2}\tau}W^{1+\alpha} \pjac|\nCurl{\theta}|^{2}dx\right)-\frac{\sigma_{2}}{4}\spaceI W^{1+\alpha} \pjac|\nCurl{\theta}|^{2}dx=\mathcal{R}(0),\label{zero-order-energy}
\end{align}
where $\mathbb{D}_{0}=\mathbb{D}^{\psi}_{(0,\underline{0})}+\mathbb{D}^{\bar{\psi}}_{\underline{0}}$, and the remainder $\mathcal{R}(0)$ is given in $(\ref{zero-order-remainder-statement})$, as required.
\end{proof}

Now we move on to the estimates.

\begin{theorem}[Zero Order Energy Estimates]\label{zero-order-energy-estimates-theorem} Let $(\nu,\theta)$ be the solution to~\eqref{E:EULERNEW2} on $[0,T]$ in the sense of Theorem~\ref{local-well-posedness-theorem}, for some $N\geq2\ceil*{\alpha}+12$. Suppose the a priori assumptions $(\ref{a-priori-assumptions})$ hold, and that $\nabla W\in\X{N}$. Then for all $0\leq\tau_{1}\leq\tau\leq T$, we have
	\begin{align}
	\left|\int_{\tau_{1}}^{\tau}\rem(0)d\tau'\right|\lesssim\int_{\tau_{1}}^{\tau}\sqrt{\delta}e^{-\frac{\sigma_{1}}{2}\tau'}\mathscr{S}_{N}(\tau')^{1/2}d\tau'+\int_{\tau_{1}}^{\tau}e^{-\frac{\sigma_{1}}{2}\tau'}\left(\mathscr{S}_{N}(\tau')+\mathscr{S}_{N}(\tau')^{3/2}\right)d\tau'.\label{zero-order-energy-estimates}
	\end{align}
\end{theorem}

\begin{proof}
	Let us first look at the quantity
	\begin{align*}
	\left|\int_{\tau_{1}}^{\tau}\spaceI e^{-\sigma_{2}\tau'}W^{\alpha}\dell_{i}(W)\nu^{i}dxd\tau'\right|.
	\end{align*}
	Using Cauchy-Schwarz, we have
	\begin{align*}
	\left|\int_{\tau_{1}}^{\tau}\spaceI e^{-\sigma_{2}\tau'}W^{\alpha}\dell_{i}(W)\nu^{i}dxd\tau'\right|\lesssim \int_{\tau_{1}}^{\tau}e^{-\sigma_{2}\tau'}\left\|W^{\frac{\alpha}{2}}\nabla W\right\|_{L^{2}(\Omega)}\left\|W^{\frac{\alpha}{2}}\nu\right\|_{L^{2}(\Omega)}d\tau'.
	\end{align*}
	To bound the $\nabla W$ term, we have
	\begin{align}
	\left\|W^{\frac{\alpha}{2}}\nabla W\right\|_{L^{2}(\Omega)}^{2}&=\spaceI\psi W^{\alpha}\left|\nabla W\right|^{2}dx+\spaceI\bar{\psi}W^{\alpha}\left|\nabla W\right|^{2}dx\nonumber\\
	&\leq\left\|\nabla W\right\|_{\X{N}}^{2}.\label{L-2-nabla-W-estimate}
	\end{align}
	The first equality is since $\psi+\bar{\psi}=1$, and the last inequality is by definition of the space $\X{N}$. Since $\nabla W\in\X{N}$, this quantity is bounded. Similarly,
	\begin{align*}
	\left\|W^{\frac{\alpha}{2}}\nu\right\|_{L^{2}(\supp{\psi})}^{2}\leq\left\|\nu\right\|_{\X{N}}^{2}\leq\delta e^{-\sigma_{1}\tau}\mathscr{S}_{N}(\tau).
	\end{align*}
	These bounds give
	\begin{align}
	\int_{\tau_{1}}^{\tau}e^{-\sigma_{2}\tau'}\left\|W^{\frac{\alpha}{2}}\nabla W\right\|_{0}\left\|W^{\frac{\alpha}{2}}\nu\right\|_{0}d\tau'\lesssim \int_{\tau_{1}}^{\tau}\sqrt{\delta}e^{-\frac{1}{2}(2\sigma_{2}+\sigma_{1})\tau'}\mathscr{S}_{N}(\tau')^{1/2}d\tau'.\label{zero-order-energy-estimate-1-source-term}
	\end{align}
	Next we have
	\begin{align*}
	\left|\int_{\tau_{1}}^{\tau}\spaceI e^{-\sigma_{2}\tau'}W^{1+\alpha}\pjac\ninv^{k}_{j}\ninv^{l}_{p}[\nabla\theta]^{p}_{i}[\nabla\theta]^{j}_{l}\dell_{k}\nu^{i}dxd\tau'\right|,
	\end{align*}
	which is controlled by the quantity
	\begin{align}
	\int_{\tau_{1}}^{\tau}\spaceI e^{-\sigma_{2}\tau'}W^{1+\alpha}\pjac\left|\ninv\right|^{2}\left|\nabla\theta\right|^{2}\left|\nabla\nu\right|dxd\tau'.\label{zero-order-nabla-theta-squared-nabla-nu-integral}
	\end{align}
	This can be written as
	\begin{align}
	\int_{\tau_{1}}^{\tau}\spaceI\psi e^{-\sigma_{2}\tau'}W^{1+\alpha}\pjac\left|\ninv\right|^{2}\left|\nabla\theta\right|^{2}\left|\nabla\nu\right|dxd\tau'+\int_{\tau_{1}}^{\tau}\spaceI\bar{\psi} e^{-\sigma_{2}\tau'}W^{1+\alpha}\pjac\left|\ninv\right|^{2}\left|\nabla\theta\right|^{2}\left|\nabla\nu\right|dxd\tau'.\label{zero-order-integral-split-into-supp-psi-supp-psi-bar}
	\end{align}
	The $\ninv$ and $\pjac$ terms can be bounded in $L^{\infty}$ using the a priori assumptions in $(\ref{a-priori-assumptions})$. For the right hand side integral above, we bound one of the copies of $\nabla\theta$ in $L^{\infty}$. To do this, we first employ Sobolev embedding $H^2(\supp{\bar{\psi}})\hookrightarrow L^\infty(\supp{\bar{\psi}})$ to obtain
	\begin{align}
	\left\|\nabla\theta\right\|_{L^{\infty}(\supp{\bar{\psi})}}\lesssim \left\|\nabla\theta\right\|_{H^{2}(\supp{\bar{\psi})}}\lesssim\left\|\theta\right\|_{H^{3}(\supp{\bar{\psi}})}.\label{nabla-theta-L-infinity-H2-embedding-1}
	\end{align}
	Then we have
	\begin{align}
	\left\|\theta\right\|_{H^{3}(\supp{\bar{\psi}})}^{2}=\sum_{|\underline{k}|=0}^{3}\int_{\supp{\bar{\psi}}}\left|\nabla^{\underline{k}}\theta\right|^{2}dx&=\sum_{|\underline{k}|=0}^{3}\int_{\supp{\bar{\psi}}}\psi\left|\nabla^{\underline{k}}\theta\right|^{2}dx+\sum_{|\underline{k}|=0}^{3}\int_{\supp{\bar{\psi}}}\bar{\psi}\left|\nabla^{\underline{k}}\theta\right|^{2}dx\nonumber\\
	&=\sum_{|\underline{k}|=0}^{3}\int_{\supp{\bar{\psi}}\cap\supp{\psi}}\psi\left|\nabla^{\underline{k}}\theta\right|^{2}dx+\sum_{|\underline{k}|=0}^{3}\int_{\supp{\bar{\psi}}}\bar{\psi}\left|\nabla^{\underline{k}}\theta\right|^{2}dx.\label{nabla-theta-L-infinity-H2-embedding-2}
	\end{align}
	Note that $\supp{\bar{\psi}}\cap\supp{\psi}$ is removed from both the origin, and the vacuum boundary. On this region, we can use Lemma \ref{Dk-to-Ndell-lemma} to write 
	\begin{align*}
	\nabla^{\underline{k}}=\sum_{m+|\underline{n}|=0}^{|\underline{k}|}\mathcal{Z}_{(m,\underline{n})}\Ndell{m}{n},
	\end{align*}
	for some smooth functions $\mathcal{Z}_{(m,\underline{n})}$. Also on this region, as well as on $\supp{\bar{\psi}}$, $W\sim1$, so we can adjust the powers of $W$ in the integrals with freedom. Hence we have the bound
	\begin{align}
	\left\|\theta\right\|_{H^{3}(\supp{\bar{\psi}})}^{2}&\lesssim \sum_{m+|\underline{n}|=0}^{3}\int_{\supp{\bar{\psi}}\cap\supp{\psi}}\psi W^{\alpha+m}\left|\Ndell{m}{n}\theta\right|^{2}dx+\sum_{|\underline{k}|=0}^{3}\int_{\supp{\bar{\psi}}}\bar{\psi}W^{\alpha}\left|\nabla^{\underline{k}}\theta\right|^{2}dx\nonumber\\
	&\lesssim \sum_{m+|\underline{n}|=0}^{3}\spaceI\psi W^{\alpha+m}\left|\Ndell{m}{n}\theta\right|^{2}dx+\sum_{|\underline{k}|=0}^{3}\spaceI\bar{\psi}W^{\alpha}\left|\nabla^{\underline{k}}\theta\right|^{2}dx\nonumber\\
	&\lesssim\left\|\theta\right\|_{\X{N}}^{2}.\label{nabla-theta-L-infinity-H2-embedding-3}
	\end{align}
	Moreover,
	\begin{align*}
	\left(\spaceI\bar{\psi}W^{1+\alpha}\left|\nabla\theta\right|^{2}dx\right)^{1/2}\left(\spaceI\bar{\psi}W^{1+\alpha}\left|\nabla\nu\right|^{2}dx\right)^{1/2}\lesssim\left\|\theta\right\|_{\X{N}}\left\|\nu\right\|_{\X{N}},
	\end{align*}
	by definition of the space $\X{N}$, and since $W\sim1$ on $\supp{\bar{\psi}}$. Therefore,
	\begin{align}
	\int_{\tau_{1}}^{\tau}\spaceI\bar{\psi} e^{-\sigma_{2}\tau'}W^{1+\alpha}\pjac\left|\ninv\right|^{2}\left|\nabla\theta\right|^{2}\left|\nabla\nu\right|dxd\tau'&\lesssim \int_{\tau_{1}}^{\tau}e^{-\sigma_{2}\tau'}\left\|\theta\right\|_{\X{N}}\spaceI\bar{\psi}W^{1+\alpha}\left|\nabla\theta\right|\left|\nabla\nu\right|dx\nonumber\\
	&\lesssim\int_{\tau_{1}}^{\tau}\delta e^{-\frac{1}{2}(2\sigma_{2}+\sigma_{1})\tau'}\mathscr{S}_{N}(\tau')^{3/2}d\tau',\label{pre-zero-order-energy-estimate-2}
	\end{align}
	where the last inequality comes from Cauchy Schwartz. The corresponding integral on $\supp{\psi}$ in $(\ref{zero-order-integral-split-into-supp-psi-supp-psi-bar})$ is estimated analogously, but on this region we utilise $(\ref{rectangular-as-ang-rad})$ to transform $\nabla$ in to radial and angular derivatives. Finally, we get
	\begin{align}
	\left|\int_{\tau_{1}}^{\tau}\spaceI e^{-\sigma_{2}\tau'}W^{1+\alpha}\pjac\ninv^{k}_{j}\ninv^{l}_{p}[\nabla\theta]^{p}_{i}[\nabla\theta]^{j}_{l}\dell_{k}\nu^{i}dxd\tau'\right|\lesssim\int_{\tau_{1}}^{\tau}\delta e^{-\frac{1}{2}(2\sigma_{2}+\sigma_{1})\tau'}\mathscr{S}_{N}(\tau')^{3/2}d\tau'.\label{zero-order-energy-estimate-2}
	\end{align}
	To bound
	\begin{align*}
	\int_{\tau_{1}}^{\tau}\spaceI e^{-\sigma_{2}\tau'}W^{1+\alpha}\left(\pjac-1\right)\dell_{i}\nu^{i}dxd\tau'
	\end{align*}
	we write $\pjac-1$ as
	\begin{align*}
	\pjac-1=\frac{1}{\alpha}\tr{[\nabla\theta]}+O(\left|\nabla\theta\right|^{2}),
	\end{align*}
	which follows from expanding to first order the following identity
	\begin{align*}
	\pjac=\left(\det{[\nabla\zeta]}\right)^{-1/\alpha}=\left(\det{[I+\nabla\theta]}\right)^{-1/\alpha}.
	\end{align*}
	From this we deduce
	\begin{align}
	\left|\int_{\tau_{1}}^{\tau}\spaceI e^{-\sigma_{2}\tau'}W^{1+\alpha}\left(\pjac-1\right)\dell_{i}\nu^{i}dxd\tau'\right|\lesssim \int_{\tau_{1}}^{\tau}\delta e^{-\frac{1}{2}(2\sigma_{2}+\sigma_{1})\tau'}\mathscr{S}_{N}(\tau')d\tau'.\label{zero-order-energy-estimate-3}
	\end{align}
	Next we look to bound the quantity
	\begin{align}
	&-\frac{1}{2\alpha}\int_{\tau_{1}}^{\tau}\spaceI e^{-\sigma_{2}\tau'} W^{1+\alpha}\njac^{-(1+\frac{1}{\alpha})}\dell_{\tau}\njac|\ngrad{\theta}|^{2}dxd\tau'-\frac{1}{4\alpha}\int_{\tau_{1}}^{\tau}\spaceI e^{-\sigma_{2}\tau'} W^{1+\alpha}\njac^{-\left(1+\frac{1}{\alpha}\right)}\dell_{\tau}\njac\left|\nCurl{\theta}\right|^{2}dxd\tau'\nonumber\\
	&-\frac{1}{2\alpha^{2}}\int_{\tau_{1}}^{\tau}\spaceI e^{-\sigma_{2}\tau'}W^{1+\alpha}\njac^{-(1+\frac{1}{\alpha})}\dell_{\tau}\njac|\ndiv{\theta}|^{2}dxd\tau'.\label{zero-order-dell-tau-jacobian-integrals}
	\end{align}
	Now $\dell_{\tau}\njac=\njac\ninv^{s}_{r}\dell_{s}^{r}$, so we can control the first integral above by
	\begin{align*}
	\int_{\tau_{1}}^{\tau}\spaceI e^{-\sigma_{2}\tau'} W^{1+\alpha}\pjac\left|\ninv\right|\left|\nabla\nu\right||\ngrad{\theta}|^{2}dxd\tau'
	\end{align*}
	The $\ninv$ term can be bounded in $L^{\infty}$ using the a priori assumptions in $(\ref{a-priori-assumptions})$. For $\nabla\nu$, we have
	\begin{align}
	\left\|\nabla\nu\right\|_{L^{\infty}(\Omega)}\lesssim\left\|\nabla\nu\right\|_{L^{\infty}(\supp{\psi})}+\left\|\nabla\nu\right\|_{L^{\infty}(\supp{\bar{\psi}})}.\label{zero-order-Dnu-estimate-L-infinity-1}
	\end{align}
	On the $\supp{\psi}$ term on the right hand side, we can use $(\ref{rectangular-as-ang-rad})$ to turn $\nabla$ in to radial and angular derivatives, and apply Lemma $\ref{L-infinity-energy-space-bound-no-weights}$. On the $\supp{\bar{\psi}}$ term, the bound is analogous to the one shown for $\nabla\theta$ in $(\ref{nabla-theta-L-infinity-H2-embedding-1})-(\ref{nabla-theta-L-infinity-H2-embedding-3})$. So we have the bound
	\begin{align}
	\left\|\nabla\nu\right\|_{L^{\infty}(\Omega)}^{2}\lesssim\delta e^{-\sigma_{1}\tau}\mathscr{S}_{N}(\tau).\label{zero-order-Dnu-estimate-L-infinity-2}
	\end{align}
	This gives
	\begin{align}
	\int_{\tau_{1}}^{\tau}\spaceI e^{-\sigma_{2}\tau} W^{1+\alpha}\pjac\left|\ninv\right|\left|\nabla\nu\right||\ngrad{\theta}|^{2}dx&\lesssim \int_{\tau_{1}}^{\tau}\sqrt{\delta}e^{-\frac{\sigma_{1}}{2}\tau'}\mathscr{S}_{N}(\tau')^{1/2}\spaceI e^{-\sigma_{2}\tau'}W^{1+\alpha}\pjac\left|\ngrad{\theta}\right|^{2}dxd\tau'\nonumber\\
	&\lesssim\int_{\tau_{1}}^{\tau}\sqrt{\delta}e^{-\frac{\sigma_{1}}{2}\tau'}\mathscr{S}_{N}(\tau')^{3/2}d\tau',\label{zero-order-energy-estimate-4}
	\end{align}
	as
	\begin{align*}
	\spaceI e^{-\sigma_{2}\tau'}W^{1+\alpha}\pjac\left|\ngrad{\theta}\right|^{2}dx\lesssim e^{-\sigma_{2}\tau'}\left\|\theta\right\|_{\Y{N}{\ngrad}}\lesssim \mathscr{S}_{N}(\tau').
	\end{align*}
	The remaining two integrals in $(\ref{zero-order-dell-tau-jacobian-integrals})$ have the same bounds.
	
	Finally we have to bound
	\begin{align}
	&\int_{\tau_{1}}^{\tau}\spaceI e^{-\sigma_{2}\tau'} W^{1+\alpha}\pjac[\ngrad{\theta}]^{i}_{j}\dell_{\tau}\left(\ninv^{k}_{j}\right)\dell_{k}\theta^{i}dxd\tau'+\frac{1}{2\alpha}\int_{\tau_{1}}^{\tau}\spaceI e^{-\sigma_{2}\tau'}W^{1+\alpha}\pjac\dell_{\tau}\left(|\ndiv{\theta}|^{2}\right)dxd\tau'\nonumber\\
	&+\sum_{i>j}\int_{\tau_{1}}^{\tau}\spaceI e^{-\sigma_{2}\tau'} W^{1+\alpha}\pjac\left[\dell_{\tau}\ninv^{k}_{i}\dell_{k}\left(\theta^{j}\right)-\dell_{\tau}\ninv^{k}_{j}\dell_{k}\left(\theta^{i}\right)\right]\left[\nCurl{\theta}\right]^{j}_{i}dxd\tau'.\label{zero-order-dell-tau-grad-div-curl-integrals}
	\end{align}
	As $\dell_{\tau}\ninv^{k}_{j}=-\ninv^{k}_{r}\ninv^{s}_{j}\dell_{s}\nu^{r}$, and $\left[\ngrad{\theta}\right]^{i}_{j}=\ninv^{l}_{j}\dell_{l}\theta^{i}$, the first integral above can be controlled by
	\begin{align*}
	\int_{\tau_{1}}^{\tau}\spaceI e^{-\sigma_{2}\tau'}W^{1+\alpha}\pjac\left|\ninv\right|^{3}\left|\nabla\theta\right|^{2}\left|\nabla\nu\right|dxd\tau'.
	\end{align*}
	This can be bounded analogously to $(\ref{zero-order-nabla-theta-squared-nabla-nu-integral})$, so
	\begin{align}
	\left|\int_{\tau_{1}}^{\tau}\spaceI e^{-\sigma_{2}\tau'} W^{1+\alpha}\pjac[\ngrad{\theta}]^{i}_{j}\dell_{\tau}\left(\ninv^{k}_{j}\right)\dell_{k}\theta^{i}dxd\tau'\right|\lesssim\int_{\tau_{1}}^{\tau}\delta e^{-\frac{1}{2}(2\sigma_{2}+\sigma_{1})\tau'}\mathscr{S}_{N}(\tau')^{3/2}d\tau'.\label{zero-order-energy-estimate-5}
	\end{align}
	The other two integrals in $(\ref{zero-order-dell-tau-grad-div-curl-integrals})$ have the same bound.
	
	Combining bounds $(\ref{zero-order-energy-estimate-1-source-term})$, $(\ref{zero-order-energy-estimate-2})$, $(\ref{zero-order-energy-estimate-3})$, $(\ref{zero-order-energy-estimate-4})$, and $(\ref{zero-order-energy-estimate-5})$ gives the result once we bound $e^{-\frac{1}{2}(2\sigma_{2}+\sigma_{1})\tau}$ by $e^{-\frac{\sigma_{1}}{2}\tau}$ wherever necessary.
\end{proof}

\begin{remark}\label{zero-order-energy-estimate-source-term-remark}
	The smallness of $\delta$ plays an important role in $(\ref{zero-order-energy-estimate-1-source-term})$; $\mathscr{S}_{N}(\tau)^{1/2}$ under the $\tau$ integral is not enough on its own to let us close the estimates, but the extra smallness coming from $\sqrt{\delta}$ on this term counteracts this. This is an example of a \textbf{source term} estimate, where all spatial derivatives fall on $W$, and we will need the same argument at higher orders. These source terms are not present in $\cite{HaJa2}$ or $\cite{ShSi2017}$.
\end{remark}

\subsection{Higher Order Energy Estimates}\label{higher-order-energy-section}

Now we move on to higher order terms. First we deal with estimates near the boundary. In this case, it is important to derive a higher order energy identity with remainder terms that have enough powers of the weight $W$ to estimate them in our energy spaces. This is a delicate issue and requires taking advantage of cancellation inherent in the identities $(\ref{commutator-ang-rad})-(\ref{rectangular-as-ang-rad})$. We do this via the following lemma.

\begin{lemma}{\label{commutators-instead-of-expansion}}
Let $Q$ be a rank $2$ tensor. Let $q>0$. Then for $i,j,l=1,2,3$,
\begin{align*}
    &\rdell\left(\frac{1}{W^{q}}\dell_{k}\left(W^{1+q}Q^{k}_{i}\right)\right)=\frac{1}{W^{1+q}}\dell_{k}\left(W^{2+q}\rdell Q^{k}_{i}\right)+\mathcal{C}^{q+1}_{i}(Q),\\
    &\angdell_{jl}\left(\frac{1}{W^{q}}\dell_{k}\left(W^{1+q}Q^{k}_{i}\right)\right)=\frac{1}{W^{q}}\dell_{k}\left(W^{1+q}\angdell_{jl}Q^{k}_{i}\right)+\mathcal{D}^{q}_{i}(j,l,Q),
\end{align*}
where
\begin{align*}
    \mathcal{C}^{q+1}_{i}(Q)=&\dell_{r}W \frac{x_{l}}{r}\angdell_{lk}Q^{k}_{i}-W\dell_{k}Q^{k}_{i}+(1+q)Q^{k}_{i}\rdell\dell_{k}W,
\end{align*}
and
\begin{align*}
    \mathcal{D}^{q}_{i}(j,l,Q)=W\comm{\angdell_{jl}}{\dell_{k}}Q^{k}_{i}+(1+q)Q^{k}_{i}\angdell_{jl}\dell_{k}W+\angdell_{jl}W\dell_{k}Q^{k}_{i}.
\end{align*}
\end{lemma}

For a proof of this result, we refer to $\cite{HaJa2}$.

\begin{remark}\label{cancellation-remark}
	The key part of this result is the fact that $\rdell W\dell_{k}Q^{k}_{i}-\dell_{k}W\rdell Q^{k}_{i}$ produces an angular derivative of $Q$. On their own, both of these terms include a radial derivative of $Q$, which requires an extra power of $W$ to bound in the correct energy space. However, the cancellation lowers this requirement, and in essence, lets us close our estimates.
\end{remark}

Next we have a lemma that shows us how $\pjac\ninv$ linearises. For a proof, we once again refer to $\cite{HaJa2}$.

\begin{lemma}\label{pjac-ninv-linearisation-lemma}
	Let $\mathscr{D}\in\left\{\dell_{i},\rdell,\angdell_{jk}\right\}$ for $i,j,k=1,2,3$. Then
	\begin{align}
	\mathscr{D}\left(\pjac\ninv^{k}_{i}\right)&=-\pjac\ninv^{k}_{j}\left[\ngrad{\mathscr{D}\theta}\right]^{i}_{j}-\frac{1}{\alpha}\pjac\ninv^{k}_{i}\ndiv{\mathscr{D}\theta}-\pjac\ninv^{k}_{j}\left[\nCurl{\mathscr{D}\theta}\right]^{j}_{i}\nonumber\\
	&-\left(\pjac\ninv^{k}_{j}\ninv^{s}_{i}\comm{\mathscr{D}}{\dell_{s}}\theta^{j}+\frac{1}{\alpha}\pjac\ninv^{k}_{i}\ninv^{s}_{j}\comm{\mathscr{D}}{\dell_{s}}\theta^{j}\right).\label{pjac-ninv-linearisation-statement}
	\end{align}
\end{lemma}


Finally, before we prove our higher order energy identity, we make some conventional choices, and define the remainder terms that we will eventually have to estimate. This is done for the purposes of streamlining exposition.

First of all, to obtain a higher order energy identity on $\supp{\psi}$, we need to act on (\ref{rescaled-euler-W-perturbation-estimate-form-1}) with a differential operator of the form $\Ndell{m}{n}$. Note that written out fully,
\begin{align*}
\Ndell{m}{n}=\rdell^{m}\angdell_{12}^{n_{1}}\angdell_{13}^{n_{2}}\angdell_{23}^{n_{3}}.
\end{align*}
To lessen the burden of index tracking when manipulating $\angdell^{\underline{n}}$, say during integration by parts, or expanding a term of the form $\angdell^{\underline{n}}(A\cdot B)$ via the product rule, we prove the higher order energy identity for a differential operator of the form
\begin{align*}
\Ndell{m}{n}=\rdell^{m}\angdell_{12}^{n},
\end{align*}
that is, $\underline{n}=(n,0,0)$. We also assume $n\geq1$. The case for general $(m,\underline{n})$ is addressed in Remarks \ref{general-higher-order-energy-estimates-remark} and \ref{general-higher-order-energy-identity-remark}.

\begin{definition}[Higher Order Remainder Terms]\label{higher-order-remainder-terms-definition}
	Let $(\nu,\theta)$ be the solution to~\eqref{E:EULERNEW2} on $[0,T]$ in the sense of Theorem~\ref{local-well-posedness-theorem}, for some $N\geq2\ceil*{\alpha}+12$. Fix $(m,\underline{n},\underline{k})\in\mathbb{Z}_{\geq0}\times\mathbb{Z}_{\geq0}^{3}\times\mathbb{Z}_{\geq0}^{3}$, with $\underline{n}=(n,0,0)$, $n\geq1$, and $\max{(m+n,|\underline{k}|)}\leq N$. For all $\tau\in[0,T]$, define, for $i,k=1,2,3$,
	\begin{align}
	\rem^{i}_{1}(m,\underline{n})=-e^{-\sigma_{2}\tau}W^{\alpha+m}\left(\sum_{l+p=m-1}\rdell^{l}\mathcal{C}^{1+\alpha+p}_{i}\left(\Ndellonetwo{p}{n}\left(\pjac\ninv\right)\right)+\sum_{a+b=n-1}\Ndellonetwo{m}{a}\mathcal{D}^{\alpha}_{i}\left(1,2,\angdell_{12}^{b}\left(\pjac\ninv\right)\right)\right)\label{rem-i-1},
	\end{align}
	\begin{align}
	\rem^{ki}_{2}(m,\underline{n})=&-\Ndellonetwo{m}{n-1}\left(\pjac\ninv^{k}_{j}\ninv^{s}_{i}\comm{\angdell_{12}}{\dell_{s}}\theta^{j}+\frac{1}{\alpha}\pjac\ninv^{k}_{i}\ninv^{s}_{j}\comm{\angdell_{12}}{\dell_{s}}\theta^{j}\right)\nonumber\\
	&-\pjac\ninv^{k}_{j}\ninv^{s}_{i}\comm{\Ndellonetwo{m}{n-1}}{\dell_{s}}\angdell_{12}\theta^{j}-\frac{1}{\alpha}\pjac\ninv^{k}_{i}\ninv^{s}_{j}\comm{\Ndellonetwo{m}{n-1}}{\dell_{s}}\angdell_{12}\theta^{j}\nonumber\\
	&-\sum_{\substack{p+l=m\\q+b=n-1\\p+q>0}}\mathscr{L}(p,q,l,b)\Ndellonetwo{p}{q}\left(\pjac\ninv^{k}_{j}\ninv^{s}_{i}\right)\Ndellonetwo{l}{b}\left(\dell_{s}\angdell_{12}\theta^{j}\right)\nonumber\\
	&-\frac{1}{\alpha}\sum_{\substack{p+l=m\\q+b=n-1\\p+q>0}}\mathscr{L}(p,q,l,b)\Ndellonetwo{p}{q}\left(\pjac\ninv^{k}_{i}\ninv^{s}_{j}\right)\Ndellonetwo{l}{b}\left(\dell_{s}\angdell_{12}\theta^{j}\right),\label{rem-k-i-2}
	\end{align}	
	\begin{align}
	\rem^{i}(m,\underline{n})&=\psi\rem^{i}_{1}(m,\underline{n})+e^{-\sigma_{2}\tau}\psi\dell_{k}\left(W^{1+\alpha+m}\rem^{ki}_{2}(m,\underline{n})\right)+e^{-\sigma_{2}\tau}W^{1+\alpha+m}\dell_{k}\psi\pjac\ninv^{k}_{j}\left[\ngrad{\Ndellonetwo{m}{n}\theta}\right]^{i}_{j}\nonumber\\
	&+e^{-\sigma_{2}\tau}W^{1+\alpha+m}\dell_{k}\psi\left(\pjac\ninv^{k}_{i}\ndiv{\Ndellonetwo{m}{n}\theta}+\pjac\ninv^{k}_{j}\left[\nCurl{\Ndellonetwo{m}{n}\theta}\right]^{j}_{i}\right),\label{rem-i}
	\end{align}
	\begin{align}
	\rem^{ki}_{2}(\underline{k})=-\sum_{\substack{\underline{p}+\underline{q}=\underline{k}\\|\underline{p}|>0}}\mathscr{L}(\underline{p},\underline{q})\left(\nabla^{\underline{p}}\left(\pjac\ninv^{k}_{j}\ninv^{s}_{i}\right)\nabla^{\underline{q}}\left(\dell_{s}\theta^{j}\right)+\frac{1}{\alpha}\nabla^{\underline{p}}\left(\pjac\ninv^{k}_{i}\ninv^{s}_{j}\right)\nabla^{\underline{q}}\left(\dell_{s}\theta^{j}\right)\right),\label{rem-k-i-2-interior}	
	\end{align}
	\begin{align}
	\rem^{i}(\underline{k})&=-\alpha e^{-\sigma_{2}\tau}\bar{\psi}\nabla^{\underline{k}}\left(\dell_{k}W\pjac\ninv^{k}_{i}\right)+e^{-\sigma_{2}\tau}\bar{\psi}\dell_{k}\left(W\rem^{ki}_{2}(\underline{k})\right)+e^{-\sigma_{2}\tau}W\dell_{k}\bar{\psi}\pjac\ninv^{k}_{j}\left[\ngrad{\nabla^{\underline{k}}\theta}\right]^{i}_{j}\nonumber\\
	&+e^{-\sigma_{2}\tau}W\dell_{k}\bar{\psi}\left(\pjac\ninv^{k}_{i}\ndiv{\nabla^{\underline{k}}\theta}+\pjac\ninv^{k}_{j}\left[\nCurl{\nabla^{\underline{k}}\theta}\right]^{j}_{i}\right),\label{rem-i-interior}
	\end{align}
	with $\mathcal{C}$ and $\mathcal{D}$ defined in Lemma \ref{commutators-instead-of-expansion}.
\end{definition}

We are now ready to state the relevant higher order energy identities.
\begin{theorem}[Higher Order Energy Identity]\label{higher-order-energy-identity-theorem}Let $(\nu,\theta)$ be the solution to~\eqref{E:EULERNEW2} on $[0,T]$ in the sense of Theorem~\ref{local-well-posedness-theorem}, for some $N\geq2\ceil*{\alpha}+12$. Fix $(m,\underline{n},\underline{k})\in\mathbb{Z}_{\geq0}\times\mathbb{Z}_{\geq0}^{3}\times\mathbb{Z}_{\geq0}^{3}$, $\underline{n}=(n,0,0)$, $n\geq1$, and such that $\max{(m+n,|\underline{k}|)}\leq N$. Then for all $\tau\in[0,T]$, we have
	\begin{align}
	&\frac{d}{d\tau}\left(\frac{1}{2\delta}\spaceI e^{\sigma_{1}\tau}\psi W^{\alpha+m}\left|\Ndellonetwo{m}{n}\nu\right|^{2}dx+\spaceI e^{-\sigma_{2}\tau}\psi W^{1+\alpha+m} \pjac\left(\frac{1}{2}|\ngrad{\Ndellonetwo{m}{n}\theta}|^{2}+\frac{1}{2\alpha}|\ndiv{\Ndellonetwo{m}{n}\theta}|^{2}\right)dx\right)\nonumber\\
	&+\frac{1}{2}\mathbb{D}^{\psi}_{(m,\underline{n})}-\frac{d}{d\tau}\left(\frac{1}{4}\spaceI e^{-\sigma_{2}\tau}\psi W^{1+\alpha+m} \pjac|\nCurl{\Ndellonetwo{m}{n}\theta}|^{2}dx\right)\nonumber\\
	&-\frac{\sigma_{2}}{4}\spaceI e^{-\sigma_{2}\tau}\psi W^{1+\alpha+m} \pjac|\nCurl{\Ndellonetwo{m}{n}\theta}|^{2}dx=\mathcal{R}(m,\underline{n}),\label{higher-order-energy-identity-statement}
	\end{align}
and
	\begin{align}
	&\frac{d}{d\tau}\left(\frac{1}{2\delta}\spaceI e^{\sigma_{1}\tau}\bar{\psi} \left|\nabla^{\underline{k}}\nu\right|^{2}dx+\spaceI e^{-\sigma_{2}\tau}\bar{\psi}W \pjac\left(\frac{1}{2}|\ngrad{\nabla^{\underline{k}}\theta}|^{2}+\frac{1}{2\alpha}|\ndiv{\nabla^{\underline{k}}\theta}|^{2}\right)dx\right)+\frac{1}{2}\mathbb{D}^{\bar{\psi}}_{\underline{k}}\nonumber\\
	&-\frac{d}{d\tau}\left(\frac{1}{4}\spaceI e^{-\sigma_{2}\tau}\bar{\psi}W \pjac|\nCurl{\nabla{\underline{k}}\theta}|^{2}dx\right)-\frac{\sigma_{2}}{4}\spaceI e^{-\sigma_{2}\tau}\bar{\psi}W \pjac|\nCurl{\nabla^{\underline{k}}\theta}|^{2}dx=\mathcal{R}(\underline{k}),\label{higher-order-energy-identity-statement-interior}
	\end{align}
where 
	\begin{align}
	\rem(m,\underline{n})=&-\frac{1}{2\alpha}\spaceI e^{-\sigma_{2}\tau}\psi W^{1+\alpha+m}\njac^{-\left(1+\frac{1}{\alpha}\right)}\dell_{\tau}\njac\left|\ngrad{\Ndellonetwo{m}{n}\theta}\right|^{2}dx\nonumber\\
	&-\frac{1}{2\alpha^{2}}\spaceI e^{-\sigma_{2}\tau}\psi W^{1+\alpha+m}\njac^{-\left(1+\frac{1}{\alpha}\right)}\dell_{\tau}\njac\left|\ndiv{\Ndellonetwo{m}{n}\theta}\right|^{2}dx\nonumber\\
	&+\frac{1}{4\alpha}\spaceI e^{-\sigma_{2}\tau}\psi W^{1+\alpha+m}\njac^{-\left(1+\frac{1}{\alpha}\right)}\dell_{\tau}\njac\left|\nCurl{\Ndellonetwo{m}{n}\theta}\right|^{2}dx\nonumber\\
	&+\spaceI e^{-\sigma_{2}\tau}\psi W^{1+\alpha+m}\pjac\dell_{\tau}\left(\ninv^{k}_{j}\right)\dell_{k}\left(\Ndellonetwo{m}{n}\theta^{i}\right)\left[\ngrad{\Ndellonetwo{m}{n}\theta}\right]^{i}_{j}dx\nonumber\\
	&+\frac{1}{\alpha}\spaceI e^{-\sigma_{2}\tau}\psi W^{1+\alpha+m}\pjac\left(\dell_{\tau}\ninv^{k}_{i}\dell_{k}\left(\Ndellonetwo{m}{n}\theta^{i}\right)\right)\ndiv{\Ndellonetwo{m}{n}\theta}dx\nonumber\\
	&-\sum_{i>j}\spaceI e^{-\sigma_{2}\tau}\psi W^{1+\alpha+m}\pjac\left[\dell_{\tau}\ninv^{k}_{i}\dell_{k}\left(\Ndellonetwo{m}{n}\theta^{j}\right)-\dell_{\tau}\ninv^{k}_{j}\dell_{k}\left(\Ndellonetwo{m}{n}\theta^{i}\right)\right]\left[\nCurl{\Ndellonetwo{m}{n}\theta}\right]^{j}_{i}dx\nonumber\\
	&+\spaceI\rem^{i}(m,\underline{n})\Ndellonetwo{m}{n}\nu^{i}dx,\label{higher-order-energy-remainder-statement}
	\end{align}
and
	\begin{align}
	\rem(\underline{k})=&-\frac{1}{2\alpha}\spaceI e^{-\sigma_{2}\tau}\bar{\psi} W\njac^{-\left(1+\frac{1}{\alpha}\right)}\dell_{\tau}\njac\left|\ngrad{\nabla^{\underline{k}}\theta}\right|^{2}dx\nonumber\\
	&-\frac{1}{2\alpha^{2}}\spaceI e^{-\sigma_{2}\tau}\bar{\psi} W\njac^{-\left(1+\frac{1}{\alpha}\right)}\dell_{\tau}\njac\left|\ndiv{\nabla^{\underline{k}}\theta}\right|^{2}dx\nonumber\\
	&+\frac{1}{4\alpha}\spaceI e^{-\sigma_{2}\tau}\bar{\psi} W\njac^{-\left(1+\frac{1}{\alpha}\right)}\dell_{\tau}\njac\left|\nCurl{\nabla^{\underline{k}}\theta}\right|^{2}dx\nonumber\\
	&+\spaceI e^{-\sigma_{2}\tau}\bar{\psi} W\pjac\dell_{\tau}\left(\ninv^{k}_{j}\right)\dell_{k}\left(\nabla^{\underline{k}}\theta^{i}\right)\left[\ngrad{\nabla^{\underline{k}}\theta}\right]^{i}_{j}dx\nonumber\\
	&+\frac{1}{\alpha}\spaceI e^{-\sigma_{2}\tau}\bar{\psi} W\pjac\left(\dell_{\tau}\ninv^{k}_{i}\dell_{k}\left(\nabla^{\underline{k}}\theta^{i}\right)\right)\ndiv{\nabla^{\underline{k}}\theta}dx\nonumber\\
	&-\sum_{i>j}\spaceI e^{-\sigma_{2}\tau}\bar{\psi} W\pjac\left[\dell_{\tau}\ninv^{k}_{i}\dell_{k}\left(\nabla^{\underline{k}}\theta^{j}\right)-\dell_{\tau}\ninv^{k}_{j}\dell_{k}\left(\nabla^{\underline{k}}\theta^{i}\right)\right]\left[\nCurl{\nabla^{\underline{k}}\theta}\right]^{j}_{i}dx\nonumber\\
	&+\spaceI\rem^{i}(\underline{k})\nabla^{\underline{k}}\nu^{i}dx,\label{higher-order-energy-remainder-statement-interior}
	\end{align}
 with $\rem^{i}(m,\underline{n})$, $\rem^{i}(\underline{k})$, $\mathcal{D}^{\psi}_{(m,\underline{n})}$, and $\mathcal{D}^{\bar{\psi}}_{\underline{k}}$ are given in Definitions $\ref{higher-order-remainder-terms-definition}$ and $\ref{damping-functional-definition}$.
\end{theorem}

\begin{remark}\label{general-higher-order-energy-identity-remark}
	A corresponding higher order energy identity holds true for the case of general $(m,\underline{n})$, with the remainder terms being appropriately defined in the case of dealing with general $\angdell^{\underline{n}}$.
\end{remark}

\begin{proof}
\textbf{Proof of~\eqref{higher-order-energy-identity-statement}.}
 First we rewrite (\ref{rescaled-euler-W-perturbation-estimate-form-1}) by dividing by $W^{\alpha}$,
\begin{align}
\frac{1}{\delta}e^{\sigma_{1}\tau}\left(\theta^{i}_{\tau\tau}+\theta^{i}_{\tau}\right)+\frac{e^{-\sigma_{2}\tau}}{W^{\alpha}}\dell_{k}\left(W^{1+\alpha}\ninv^{k}_{i}\pjac\right)=0.\label{rescaled-euler-W-perturbation-estimate-form}
\end{align}
We act on this equation by $\rdell^{m}\angdell_{12}^{n}$:
\begin{align}
\frac{1}{\delta}e^{\sigma_{1}\tau}\left(\dell_{\tau}\rdell^{m}\angdell_{12}^{n}\nu^{i}+\rdell^{m}\angdell_{12}^{n}\nu^{i}\right)&+\rdell^{m}\angdell_{12}^{n}\left(\frac{e^{-\sigma_{2}\tau}}{W^{\alpha}}\dell_{k}\left(W^{1+\alpha}\ninv^{k}_{i}\pjac\right)\right)=0.\label{rescaled-euler-W-perturbation-estimate-form-acted-on-by-Ndell}
\end{align}
Apply Lemma \ref{commutators-instead-of-expansion} to the pressure term in (\ref{rescaled-euler-W-perturbation-estimate-form-acted-on-by-Ndell}) in an inductive manner. This leaves us with the equality
\begin{align*}
    \angdell_{12}^{n}\left(\frac{e^{-\sigma_2\tau}}{W^{\alpha}}\dell_{k}\left(W^{1+\alpha}\ninv^{k}_{i}\pjac\right)\right)&=\frac{e^{-\sigma_{2}\tau}}{W^{\alpha}}\dell_{k}\left(W^{1+\alpha}\angdell_{12}^{n}\left(\ninv^{k}_{i}\pjac\right)\right)\\
    &+e^{-\sigma_{2}\tau}\sum_{a+b=n-1}\angdell_{12}^{a}\mathcal{D}^{\alpha}_{i}\left(1,2,\angdell_{12}^{b}\left(\pjac\ninv\right)\right).
\end{align*}
Then we apply $\rdell^{m}$ to the first term on the right hand side. This gives us
\begin{align*}
    \rdell^{m}\left(\frac{e^{-\sigma_{2}\tau}}{W^{\alpha}}\dell_{k}\left(W^{1+\alpha}\angdell_{12}^{n}\left(\ninv^{k}_{i}\pjac\right)\right)\right)&=\frac{e^{-\sigma_{2}\tau}}{W^{\alpha+m}}\dell_{k}\left(W^{1+\alpha+m}\Ndellonetwo{m}{n}\left(\ninv^{k}_{i}\pjac\right)\right)\\
    &+e^{-\sigma_{2}\tau}\sum_{l+p=m-1}\rdell^{l}\mathcal{C}^{1+\alpha+p}_{i}\left(\Ndellonetwo{p}{n}\left(\pjac\ninv\right)\right).
\end{align*}
Using these identities in (\ref{rescaled-euler-W-perturbation-estimate-form-acted-on-by-Ndell}), we obtain
\begin{align*}
    \frac{1}{\delta}e^{\sigma_{1}\tau}W^{\alpha+m}\left(\dell_{\tau}\Ndellonetwo{m}{n}\nu^{i}+\Ndellonetwo{m}{n}\nu^{i}\right)&+e^{-\sigma_{2}\tau}\dell_{k}\left(W^{1+\alpha+m}\Ndellonetwo{m}{n}\left(\ninv^{k}_{i}\pjac\right)\right)=\rem^{i}_{1}(m,\underline{n}),
\end{align*}
where $\rem^{i}_{1}(m,\underline{n})$, as in Definition \ref{higher-order-remainder-terms-definition}, is given by
\begin{align}
    -e^{-\sigma_{2}\tau}W^{\alpha+m}\left(\sum_{l+p=m-1}\rdell^{l}\mathcal{C}^{1+\alpha+p}_{i}\left(\Ndellonetwo{p}{n}\left(\pjac\ninv\right)\right)+\sum_{a+b=n-1}\Ndellonetwo{m}{a}\mathcal{D}^{\alpha}_{i}\left(1,2,\angdell_{12}^{b}\left(\pjac\ninv\right)\right)\right).\label{remainder-from-commutators-instead-of-expansion-lemma}
\end{align}
Now we concentrate on $\Ndellonetwo{m}{n}\left(\ninv^{k}_{i}\pjac\right)$. First of all, we see how $\ninv^{k}_{i}\pjac$ behaves upon the application of an angular derivative (note that the differentiation identities $(\ref{inverse-differentiation-zeta})$ and $(\ref{jacobian-differentiation-zeta})$ still hold for $\angdell$). From the proof of Lemma $\ref{pjac-ninv-linearisation-lemma}$ we have
\begin{align*}
    \angdell_{12}\left(\ninv^{k}_{i}\pjac\right)=&-\pjac\ninv^{k}_{j}\ninv^{s}_{i}\dell_{s}\angdell_{12}\theta^{j}-\frac{1}{\alpha}\pjac\ninv^{k}_{i}\ninv^{s}_{j}\dell_{s}\angdell_{12}\theta^{j}\\
    -&\left(\pjac\ninv^{k}_{j}\ninv^{s}_{i}\comm{\angdell_{12}}{\dell_{s}}\theta^{j}+\frac{1}{\alpha}\pjac\ninv^{k}_{i}\ninv^{s}_{j}\comm{\angdell_{12}}{\dell_{s}}\theta^{j}\right).
\end{align*}
To apply this to our identity, act on both sides with $\Ndellonetwo{m}{n-1}$:
\begin{align*}
    &\Ndellonetwo{m}{n}\left(\ninv^{k}_{i}\pjac\right)=-\pjac\ninv^{k}_{j}\ninv^{s}_{i}\Ndellonetwo{m}{n-1}\dell_{s}\angdell_{12}\theta^{j}-\frac{1}{\alpha}\pjac\ninv^{k}_{i}\ninv^{s}_{j}\Ndellonetwo{m}{n-1}\dell_{s}\angdell_{12}\theta^{j}\nonumber\\
    &-\Ndellonetwo{m}{n-1}\left(\pjac\ninv^{k}_{j}\ninv^{s}_{i}\comm{\angdell_{12}}{\dell_{s}}\theta^{j}+\frac{1}{\alpha}\pjac\ninv^{k}_{i}\ninv^{s}_{j}\comm{\angdell_{12}}{\dell_{s}}\theta^{j}\right)\nonumber\\
    &-\sum_{\substack{p+l=m\\q+b=n-1\\
    p+q>0}}\mathscr{L}(p,q,l,b)\left(\Ndellonetwo{p}{q}\left(\pjac\ninv^{k}_{j}\ninv^{s}_{i}\right)\Ndellonetwo{l}{b}\left(\dell_{s}\angdell_{12}\theta^{j}\right)+\frac{1}{\alpha}\Ndellonetwo{p}{q}\left(\pjac\ninv^{k}_{i}\ninv^{s}_{j}\right)\Ndellonetwo{l}{b}\left(\dell_{s}\angdell_{12}\theta^{j}\right)\right).
\end{align*}
Commute $\Ndellonetwo{m}{n-1}$ with $\dell_{s}$ on the first two terms on the right hand side above:
\begin{align*}
   &\Ndellonetwo{m}{n}\left(\ninv^{k}_{i}\pjac\right)=-\pjac\ninv^{k}_{j}\ninv^{s}_{i}\dell_{s}\Ndellonetwo{m}{n}\theta^{j}-\frac{1}{\alpha}\pjac\ninv^{k}_{i}\ninv^{s}_{j}\dell_{s}\Ndellonetwo{m}{n}\theta^{j}\\
   &-\Ndellonetwo{m}{n-1}\left(\pjac\ninv^{k}_{j}\ninv^{s}_{i}\comm{\angdell_{12}}{\dell_{s}}\theta^{j}+\frac{1}{\alpha}\pjac\ninv^{k}_{i}\ninv^{s}_{j}\comm{\angdell_{12}}{\dell_{s}}\theta^{j}\right)\\
   &-\sum_{\substack{p+l=m\\q+b=n-1\\p+|q|>0}}\mathscr{L}(p,q,l,b)\left(\Ndellonetwo{p}{q}\left(\pjac\ninv^{k}_{j}\ninv^{s}_{i}\right)\Ndellonetwo{l}{b}\left(\dell_{s}\angdell_{12}\theta^{j}\right)+\frac{1}{\alpha}\Ndellonetwo{p}{q}\left(\pjac\ninv^{k}_{i}\ninv^{s}_{j}\right)\Ndellonetwo{l}{b}\left(\dell_{s}\angdell_{12}\theta^{j}\right)\right)\\
   &-\pjac\ninv^{k}_{j}\ninv^{s}_{i}\comm{\Ndellonetwo{m}{n-1}}{\dell_{s}}\angdell_{12}\theta^{j}-\frac{1}{\alpha}\pjac\ninv^{k}_{i}\ninv^{s}_{j}\comm{\Ndellonetwo{m}{n-1}}{\dell_{s}}\angdell_{12}\theta^{j}.
\end{align*}
Once again, using the same idea as in the proof of Lemma \ref{pjac-ninv-linearisation-lemma}, we can rewrite this as
\begin{align}
    \Ndellonetwo{m}{n}\left(\ninv^{k}_{i}\pjac\right)=&-\pjac\ninv^{k}_{j}\left[\ngrad{\Ndellonetwo{m}{n}\theta}\right]^{i}_{j}-\frac{1}{\alpha}\pjac\ninv^{k}_{i}\ndiv{\Ndellonetwo{m}{n}\theta}\nonumber\\
    &-\pjac\ninv^{k}_{j}\left[\nCurl{\Ndellonetwo{m}{n}\theta}\right]^{j}_{i}-\rem^{ki}_{2}(m,\underline{n}),\label{linearising-pjac-ninv-using-Ndell-top-order-and-remainder}
\end{align}
where, as in Definition \ref{higher-order-remainder-terms-definition},
\begin{align}
\rem^{ki}_{2}(m,\underline{n})=&-\Ndellonetwo{m}{n-1}\left(\pjac\ninv^{k}_{j}\ninv^{s}_{i}\comm{\angdell_{12}}{\dell_{s}}\theta^{j}+\frac{1}{\alpha}\pjac\ninv^{k}_{i}\ninv^{s}_{j}\comm{\angdell_{12}}{\dell_{s}}\theta^{j}\right)\nonumber\\
&-\pjac\ninv^{k}_{j}\ninv^{s}_{i}\comm{\Ndellonetwo{m}{n-1}}{\dell_{s}}\angdell_{12}\theta^{j}-\frac{1}{\alpha}\pjac\ninv^{k}_{i}\ninv^{s}_{j}\comm{\Ndellonetwo{m}{n-1}}{\dell_{s}}\angdell_{12}\theta^{j}\nonumber\\
&-\sum_{\substack{p+l=m\\q+b=n-1\\p+q>0}}\mathscr{L}(p,q,l,b)\Ndellonetwo{p}{q}\left(\pjac\ninv^{k}_{j}\ninv^{s}_{i}\right)\Ndellonetwo{l}{b}\left(\dell_{s}\angdell_{12}\theta^{j}\right)\nonumber\\
&-\frac{1}{\alpha}\sum_{\substack{p+l=m\\q+b=n-1\\p+q>0}}\mathscr{L}(p,q,l,b)\Ndellonetwo{p}{q}\left(\pjac\ninv^{k}_{i}\ninv^{s}_{j}\right)\Ndellonetwo{l}{b}\left(\dell_{s}\angdell_{12}\theta^{j}\right).\label{linearising-pjac-ninv-using-Ndell-remainder}
\end{align}
So we can rewrite our equation as
\begin{align*}
    &\frac{1}{\delta}e^{\sigma_{1}\tau}W^{\alpha+m}\left(\dell_{\tau}\Ndellonetwo{m}{n}\nu^{i}+\Ndellonetwo{m}{n}\nu^{i}\right)-e^{-\sigma_{2}\tau}\dell_{k}\left(W^{1+\alpha+m}\pjac\ninv^{k}_{j}\left[\ngrad{\Ndellonetwo{m}{n}\theta}\right]^{i}_{j}\right)\nonumber\\
    &-e^{-\sigma_{2}\tau}\dell_{k}\left(W^{1+\alpha+m}\left(\frac{1}{\alpha}\pjac\ninv^{k}_{i}\ndiv{\Ndellonetwo{m}{n}\theta}+\pjac\ninv^{k}_{j}\left[\nCurl{\Ndellonetwo{m}{n}\theta}\right]^{j}_{i}\right)\right)\\
    &=\rem^{i}_{1}(m,\underline{n})+e^{-\sigma_{2}\tau}\dell_{k}\left(W^{1+\alpha+m}\rem^{ki}_{2}(m,\underline{n})\right).
\end{align*}
Finally, multiply by $\psi$ and commute $\psi$ with $\dell_{k}$ to acquire
\begin{align}
    &\frac{1}{\delta}\psi W^{\alpha+m}e^{\sigma_{1}\tau}\left(\dell_{\tau}\Ndellonetwo{m}{n}\nu^{i}+\Ndellonetwo{m}{n}\nu^{i}\right)-e^{-\sigma_{2}\tau}\dell_{k}\left(\psi W^{1+\alpha+m}\left(\pjac\ninv^{k}_{j}\left[\ngrad{\Ndellonetwo{m}{n}\theta}\right]^{i}_{j}\right)\right)\nonumber\\
    &-e^{-\sigma_{2}\tau}\dell_{k}\left(\psi W^{1+\alpha+m}\left(\frac{1}{\alpha}\pjac\ninv^{k}_{i}\ndiv{\Ndellonetwo{m}{n}\theta}+\pjac\ninv^{k}_{j}\left[\nCurl{\Ndellonetwo{m}{n}\theta}\right]^{j}_{i}\right)\right)=\nonumber\\
    &\psi\rem^{i}_{1}(m,\underline{n})+e^{-\sigma_{2}\tau}\psi\dell_{k}\left(W^{1+\alpha+m}\rem^{ki}_{2}(m,\underline{n})\right)+e^{-\sigma_{2}\tau}W^{1+\alpha+m}\dell_{k}\psi\pjac\ninv^{k}_{j}\left[\ngrad{\Ndellonetwo{m}{n}\theta}\right]^{i}_{j}\nonumber\\
    &+e^{-\sigma_{2}\tau}W^{1+\alpha+m}\dell_{k}\psi\left(\pjac\ninv^{k}_{i}\ndiv{\Ndellonetwo{m}{n}\theta}+\pjac\ninv^{k}_{j}\left[\nCurl{\Ndellonetwo{m}{n}\theta}\right]^{j}_{i}\right)=\rem^{i}(m,\underline{n}).\label{pre-integrated-higher-order-energy-identity}
\end{align}
Take the scalar product of this equation with $\Ndellonetwo{m}{n}\theta_{\tau}^{i}$, and integrate over the domain $\Omega$. This gives us
\begin{align*}
    &\frac{d}{d\tau}\left(\frac{1}{2\delta}\spaceI e^{\sigma_{1}\tau}\psi W^{\alpha+m}\left|\Ndellonetwo{m}{n}\nu\right|^{2}dx\right)+\frac{1}{2\delta}\left(2-\sigma_{1}\right)\spaceI e^{\sigma_{1}\tau}\psi W^{\alpha+m}\left|\Ndellonetwo{m}{n}\nu\right|^{2}dx\nonumber\\
    &+\spaceI e^{-\sigma_{2}\tau}\psi W^{1+\alpha+m}\pjac\dell_{k}\left(\Ndellonetwo{m}{n}\theta^{i}_{\tau}\right)\left(\ninv^{k}_{j}\left[\ngrad{\Ndellonetwo{m}{n}\theta}\right]^{i}_{j}+\frac{1}{\alpha}\ninv^{k}_{i}\ndiv{\Ndellonetwo{m}{n}\theta}\right)dx\nonumber\\
    &+\spaceI e^{-\sigma_{2}\tau}\psi W^{1+\alpha+m}\dell_{k}\left(\Ndellonetwo{m}{n}\theta^{i}_{\tau}\right)\pjac\ninv^{k}_{j}\left[\nCurl{\Ndellonetwo{m}{n}\theta}\right]^{j}_{i}\ dx=\spaceI \rem^{i}(m,\underline{n})\Ndellonetwo{m}{n}\nu^{i}dx.
\end{align*}
Note that we have used the divergence theorem on the top order $\ngrad$, $\ndiv$, and $\nCurl$ terms, as well as the fact that $W$ is $0$ on $\dell\Omega$. We now concentrate on these three integrals on the left hand side. First we have the $\ngrad$ identity given by
\begin{align}
&e^{-\sigma_{2}\tau} \pjac[\ngrad{\Ndellonetwo{m}{n}\theta}]^{i}_{j}\ninv^{k}_{j}\dell_{k}\left(\Ndellonetwo{m}{n}\theta^{i}_{\tau}\right) \notag \\
&\ \ \ \ =\frac{1}{2}\dell_{\tau}\left(e^{-\sigma_{2}\tau} \pjac|\ngrad{\Ndellonetwo{m}{n}\theta}|^{2}\right)+\frac{\sigma_{2}}{2}e^{-\sigma_{2}\tau}\pjac|\ngrad{\Ndellonetwo{m}{n}\theta}|^{2}\nonumber\\
&\ \ \ \ \ \ \ \  +\frac{1}{2\alpha}e^{-\sigma_{2}\tau} \njac^{-(1+\frac{1}{\alpha})}\dell_{\tau}\njac|\ngrad{\Ndellonetwo{m}{n}\theta}|^{2}-e^{-\sigma_{2}\tau} \pjac[\ngrad{\Ndellonetwo{m}{n}\theta}]^{i}_{j}\dell_{\tau}\left(\ninv^{k}_{j}\right)\dell_{k}\left(\Ndellonetwo{m}{n}\theta^{i}\right).\label{higher-order-ngrad-dell-tau-identity}
\end{align}
For $\ndiv$ we obtain
\begin{align}
    &\frac{1}{\alpha}e^{-\sigma_{2}\tau} \pjac\ndiv{\Ndellonetwo{m}{n}\theta}\ninv^{k}_{i}\dell_{k}\left(\Ndellonetwo{m}{n}\theta^{i}_{\tau}\right) \notag \\
    & \ \ \ \  =\frac{1}{2\alpha}\dell_{\tau}\left(e^{-\sigma_{2}\tau} \pjac|\ndiv{\Ndellonetwo{m}{n}\theta}|^{2}\right)+\frac{\sigma_{2}}{2\alpha}e^{-\sigma_{2}\tau}\pjac|\ndiv{\Ndellonetwo{m}{n}\theta}|^{2}\nonumber\\
    & \ \ \ \ \ \ \ \ +\frac{1}{2\alpha^{2}}e^{-\sigma_{2}\tau} \njac^{-(1+\frac{1}{\alpha})}\dell_{\tau}\njac|\ndiv{\Ndellonetwo{m}{n}\theta}|^{2}-\frac{1}{\alpha}e^{-\sigma_{2}\tau} \pjac\left(\ndiv{\Ndellonetwo{m}{n}\theta}\right)\dell_{\tau}\left(\ninv^{k}_{i}\right)\dell_{k}\left(\Ndellonetwo{m}{n}\theta^{i}\right).\label{higher-order-ndiv-dell-tau-identity}
\end{align}
The $\nCurl$ term requires an antisymmetrisation argument as with the zero order energy identity in Theorem $\ref{zero-order-energy-identity-theorem}$, so we end up with
\begin{align}
&e^{-\sigma_{2}\tau} \pjac\left[\nCurl{\Ndellonetwo{m}{n}\theta}\right]^{j}_{i}\ninv^{k}_{j}\dell_{k}\left(\Ndellonetwo{m}{n}\theta^{i}_{\tau}\right) \notag \\
& =-\frac{1}{4}\dell_{\tau}\left(e^{-\sigma_{2}\tau} \pjac\left|\nCurl{\Ndellonetwo{m}{n}\theta}\right|^{2}\right)-\frac{\sigma_{2}}{4}e^{-\sigma_{2}\tau}\pjac\left|\nCurl{\Ndellonetwo{m}{n}\theta}\right|^{2}\nonumber\\
&\ \ \ \ -\frac{1}{4\alpha}e^{-\sigma_{2}\tau}\njac^{-\left(1+\frac{1}{\alpha}\right)}\dell_{\tau}\njac\left|\nCurl{\Ndellonetwo{m}{n}\theta}\right|^{2} \notag \\
& \ \ \ \ +\sum_{i>j}e^{-\sigma_{2}\tau} \pjac\left[\dell_{\tau}\ninv^{k}_{i}\dell_{k}\left(\Ndellonetwo{m}{n}\theta^{j}\right)
-\dell_{\tau}\ninv^{k}_{j}\dell_{k}\left(\Ndellonetwo{m}{n}\theta^{i}\right)\right]\left[\nCurl{\Ndellonetwo{m}{n}\theta}\right]^{j}_{i}.\label{higher-order-nCurl-dell-tau-identity}
\end{align}
Using $(\ref{higher-order-ngrad-dell-tau-identity})-(\ref{higher-order-nCurl-dell-tau-identity})$ leaves us with
\begin{align*}
&\frac{d}{d\tau}\left(\frac{1}{2\delta}\spaceI e^{\sigma_{1}\tau}\psi W^{\alpha+m}\left|\Ndellonetwo{m}{n}\nu\right|^{2}dx+\spaceI e^{-\sigma_{2}\tau}\psi W^{1+\alpha+m} \pjac\left(\frac{1}{2}|\ngrad{\Ndellonetwo{m}{n}\theta}|^{2}+\frac{1}{2\alpha}|\ndiv{\Ndellonetwo{m}{n}\theta}|^{2}\right)dx\right) \notag \\
& \ \ \ \ +\frac{1}{2}\mathbb{D}^{\psi}_{(m,\underline{n})}
-\frac{d}{d\tau}\left(\frac{1}{4}\spaceI e^{-\sigma_{2}\tau}\psi W^{1+\alpha+m} \pjac|\nCurl{\Ndellonetwo{m}{n}\theta}|^{2}dx\right) \notag \\
& \ \ \ \ -\frac{\sigma_{2}}{4}\spaceI e^{-\sigma_{2}\tau}\psi W^{1+\alpha+m} \pjac|\nCurl{\Ndellonetwo{m}{n}\theta}|^{2}dx=\mathcal{R}(m,\underline{n}),\nonumber
\end{align*}
with $\mathbb{D}^{\psi}_{(m,\underline{n})}$ given in  $\ref{damping-functional-definition}$, and $\rem(m,\underline{n})$ in $(\ref{higher-order-energy-remainder-statement})$.\\

\noindent \textbf{Proof of $(\ref{higher-order-energy-identity-statement-interior})$}: This time we write $(\ref{rescaled-euler-W-perturbation-estimate-form-1})$ as
\begin{align}
\frac{1}{\delta}e^{\sigma_{1}\tau}\left(\theta_{\tau\tau}^{i}+\theta_{\tau}^{i}\right)+e^{-\sigma_{2}\tau}\dell_{k}\left(W\pjac\ninv^{k}_{i}\right)=-\alpha e^{-\sigma_{2}\tau}\dell_{k}W \pjac\ninv^{k}_{i}.\label{rescaled-euler-W-perturbation-estimate-form-interior}
\end{align}
Then we act on $(\ref{rescaled-euler-W-perturbation-estimate-form-interior})$ by $\nabla^{\underline{k}}$:
\begin{align}
\frac{1}{\delta}e^{\sigma_{1}\tau}\left(\nabla^{\underline{k}}\theta_{\tau\tau}^{i}+\nabla^{\underline{k}}\theta_{\tau}^{i}\right)+e^{-\sigma_{2}\tau}\dell_{k}\nabla^{\underline{k}}\left(W\pjac\ninv^{k}_{i}\right)=-\alpha e^{-\sigma_{2}\tau}\nabla^{\underline{k}}\left(\dell_{k}W \pjac\ninv^{k}_{i}\right).\label{rescaled-euler-W-perturbation-estimate-form-acted-on-by-nabla-k}
\end{align}
From here the proof of $(\ref{higher-order-energy-identity-statement-interior})$ follows that of $(\ref{higher-order-energy-identity-statement})$.
\end{proof}

We move on to higher order energy inequalities.

\begin{theorem}[Higher Order Energy Estimates]\label{higher-order-energy-estimates-theorem}Let $(\nu,\theta)$ be the solution to~\eqref{E:EULERNEW2} on $[0,T]$ in the sense of Theorem~\ref{local-well-posedness-theorem}, for some $N\geq2\ceil*{\alpha}+12$. Suppose the a priori assumptions $(\ref{a-priori-assumptions})$ hold, and $\nabla W\in\X{N}$. Fix $(m,\underline{n},\underline{k})\in\mathbb{Z}_{\geq0}\times\mathbb{Z}_{\geq0}^{3}\times\mathbb{Z}_{\geq0}^{3}$, with $\underline{n}=(n,0,0)$, $n\geq1$, and $\max{(m+n,|\underline{k}|)}\leq N$. Then for all $0\leq\tau_{1}\leq\tau\leq T$,
	\begin{align}
	\left|\int_{\tau_{1}}^{\tau}\left(\rem(m,\underline{n})+\rem(\underline{k})\right)d\tau'\right|\lesssim \mathscr{S}_{N}(\tau_{1})+\sqrt{\delta}\mathscr{S}_{N}(\tau)&+\int_{\tau_{1}}^{\tau}\sqrt{\delta}\left(e^{-\frac{\sigma_{1}}{2}\tau'}+\sigma_{2}e^{-\frac{\sigma_{2}}{2}\tau'}\right) \left(\mathscr{S}_{N}(\tau')^{3/2}+\mathscr{S}_{N}(\tau')\right)d\tau'\nonumber\\
	&+\int_{\tau_{1}}^{\tau}\sqrt{\delta}e^{-\frac{\sigma_{1}}{2}\tau'}\mathscr{S}_{N}(\tau')^{1/2}d\tau'.\label{higher-order-energy-estimates}
	\end{align}
\end{theorem}

\begin{remark}\label{general-higher-order-energy-estimates-remark}
	Once again, these higher order energy estimates follow for general $(m,\underline{n})$ by bounding the corresponding remainder terms in a manner completely analogous to the proof of the Theorem \ref{higher-order-energy-estimates-theorem}.
\end{remark}

\begin{proof}
	To begin with, we look at the first line of the right hand side of $(\ref{higher-order-energy-remainder-statement})$:
	\begin{align*}
	\spaceI e^{-\sigma_{2}\tau}\psi W^{1+\alpha+m}\njac^{-\left(1+\frac{1}{\alpha}\right)}\dell_{\tau}\njac\left|\ngrad{\Ndellonetwo{m}{n}\theta}\right|^{2}dx.
	\end{align*}
	Since $\dell_{\tau}\njac=\njac\ninv^{s}_{r}\dell_{s}\nu^{r}$, we have
	\begin{align}
	\left|\int_{\tau_{1}}^{\tau}\spaceI e^{-\sigma_{2}\tau'}\psi W^{1+\alpha+m}\njac^{-\left(1+\frac{1}{\alpha}\right)}\dell_{\tau}\njac\left|\ngrad{\Ndellonetwo{m}{n}\theta}\right|^{2}dx d\tau'\right|\lesssim \int_{\tau_{1}}^{\tau}\sqrt{\delta}e^{-\frac{\sigma_{1}}{2}\tau'} \mathscr{S}_{N}(\tau')^{3/2}d\tau'.\label{higher-order-energy-estimate-1}
	\end{align}
	We use $(\ref{Dnu-estimate})$ and $(\ref{a-priori-assumptions})$ respectively to bound the $\nabla\nu$ and $\ninv$ terms in $L^{\infty}$. The second and third lines on the right hand side of $(\ref{higher-order-energy-remainder-statement})$ have the same bound.
	
	Next is the fourth line on the right hand side of $(\ref{higher-order-energy-remainder-statement})$:
	\begin{align*}
	\spaceI e^{-\sigma_{2}\tau}\psi W^{1+\alpha+m}\pjac\dell_{\tau}\left(\ninv^{k}_{j}\right)\dell_{k}\left(\Ndellonetwo{m}{n}\theta^{i}\right)\left[\ngrad{\Ndellonetwo{m}{n}\theta}\right]^{i}_{j}dx.
	\end{align*}
	We write, using (\ref{inverse-differentiation-zeta}),
	\begin{align*}
	\dell_{\tau}\left(\ninv^{k}_{j}\right)\dell_{k}\left(\Ndellonetwo{m}{n}\theta^{i}\right)\left[\ngrad{\Ndellonetwo{m}{n}\theta}\right]^{i}_{j}&=-\ninv^{s}_{j}\dell_{s}\nu^{r}\ninv^{k}_{r}\dell_{k}\left(\Ndellonetwo{m}{n}\theta^{i}\right)\left[\ngrad{\Ndellonetwo{m}{n}\theta}\right]^{i}_{j}\\
	&=-\ninv^{s}_{j}\dell_{s}\nu^{r}\left[\ngrad{\Ndellonetwo{m}{n}\theta}\right]^{i}_{r}\left[\ngrad{\Ndellonetwo{m}{n}\theta}\right]^{i}_{j}.
	\end{align*}
	So we have the bound
	\begin{align*}
	&\left|\int_{\tau_{1}}^{\tau}\spaceI e^{-\sigma_{2}\tau'}\psi W^{1+\alpha+m}\pjac\dell_{\tau}\left(\ninv^{k}_{j}\right)\dell_{k}\left(\Ndellonetwo{m}{n}\theta^{i}\right)\left[\ngrad{\Ndellonetwo{m}{n}\theta}\right]^{i}_{j}dxd\tau'\right|\\
	&\lesssim\int_{\tau_{1}}^{\tau}\spaceI e^{-\sigma_{2}\tau'}\psi W^{1+\alpha+m}\pjac\left|\ninv\right|\left|\nabla\nu\right|\left|\ngrad{\Ndellonetwo{m}{n}\theta}\right|^{2}dxd\tau'.
	\end{align*}
	Once again using $(\ref{Dnu-estimate})$ on $\nabla\nu$, and a priori assumptions on $\ninv$ we have the bound
	\begin{align}
	\left|\int_{\tau_{1}}^{\tau}\spaceI e^{-\sigma_{2}\tau'}\psi W^{1+\alpha+m}\pjac\dell_{\tau}\left(\ninv^{k}_{j}\right)\dell_{k}\left(\Ndellonetwo{m}{n}\theta^{i}\right)\left[\ngrad{\Ndellonetwo{m}{n}\theta}\right]^{i}_{j}dxd\tau'\right|\lesssim\int_{\tau_{1}}^{\tau}\sqrt{\delta}e^{-\frac{\sigma_{1}}{2}\tau'}\mathscr{S}_{N}(\tau')^{3/2}d\tau'.\label{higher-order-energy-estimate-2}
	\end{align}
	The corresponding integrals with $\ndiv{\Ndellonetwo{m}{n}\theta}$ and $\nCurl{\Ndellonetwo{m}{n}}$ terms have the same bound. This deals with all terms on the right hand side of $(\ref{higher-order-energy-remainder-statement})$ except for the last.\\
	
	\noindent \textbf{Estimating $\rem^{i}(m,\underline{n})$.}\\
	
	\noindent We have
	\begin{align}
	\rem^{i}(m,\underline{n})&=\psi\rem^{i}_{1}(m,\underline{n})+e^{-\sigma_{2}\tau}\psi\dell_{k}\left(W^{1+\alpha+m}\rem^{ki}_{2}(m,\underline{n})\right)+e^{-\sigma_{2}\tau}W^{1+\alpha+m}\dell_{k}\psi\pjac\ninv^{k}_{j}\left[\ngrad{\Ndellonetwo{m}{n}\theta}\right]^{i}_{j}\nonumber\\
	&+e^{-\sigma_{2}\tau}W^{1+\alpha+m}\dell_{k}\psi\left(\pjac\ninv^{k}_{i}\ndiv{\Ndellonetwo{m}{n}\theta}\right)+e^{-\sigma_{2}\tau}W^{1+\alpha+m}\dell_{k}\psi\left(\pjac\ninv^{k}_{j}\left[\nCurl{\Ndellonetwo{m}{n}\theta}\right]^{j}_{i}\right),\label{rem-i-1-restatement}
	\end{align}
	and we start by estimating the third term on the right hand side. Since $\psi+\bar{\psi}\equiv1$, we have
	\begin{align*}
	&\int_{\tau_{1}}^{\tau}\spaceI e^{-\sigma_{2}\tau'}W^{1+\alpha+m}\dell_{k}\psi\pjac\ninv^{k}_{j}\left[\ngrad{\Ndellonetwo{m}{n}\theta}\right]^{i}_{j}\Ndellonetwo{m}{n}\nu^{i}dxd\tau'\\ &=\int_{\tau_{1}}^{\tau}\spaceI\psi e^{-\sigma_{2}\tau'}W^{1+\alpha+m}\dell_{k}\psi\pjac\ninv^{k}_{j}\left[\ngrad{\Ndellonetwo{m}{n}\theta}\right]^{i}_{j}\Ndellonetwo{m}{n}\nu^{i}dxd\tau'\\
	&+\int_{\tau_{1}}^{\tau}\spaceI\bar{\psi} e^{-\sigma_{2}\tau'}W^{1+\alpha+m}\dell_{k}\psi\pjac\ninv^{k}_{j}\left[\ngrad{\Ndellonetwo{m}{n}\theta}\right]^{i}_{j}\Ndellonetwo{m}{n}\nu^{i}dxd\tau'.
	\end{align*}
	For the first integral on the right hand side we have
	\begin{align}
	&\left|\int_{\tau_{1}}^{\tau}\spaceI\psi e^{-\sigma_{2}\tau'}W^{1+\alpha+m}\dell_{k}\psi\pjac\ninv^{k}_{j}\left[\ngrad{\Ndellonetwo{m}{n}\theta}\right]^{i}_{j}\Ndellonetwo{m}{n}\nu^{i}dxd\tau'\right|\nonumber\\
	&\lesssim\int_{\tau_{1}}^{\tau}\left\|\theta\right\|_{\Y{N}{\ngrad}}\left\|\nu\right\|_{\X{N}}d\tau'\lesssim\int_{\tau_{1}}^{\tau}\sqrt{\delta}e^{-\frac{\sigma_{1}}{2}\tau'}\mathscr{S}_{N}(\tau')d\tau'.\label{Dk-psi-ngrad-estimate-1}
	\end{align}
	We have used the a priori assumptions $(\ref{a-priori-assumptions})$ to bound $\ninv$ in $L^{\infty}$. Moreover, since $\psi$ is smooth, $\nabla\psi$ can also be bounded in $L^{\infty}$. From there, we use Cauchy-Schwarz to obtain the estimate.
	
	For the corresponding $\supp{\psi}$ integral, we can use $(\ref{Ndell-to-Dk-statement})$ in Lemma \ref{Ndell-to-Dk-lemma} to turn all $\Ndell{a}{b}$ terms in to rectangular derivatives, and we can again bound all $\ninv$ terms in $L^{\infty}$ using $(\ref{a-priori-assumptions})$. Therefore, this integral has the same bound as in $(\ref{Dk-psi-ngrad-estimate-1})$. Overall, we end up with
	\begin{align}
	\left|\int_{\tau_{1}}^{\tau}\spaceI\psi e^{-\sigma_{2}\tau'}W^{1+\alpha+m}\dell_{k}\psi\pjac\ninv^{k}_{j}\left[\ngrad{\Ndellonetwo{m}{n}\theta}\right]^{i}_{j}\Ndellonetwo{m}{n}\nu^{i}dxd\tau'\right|\lesssim \int_{\tau_{1}}^{\tau}\sqrt{\delta}e^{-\frac{\sigma_{1}}{2}\tau'}\mathscr{S}_{N}(\tau')d\tau'.\label{higher-order-energy-estimate-3}
	\end{align}	
	The same bounds hold for the last two terms in  the $(\ref{rem-i-1-restatement})$ involving $\ndiv{\Ndellonetwo{m}{n}\theta}$ and $\nCurl{\Ndellonetwo{m}{n}\theta}$ terms.
	
	Now we turn our attention to estimating $\rem^{i}_{1}(m,\underline{n})$.\\
	
	\noindent \textbf{Remainder Terms from Applying Lemma \ref{commutators-instead-of-expansion} to $\pjac\ninv$.}\\
	
	\noindent We recall that for $i=1,2,3$,
	\begin{align*}
	\rem^{i}_{1}(m,\underline{n})=-e^{-\sigma_{2}\tau}W^{\alpha+m}\left(\sum_{l+p=m-1}\rdell^{l}\mathcal{C}^{1+\alpha+p}_{i}\left(\Ndellonetwo{p}{n}\left(\pjac\ninv\right)\right)+\sum_{a+b=n-1}\Ndellonetwo{m}{a}\mathcal{D}^{\alpha}_{i}\left(1,2,\angdell_{12}^{b}\left(\pjac\ninv\right)\right)\right).
	\end{align*}
	First we deal the term
	\begin{align}
	e^{-\sigma_{2}\tau}W^{\alpha+m}\Ndellonetwo{m}{a}\mathcal{D}_{i}^{\alpha}\left(1,2,\angdell_{12}^{b}\left(\pjac\ninv\right)\right),\label{mathcal-D-expansion-1}
	\end{align}
	for $a+b=n-1$. This term is, by definition,
	\begin{align}
	e^{-\sigma_{2}\tau}W^{\alpha+m}\Ndellonetwo{m}{a}\left(W\comm{\angdell_{12}}{\dell_{k}}\angdell_{12}^{b}\left(\pjac\ninv^{k}_{i}\right)+(1+\alpha)\angdell_{12}^{b}\left(\pjac\ninv^{k}_{i}\right)\angdell_{12}\dell_{k}W+\angdell_{12}W\dell_{k}\angdell_{12}^{b}\left(\pjac\ninv^{k}_{i}\right)\right).\label{mathcal-D-expansion-2}
	\end{align}
	From $(\ref{mathcal-D-expansion-2})$, we must take extra care with terms of the form
	\begin{align}
	&\underbrace{e^{-\sigma_{2}\tau}W^{\alpha+m}\angdell_{12}^{c}W\Ndellonetwo{m}{a}\comm{\angdell_{12}}{\dell_{k}}\angdell_{12}^{b}\left(\pjac\ninv^{k}_{i}\right)}_{\mathcal{A}_{1}}+\underbrace{(1+\alpha)e^{-\sigma_{2}\tau}W^{\alpha+m}\Ndellonetwo{m}{a'}\left(\pjac\ninv^{k}_{i}\right)\angdell_{12}^{b'}\dell_{k}W}_{\mathcal{A}_{2}}\nonumber\\
	&+\underbrace{e^{-\sigma_{2}\tau}W^{\alpha+m}\angdell_{12}^{c''}W\Ndellonetwo{m}{a''}\dell_{k}\angdell_{12}^{b}\left(\pjac\ninv^{k}_{i}\right)}_{\mathcal{A}_{3}},\label{mathcal-D-expansion-3}
	\end{align}
	for $a+c+b=a''+c''+b=n-1$, and $a'+b'=n$. The extra care is required in dealing with the powers of $W$ in the integrals, as $W^{\alpha+m}$ is not suitable by itself for these particular terms. First we look at $\mathcal{A}_{1}$. We have
	\begin{align*}
	e^{-\sigma_{2}\tau}W^{\alpha+m}\angdell_{12}^{c}W\Ndellonetwo{m}{a}\comm{\angdell_{12}}{\dell_{k}}\angdell_{12}^{b}\left(\pjac\ninv^{k}_{i}\right)=e^{-\sigma_{2}\tau}d_{\Omega}W^{\alpha+m}\angdell_{12}^{c}\left(\frac{W}{d_{\Omega}}\right)\Ndellonetwo{m}{a}\comm{\angdell_{12}}{\dell_{k}}\angdell_{12}^{b}\left(\pjac\ninv^{k}_{i}\right),
	\end{align*}
	where the equality comes from $d(x,\dell\Omega)$ being a function of $r$ only. On $\supp{\psi}$, $d_{\Omega}W^{\alpha+m}\sim W^{1+\alpha+m}$. If $c\leq N-2$, we have the bound
	\begin{align}
	&\left|\int_{\tau_{1}}^{\tau}\spaceI e^{-\sigma_{2}\tau'}\psi d_{\Omega}W^{\alpha+m}\angdell_{12}^{c}\left(\frac{W}{d_{\Omega}}\right)\Ndellonetwo{m}{a}\comm{\angdell_{12}}{\dell_{k}}\angdell_{12}^{b}\left(\pjac\ninv^{k}_{i}\right)\Ndellonetwo{m}{n}\nu^{i}dxd\tau'\right|\nonumber\\
	&\lesssim\left\|\angdell_{12}^{c}\left(\frac{W}{d_{\Omega}}\right)\right\|_{L^{\infty}(\supp{\psi})}\int_{\tau_{1}}^{\tau}\spaceI e^{-\sigma_{2}\tau'}\psi W^{1+\alpha+m}\left|\Ndellonetwo{m}{a}\comm{\angdell_{12}}{\dell_{k}}\angdell_{12}^{b}\left(\pjac\ninv^{k}_{i}\right)\right|\left|\Ndellonetwo{m}{n}\nu^{i}\right|dxd\tau'\nonumber\\
	&\lesssim\int_{\tau_{1}}^{\tau}\sqrt{\delta}e^{-\frac{\sigma_{1}}{2}\tau'}\mathscr{S}_{N}(\tau')d\tau'.\label{higher-order-energy-estimate-7}
	\end{align}
	To obtain the final inequality, we use the argument given in (\ref{W-over-d-estimate-1}) to bound the $W/d_{\Omega}$ term, and then use Cauchy-Schwarz on the integral over $\Omega$. We use $(\ref{commutator-rectangular-ang})$ and $(\ref{rectangular-as-ang-rad})$ to convert the commutator $\comm{\angdell_{12}}{\dell_{k}}$ in to a linear combination of radial and angular derivatives, and then apply Lemma \ref{lower-order-estimates-lemma}. If $c=N-1$ or $N$, then we can bound $\Ndellonetwo{m}{a}\comm{\angdell_{12}}{\dell_{k}}\angdell_{12}^{b}\left(\pjac\ninv^{k}_{i}\right)$ in $L^{\infty}$ instead, and obtain the same bound as in $(\ref{higher-order-energy-estimate-7})$. Estimating $\mathcal{A}_{3}$ is also completely analogous to $(\ref{higher-order-energy-estimate-7})$, and gives the same bound.
	
	Finally, we look at $\mathcal{A}_{2}$:
	\begin{align*}
	(1+\alpha)e^{-\sigma_{2}\tau}W^{\alpha+m}\Ndellonetwo{m}{a'}\left(\pjac\ninv^{k}_{i}\right)\angdell_{12}^{b'}\dell_{k}W&=(1+\alpha)e^{-\sigma_{2}\tau}W^{\alpha+m}\angdell_{12}^{b'}\left(d_{\Omega}\dell_{k}\left(\frac{W}{d_{\Omega}}\right)\right)\Ndellonetwo{m}{a'}\left(\pjac\ninv^{k}_{i}\right)\nonumber\\
	&+(1+\alpha)e^{-\sigma_{2}\tau}W^{\alpha+m}\angdell_{12}^{b'}\left(\frac{W}{d_{\Omega}}\dell_{k}d_{\Omega}\right)\Ndellonetwo{m}{a'}\left(\pjac\ninv^{k}_{i}\right).
	\end{align*}
	The first term on the right hand side can be bounded analogously to $(\ref{higher-order-energy-estimate-7})$. For the other term we first write
	\begin{align}
	W^{\alpha+m}\angdell_{12}^{b'}\left(\frac{W}{d_{\Omega}}\dell_{k}d_{\Omega}\right)=\sum_{c'+d'=b'} W^{\alpha+m}\mathscr{L}(c',d')\angdell_{12}^{c'}\left(\frac{W}{d_{\Omega}}\right)\angdell_{12}^{d'}\left(\dell_{k}d_{\Omega}\right).\label{W-dell-k-d-identity-1}
	\end{align}
	Now, using the fact that $d(x,\dell\Omega)=d_{\Omega}(x)$ is a function of $r$ only, as well as $(\ref{commutator-rectangular-ang})$, we have
	\begin{align}
	\angdell_{12}\dell_{k}d_{\Omega}=\comm{\angdell_{12}}{\dell_{k}}d_{\Omega}=\delta_{k2}\dell_{1}d_{\Omega}-\delta_{k1}\dell_{2}d_{\Omega}.\label{W-dell-k-d-identity-2}
	\end{align}
	We also know that $\comm{\angdell_{12}}{\dell_{1}}=-\dell_{2}$, and $\comm{\angdell_{12}}{\dell_{2}}=\dell_{2}$. Hence for any $d'\leq b'$, and $k=1,2,3$, we have
	\begin{align}
	\angdell_{12}^{d'}\dell_{k}d_{\Omega}=\mathscr{L}(d',k)\dell_{1}d_{\Omega}+\bar{\mathscr{L}}(d',k)\dell_{2}d_{\Omega},\label{W-dell-k-d-identity-3}
	\end{align}
	for some constants $\mathscr{L}$ and $\bar{\mathscr{L}}$ depending on $d'$ and $k$. Owing to $(\ref{W-dell-k-d-identity-3})$, we first look at the integral
	\begin{align}
	\spaceI e^{-\sigma_{2}\tau}\psi W^{\alpha+m}\dell_{1}d_{\Omega}\angdell_{12}^{c'}\left(\frac{W}{d_{\Omega}}\right)\Ndellonetwo{m}{a'}\left(\pjac\ninv^{k}_{i}\right)\Ndellonetwo{m}{n}\nu^{i}dx,\label{d-1-alpha-m-ibp-1}
	\end{align}
	The following identity is also key: 
	\begin{align}
	W^{\alpha+m}\dell_{1}d_{\Omega}=\frac{1}{1+\alpha+m}\left(\frac{W}{d_{\Omega}}\right)^{\alpha+m}\dell_{1}\left(d_{\Omega}^{1+\alpha+m}\right).\label{W-dell-k-d-identity-4}
	\end{align}
	So $(\ref{d-1-alpha-m-ibp-1})$ becomes
	\begin{align}
	&\spaceI e^{-\sigma_{2}\tau}\psi W^{\alpha+m}\dell_{1}d_{\Omega}\angdell_{12}^{c'}\left(\frac{W}{d_{\Omega}}\right)\Ndellonetwo{m}{a'}\left(\pjac\ninv^{k}_{i}\right)\Ndellonetwo{m}{n}\nu^{i}dx\nonumber\\
	&=-\frac{1}{1+\alpha+m}\spaceI e^{-\sigma_{2}\tau}d_{\Omega}^{1+\alpha+m}\ \dell_{1}\left(\psi\left(\frac{W}{d_{\Omega}}\right)^{\alpha+m}\angdell_{12}^{c'}\left(\frac{W}{d_{\Omega}}\right) \Ndellonetwo{m}{a'}\left(\pjac\ninv^{k}_{i}\right)\Ndellonetwo{m}{n}\nu^{i}\right)dx,\label{d-1-alpha-m-ibp-2}
	\end{align}
	with the right hand side coming from integration by parts in $\dell_{1}$ as well as $(\ref{W-dell-k-d-identity-4})$. Up to constant, the integral on the right hand side is
	\begin{align}
	&-\spaceI e^{-\sigma_{2}\tau}d_{\Omega}^{1+\alpha+m}\ \dell_{1}\psi\left(\frac{W}{d_{\Omega}}\right)^{\alpha+m}\angdell_{12}^{c'}\left(\frac{W}{d_{\Omega}}\right) \Ndellonetwo{m}{a'}\left(\pjac\ninv^{k}_{i}\right)\Ndellonetwo{m}{n}\nu^{i}dx\nonumber\\
	&-\left(\alpha+m\right)\spaceI e^{-\sigma_{2}\tau}d_{\Omega}^{1+\alpha+m}\ \psi\left(\frac{W}{d_{\Omega}}\right)^{\alpha+m-1}\dell_{1}\left(\frac{W}{d_{\Omega}}\right)\angdell_{12}^{c'}\left(\frac{W}{d_{\Omega}}\right) \Ndellonetwo{m}{a'}\left(\pjac\ninv^{k}_{i}\right)\Ndellonetwo{m}{n}\nu^{i}dx\nonumber\\
	&-\spaceI e^{-\sigma_{2}\tau}d_{\Omega}^{1+\alpha+m}\ \psi\left(\frac{W}{d_{\Omega}}\right)^{\alpha+m}\angdell_{12}^{c'}\left(\frac{W}{d_{\Omega}}\right) \dell_{1}\Ndellonetwo{m}{a'}\left(\pjac\ninv^{k}_{i}\right)\Ndellonetwo{m}{n}\nu^{i}dx\nonumber\\
	&-\spaceI e^{-\sigma_{2}\tau}d_{\Omega}^{1+\alpha+m}\ \psi\left(\frac{W}{d_{\Omega}}\right)^{\alpha+m}\angdell_{12}^{c'}\left(\frac{W}{d_{\Omega}}\right) \Ndellonetwo{m}{a'}\left(\pjac\ninv^{k}_{i}\right)\dell_{1}\Ndellonetwo{m}{n}\nu^{i}dx.\label{d-1-alpha-m-ibp-3}
	\end{align}
	Note that due to $(\ref{W-and-d-equivalence})$,
	\begin{align*}
	\left|\frac{W}{d_{\Omega}}\right|+\left|\frac{d_{\Omega}}{W}\right|\lesssim1
	\end{align*}
	on $\supp{\psi}$. Hence, the first three integrals above can be bounded as in $(\ref{higher-order-energy-estimate-7})$. For the third integral, we also use $(\ref{rectangular-as-ang-rad})$ to write $\dell_{1}$ as a linear combination of radial and angular derivatives with smooth coefficients.
	
	This leaves us with the last integral, in which $\dell_{1}\Ndellonetwo{m}{n}\nu^{i}$ has too many derivatives to be bounded in the correct energy space. So we write
	\begin{align*}
	e^{-\sigma_{2}\tau} \Ndellonetwo{m}{a'}\left(\pjac\ninv^{k}_{i}\right)\dell_{1}\left(\Ndellonetwo{m}{n}\nu^{i}\right)=&\dell_{\tau}\left(e^{-\sigma_{2}\tau} \Ndellonetwo{m}{a'}\left(\pjac\ninv^{k}_{i}\right)\dell_{1}\left(\Ndellonetwo{m}{n}\theta^{i}\right)\right)\\
	&+\sigma_{2}e^{-\sigma_{2}\tau} \Ndellonetwo{m}{a'}\left(\pjac\ninv^{k}_{i}\right)\dell_{1}\left(\Ndellonetwo{m}{n}\theta^{i}\right)\\
	&-e^{-\sigma_{2}\tau} \Ndellonetwo{m}{a'}\left(\pjac\ninv^{k}_{j}\ninv^{s}_{i}\dell_{s}\nu^{i}\right)\dell_{1}\left(\Ndellonetwo{m}{n}\theta^{i}\right)\\
	&-\frac{1}{\alpha}e^{-\sigma_{2}\tau} \Ndellonetwo{m}{a'}\left(\pjac\ninv^{k}_{i}\ninv^{s}_{j}\dell_{s}\nu^{j}\right)\dell_{1}\left(\Ndellonetwo{m}{n}\theta^{i}\right).
	\end{align*}
	Every term on the right hand side above can be estimated as in $(\ref{higher-order-energy-estimate-7})$, along with the Fundamental Theorem of Calculus in $\tau$ for the first term. We end up with the bound
	\begin{align}
	&\left|\int_{\tau_{1}}^{\tau}\spaceI e^{-\sigma_{2}\tau'}\psi\left(\frac{W}{d_{\Omega}}\right)^{\alpha+m}\angdell_{12}^{c'}\left(\frac{W}{d_{\Omega}}\right)d_{\Omega}^{1+\alpha+m} \Ndellonetwo{m}{a'}\left(\pjac\ninv^{k}_{i}\right)\dell_{1}\left(\Ndellonetwo{m}{n}\nu^{i}\right)dxd\tau'\right|\nonumber\\
	&\lesssim \sqrt{\delta}e^{-\frac{\sigma_{2}}{2}\tau_{1}}\mathscr{S}_{N}(\tau_{1})+\sqrt{\delta}e^{-\frac{\sigma_{2}}{2}\tau}\mathscr{S}_{N}(\tau)+\int_{\tau_{1}}^{\tau}\sqrt{\delta}\left(\sigma_{2}e^{-\frac{\sigma_{2}}{2}\tau'}+e^{-\frac{\beta}{2}\tau'}\right)\mathscr{S}_{N}(\tau')d\tau'\nonumber\\
	&\lesssim \mathscr{S}_{N}(\tau_{1})+\sqrt{\delta}\mathscr{S}_{N}(\tau)+\int_{\tau_{1}}^{\tau}\sqrt{\delta}\left(\sigma_{2}e^{-\frac{\sigma_{2}}{2}\tau'}+e^{-\frac{\sigma_{1}}{2}\tau'}\right)\mathscr{S}_{N}(\tau')d\tau'.\label{higher-order-energy-estimate-11}
	\end{align}
	We bound the integral with $\bar{\mathscr{L}}(d',k)\dell_{2}$ instead of $\mathscr{L}(d',k) \dell_{1}$ analogously. Combining the bounds $(\ref{higher-order-energy-estimate-7})$ and $(\ref{higher-order-energy-estimate-14})$ gives us the estimate for $(\ref{mathcal-D-expansion-3})$.
	
 Any terms coming from~\eqref{mathcal-D-expansion-2} that are not of the form $\mathcal{A}_{1}$, $\mathcal{A}_{2}$, or $\mathcal{A}_{3}$ can be bounded using Lemma \ref{lower-order-estimates-lemma}. Therefore, we have the bound
	\begin{align}
	&\left|\sum_{a+b=n-1}\int_{\tau_{1}}^{\tau}\spaceI\psi e^{-\sigma_{2}\tau}W^{\alpha+m}\Ndellonetwo{m}{a}\mathcal{D}_{i}^{\alpha}\left(1,2,\angdell_{12}^{b}\left(\pjac\ninv\right)\right)\Ndellonetwo{m}{n}\nu^{i}dxd\tau'\right|\nonumber\\
	&\lesssim \mathscr{S}_{N}(\tau_{1})+\sqrt{\delta}\mathscr{S}_{N}(\tau)+\int_{\tau_{1}}^{\tau}\sqrt{\delta}\left(\sigma_{2}e^{-\frac{\sigma_{2}}{2}\tau'}+e^{-\frac{\sigma_{1}}{2}\tau'}\right)\mathscr{S}_{N}(\tau')d\tau'+\int_{\tau_{1}}^{\tau}\sqrt{\delta}e^{-\frac{\sigma_{1}}{2}\tau'}\mathscr{S}_{N}(\tau')^{1/2}d\tau'.\label{higher-order-energy-estimate-14}
	\end{align}
	It remains to bound all terms of the form
	\begin{align}
	e^{-\sigma_{2}\tau}W^{\alpha+m}\rdell^{l}\mathcal{C}^{1+\alpha+p}_{i}\left(\Ndellonetwo{p}{n}\left(\pjac\ninv\right)\right),\label{mathcal-C-expansion-1}
	\end{align}
	where $l+p=m-1$. By definition this term is equal to
	\begin{align}
	e^{-\sigma_{2}\tau}W^{\alpha+m}\rdell^{l}\left(\dell_{r}W \frac{x_{j}}{r}\angdell_{jk}\Ndellonetwo{p}{n}\left(\pjac\ninv^{k}_{i}\right)-W\dell_{k}\Ndellonetwo{p}{n}\left(\pjac\ninv^{k}_{i}\right)+(1+\alpha+p)\Ndellonetwo{p}{n}\left(\pjac\ninv^{k}_{i}\right)\rdell\dell_{k}W\right).\label{mathcal-C-expansion-2}
	\end{align}
	We can bound this term analogously to $(\ref{higher-order-energy-estimate-7})$ using Lemma \ref{lower-order-estimates-lemma}. Thus we have
	\begin{align}
	\left|\sum_{l+p=m-1}\int_{\tau_{1}}^{\tau}\spaceI\psi e^{-\sigma_{2}\tau}W^{\alpha+m}\rdell^{l}\mathcal{C}^{1+\alpha+p}_{i}\left(\Ndellonetwo{p}{n}\left(\pjac\ninv\right)\right)\Ndellonetwo{m}{n}\nu^{i}dxd\tau'\right|\lesssim \int_{\tau_{1}}^{\tau}\sqrt{\delta}e^{-\frac{\sigma_{1}}{2}\tau'}\mathscr{S}_{N}(\tau')d\tau'.\label{higher-order-energy-estimate-15}
	\end{align}
	We are left to bound $\rem^{ki}_{2}(m,\underline{n})$.\\
	
	\noindent \textbf{Remainder Terms from Linearising $\pjac\ninv$.}\\
	
	\noindent We look at the expression $e^{-\sigma_{2}\tau}\psi\dell_{k}\left(W^{1+\alpha+m}\rem^{ki}_{2}(m,\underline{n})\right)$, which by definition is given by
	\begin{align*}
	&e^{-\sigma_{2}\tau}\psi\dell_{k}\left(-W^{1+\alpha+m}\Ndellonetwo{m}{n-1}\left(\pjac\ninv^{k}_{j}\ninv^{s}_{i}\comm{\angdell_{12}}{\dell_{s}}\theta^{j}+\frac{1}{\alpha}\pjac\ninv^{k}_{i}\ninv^{s}_{j}\comm{\angdell_{12}}{\dell_{s}}\theta^{j}\right)\right)\nonumber\\
	&-e^{-\sigma_{2}\tau}\psi\dell_{k}\left(W^{1+\alpha+m}\pjac\ninv^{k}_{j}\ninv^{s}_{i}\comm{\Ndellonetwo{m}{n-1}}{\dell_{s}}\angdell_{12}\theta^{j}-\frac{1}{\alpha}W^{1+\alpha+m}\pjac\ninv^{k}_{i}\ninv^{s}_{j}\comm{\Ndellonetwo{m}{n-1}}{\dell_{s}}\angdell_{12}\theta^{j}\right)\nonumber\\
	&-e^{-\sigma_{2}\tau}\psi\dell_{k}\left(W^{1+\alpha+m}\sum_{\substack{p+l=m\\q+b=n-1\\p+q>0}}\mathscr{L}(p,q,l,b)\Ndellonetwo{p}{q}\left(\pjac\ninv^{k}_{j}\ninv^{s}_{i}\right)\Ndellonetwo{l}{b}\left(\dell_{s}\angdell_{12}\theta^{j}\right)\right)\nonumber\\
	&-\frac{1}{\alpha}e^{-\sigma_{2}\tau}\psi\dell_{k}\left(W^{1+\alpha+m}\sum_{\substack{p+l=m\\q+b=n-1\\p+q>0}}\mathscr{L}(p,q,l,b)\Ndellonetwo{p}{q}\left(\pjac\ninv^{k}_{i}\ninv^{s}_{j}\right)\Ndellonetwo{l}{b}\left(\dell_{s}\angdell_{12}\theta^{j}\right)\right).
	\end{align*}
	We first transform this expression using integration by parts:
	\begin{align}
	\spaceI e^{-\sigma_{2}\tau}\psi\dell_{k}\left(W^{1+\alpha+m}\rem^{ki}_{2}(m,\underline{n})\right)\Ndellonetwo{m}{n}\nu^{i}dx=&-\spaceI e^{-\sigma_{2}\tau}\psi W^{1+\alpha+m}\rem^{ki}_{2}(m,\underline{n})\dell_{k}\left(\Ndellonetwo{m}{n}\nu^{i}\right)dx\nonumber\\
	&-\spaceI e^{-\sigma_{2}\tau}\dell_{k}\psi W^{1+\alpha+m}\rem^{ki}_{2}(m,\underline{n})\Ndellonetwo{m}{n}\nu^{i}dx.\label{dell-k-R-ki-ibp}
	\end{align}
	The second term on the right hand side is lower order, and we can use an argument similar to that in Lemma \ref{lower-order-estimates-lemma} to bound all terms coming from $\rem^{ki}_{2}$ in this case. We have the estimate
	\begin{align}
	\left|\int_{\tau_{1}}^{\tau}\spaceI e^{-\sigma_{2}\tau'}\dell_{k}\psi W^{1+\alpha+m}\rem^{ki}_{2}(m,\underline{n})\Ndellonetwo{m}{n}\nu^{i}dxd\tau'\right|\lesssim\int_{\tau_{1}}^{\tau}\delta e^{-\frac{\sigma_{1}}{2}\tau'}\mathscr{S}_{N}(\tau')d\tau'.\label{higher-order-energy-estimate-12}
	\end{align}
	For the other, top order integral, we first write
	\begin{align*}
	e^{-\sigma_{2}\tau}\psi W^{1+\alpha+m}\rem^{ki}_{2}(m,\underline{n})\dell_{k}\left(\Ndellonetwo{m}{n}\nu^{i}\right)=&\dell_{\tau}\left(e^{-\sigma_{2}\tau}\psi W^{1+\alpha+m}\rem^{ki}_{2}(m,\underline{n})\dell_{k}\left(\Ndellonetwo{m}{n}\theta^{i}\right)\right)\\
	&+\sigma_{2}e^{-\sigma_{2}\tau}\psi W^{1+\alpha+m}\rem^{ki}_{2}(m,\underline{n})\dell_{k}\left(\Ndellonetwo{m}{n}\theta^{i}\right)\\
	&-e^{-\sigma_{2}\tau}\psi W^{1+\alpha+m}\dell_{\tau}\rem^{ki}_{2}(m,\underline{n})\dell_{k}\left(\Ndellonetwo{m}{n}\theta^{i}\right).
	\end{align*}
	All of these terms can again be estimated using the methods set out in the proof of Lemma \ref{lower-order-estimates-lemma}, so we have
	\begin{align}
	&\left|\int_{\tau_{1}}^{\tau}\spaceI e^{-\sigma_{2}\tau'}\psi W^{1+\alpha+m}\rem^{ki}_{2}(m,\underline{n})\dell_{k}\left(\Ndellonetwo{m}{n}\nu^{i}\right)dxd\tau'\right|\nonumber\\
	&\lesssim\sqrt{\delta}e^{-\frac{\sigma_{2}}{2}\tau_{1}}\mathscr{S}_{N}(\tau_{1})+\sqrt{\delta}e^{-\frac{\sigma_{2}}{2}\tau}\mathscr{S}_{N}(\tau)+\int_{\tau_{1}}^{\tau}\sqrt{\delta}(\sigma_{2}e^{-\frac{\sigma_{2}}{2}\tau'}+e^{-\frac{\beta}{2}\tau'})\mathscr{S}_{N}(\tau')d\tau'\nonumber\\
	&\lesssim\mathscr{S}_{N}(\tau_{1})+\sqrt{\delta}\mathscr{S}_{N}(\tau)+\int_{\tau_{1}}^{\tau}\sqrt{\delta}(\sigma_{2}e^{-\frac{\sigma_{2}}{2}\tau'}+e^{-\frac{\sigma_{1}}{2}\tau'})\mathscr{S}_{N}(\tau')d\tau'.\label{higher-order-energy-estimate-13}
	\end{align}
	The bounds for $\rem(\underline{k})$ are analogous as $W\sim1$ on $\supp{\bar{\psi}}$. Thus we have Theorem \ref{higher-order-energy-estimates-theorem}.
\end{proof}

\section{Energy Inequality and Proof of Theorem \ref{main-theorem}}

In this section we prove the main result of this paper, Theorem \ref{main-theorem}. However, we first need an energy inequality.

\begin{theorem}[Energy Inequality]\label{energy-inequality-theorem}Let $(\nu,\theta)$ be the solution to~\eqref{E:EULERNEW2} on $[0,T]$ in the sense of Theorem~\ref{local-well-posedness-theorem}, for some $N\geq2\ceil*{\alpha}+12$. Suppose the a priori assumptions $(\ref{a-priori-assumptions})$ hold, and $\nabla W\in\X{N}$. Then for all $0\leq\tau_{1}\leq\tau^{*}\leq T$,
	\begin{align}
	\mathscr{S}_{N}(\tau^{*};\tau_{1})\leq C_{1}\mathscr{S}_{N}(\tau_{1})+C_{2}\sqrt{\delta}+C_{3}\left(\mathscr{C}_{N}(0)+\mathscr{S}_{N}(0)\right)+C_{4}\sqrt{\delta}\mathscr{S}_{N}(\tau^{*};\tau_{1})+C_{5}\int_{\tau_{1}}^{\tau^{*}}\sqrt{\delta}\mathcal{G}(\beta,\tau')\mathscr{S}_{N}(\tau';\tau_{1})d\tau',\label{energy-inequality-statement}
	\end{align}
with $C_{j}$, $j=1,2,3,4,5$ constants $\geq1$, and $\mathcal{G}(\beta,\tau)$ integrable in $\tau$ for all fixed $\beta$.
\end{theorem}

Taken from \cite{HaJa2}, the notation $\mathscr{S}_{N}(\tau;\tau_{1})$ is a modification of Definition~\ref{energy-function-for-solutions}, in that we take the supremum over $[\tau_{1},\tau]$ instead of $[0,\tau]$.

The proof will require using Theorems \ref{higher-order-energy-identity-theorem} and \ref{higher-order-energy-estimates-theorem}. Once again we stress that while we assumed a specific form of $(m,\underline{n})$ to prove these energy identities and estimates on the vacuum boundary, this was for simplicity, and there are no additional technicalities when proving the same statements for general $(m,\underline{n})$, only extra care when dealing with $\angdell^{\underline{n}}$. Thus, throughout the proof we assume the statement of these theorems for all $(m,\underline{n})$ such that $m+|\underline{n}|\leq N$.

\begin{proof}
	Fix $\tau_{1}\leq\tau^{*}\leq T$. To begin with we integrate the energy identities $(\ref{zero-order-energy-identity-statement})$, $(\ref{higher-order-energy-identity-statement})$ and $(\ref{higher-order-energy-identity-statement-interior})$ in $\tau\in[\tau_{1},\tau^{*}]$, sum over all $(m,\underline{n})$ and $\underline{k}$ such that $\max{(m+|\underline{n}|,|\underline{k}|)}\leq N$, and obtain
	\begin{align}
	&\frac{1}{2}\left(\frac{1}{\delta}e^{\sigma_{1}\tau}\left\|\nu\right\|_{\X{N}}^{2}(\tau)+e^{-\sigma_{2}\tau}\left\|\theta\right\|_{\Y{N}{\ngrad}}^{2}(\tau)+\frac{1}{\alpha}e^{-\sigma_{2}\tau}\left\|\theta\right\|_{\Y{N}{\ndiv}}^{2}(\tau)\right)\nonumber\\
	&\lesssim\frac{1}{2}\left(\frac{1}{\delta}e^{\sigma_{1}\tau_{1}}\left\|\nu\right\|_{\X{N}}^{2}(\tau_{1})+e^{-\sigma_{2}\tau_{1}}\left\|\theta\right\|_{\Y{N}{\ngrad}}^{2}(\tau_{1})+\frac{1}{\alpha}e^{-\sigma_{2}\tau_{1}}\left\|\theta\right\|_{\Y{N}{\ndiv}}^{2}(\tau_{1})\right)\nonumber\\
	&+\frac{1}{4}\left(e^{-\sigma_{2}\tau}\left\|\theta\right\|_{\Y{N}{\nCurl}}^{2}(\tau)-e^{-\sigma_{2}\tau_{1}}\left\|\theta\right\|_{\Y{N}{\nCurl}}^{2}(\tau_{1})\right)\nonumber\\
	&+\frac{\sigma_{2}}{4}\int_{\tau_{1}}^{\tau}e^{-\sigma_{2}\tau}\left\|\theta\right\|_{\Y{N}{\nCurl}}^{2}(\tau')d\tau'+\int_{\tau_{1}}^{\tau}\left(\rem(0)+\rem(m,\underline{n})+\rem(\underline{k})\right)d\tau'.\label{energy-inequality-proof-1}
	\end{align}
	Note that whilst $(\ref{higher-order-energy-identity-statement-interior})$ does not have the same powers of $W$ on the $\supp{\bar{\psi}}$ integrals as the definition of the energy norms in $(\ref{energy-space-norm})$ and $(\ref{energy-space-norm-2})$, $W\sim1$ on this region, so this is not an issue. Now we apply Theorems \ref{zero-order-energy-estimates-theorem} and  \ref{higher-order-energy-estimates-theorem}, and get
	\begin{align}
	&\frac{1}{2}\left(\frac{1}{\delta}e^{\sigma_{1}\tau}\left\|\nu\right\|_{\X{N}}^{2}(\tau)+e^{-\sigma_{2}\tau}\left\|\theta\right\|_{\Y{N}{\ngrad}}^{2}(\tau)+\frac{1}{\alpha}e^{-\sigma_{2}\tau}\left\|\theta\right\|_{\Y{N}{\ndiv}}^{2}(\tau)\right)\nonumber\\
	&\lesssim\frac{1}{4}e^{-\sigma_{2}\tau}\left\|\theta\right\|_{\Y{N}{\nCurl}}^{2}(\tau)+\frac{\sigma_{2}}{4}\int_{\tau_{1}}^{\tau}e^{-\sigma_{2}\tau}\left\|\theta\right\|_{\Y{N}{\nCurl}}^{2}(\tau')d\tau'+\mathscr{S}_{N}(\tau_{1})+\sqrt{\delta}\mathscr{S}_{N}(\tau)\nonumber\\
	&+\int_{\tau_{1}}^{\tau}\sqrt{\delta}\left(e^{-\frac{\sigma_{1}}{2}\tau'}+\sigma_{2}e^{-\frac{\sigma_{2}}{2}\tau'}\right) \left(\mathscr{S}_{N}(\tau')^{3/2}+\mathscr{S}_{N}(\tau')\right)d\tau'+\int_{\tau_{1}}^{\tau}\sqrt{\delta}e^{-\frac{\sigma_{1}}{2}\tau'}\mathscr{S}_{N}(\tau')^{1/2}d\tau'.\label{energy-inequality-proof-2}
	\end{align}
	Applying Theorem \ref{Curl-estimates-theorem-theta}, this bound becomes
	\begin{align}
	&\frac{1}{2}\left(\frac{1}{\delta}e^{\sigma_{1}\tau}\left\|\nu\right\|_{\X{N}}^{2}(\tau)+e^{-\sigma_{2}\tau}\left\|\theta\right\|_{\Y{N}{\ngrad}}^{2}(\tau)+\frac{1}{\alpha}e^{-\sigma_{2}\tau}\left\|\theta\right\|_{\Y{N}{\ndiv}}^{2}(\tau)\right)\nonumber\\
	&\lesssim \left(\mathscr{C}_{N}(0)+\mathscr{S}_{N}(0)\right)+\sqrt{\delta}\mathscr{S}_{N}(\tau)+\mathscr{S}_{N}(\tau_{1})+\int_{\tau_{1}}^{\tau}\sqrt{\delta}\bar{\mathcal{G}}(\beta,\tau')\mathscr{S}_{N}(\tau')d\tau'+\int_{\tau_{1}}^{\tau}\sqrt{\delta}e^{-\frac{\sigma_{1}}{2}\tau'}\mathscr{S}_{N}(\tau')^{1/2}d\tau'.\label{energy-inequality-proof-3}
	\end{align}
	Here $\bar{\mathcal{G}}$ is defined as
	\begin{align}
	\bar{\mathcal{G}}(\beta,\tau)=e^{-\frac{\sigma_{1}}{2}\tau}+\sigma_{2}\left(e^{-\frac{\sigma_{2}}{2}\tau}+\sum_{j=1}^{2}e^{-\sigma_{2}\tau}H_{j}(\beta,\tau)+\sum_{i=1}^{6}e^{-\sigma_{2}\tau}\tilde{G}_{i}(\beta,\tau)\right),
	\end{align}
	where $H_{j}$ and $\tilde{G}_{i}$ are given in Definition \ref{definition-of-H-G-dash-curl-estimates}. We have also used the property that $e^{-\sigma_{2}\tau}H_{j}(\beta,\tau)$ and $e^{-\sigma_{2}\tau}\tilde{G}_{i}(\beta,\tau)$ are uniformly bounded in $\tau$ for all $\beta>0$. Next we use Young's Inequality on the last integrand on the right hand side of $(\ref{energy-inequality-proof-3})$. We also add $\left\|\theta\right\|_{\X{N}}^{2}$ to both sides and employ Lemma \ref{theta-estimates-lemma} to obtain
	\begin{align}
	&\left(\frac{1}{\delta}e^{\sigma_{1}\tau}\left\|\nu(\tau)\right\|_{\X{N}}^{2}+\left\|\theta(\tau)\right\|_{\X{N}}^{2}+e^{-\sigma_{2}\tau}\left\|\theta(\tau)\right\|_{\Y{N}{\ngrad}}^{2}+\frac{1}{\alpha}e^{-\sigma_{2}\tau}\left\|\theta(\tau)\right\|_{\Y{N}{\ndiv}}^{2}\right)\nonumber\\
	&\lesssim \left(\mathscr{C}_{N}(0)+\mathscr{S}_{N}(0)\right)+\sqrt\delta+\sqrt{\delta}\mathscr{S}_{N}(\tau)+\mathscr{S}_{N}(\tau_{1})+\int_{\tau_{1}}^{\tau}\sqrt{\delta}\mathcal{G}(\beta,\tau)\mathscr{S}_{N}(\tau')d\tau'.\label{energy-inequality-proof-4}
	\end{align}
	Note that $\mathcal{G}$ is given by $\bar{\mathcal{G}}+e^{-\frac{\sigma_{1}}{2}\tau}$. Moreover, since $\sigma_{2}(\beta)=0$ for all $\beta\in(0,2]$, and $e^{-\sigma_{2}\tau}H_{j}(\beta,\tau)$ and $e^{-\sigma_{2}\tau}\tilde{G}_{i}(\beta,\tau)$ are integrable in $\tau$ for all $\beta>2$, we have that $\mathcal{G}$ is integrable in $\tau$ for all $\beta>0$. By definition, $\mathscr{S}_{N}(\tau)\leq\mathscr{S}_{N}(\tau;\tau_{1})+\mathscr{S}_{N}(\tau_{1})$. As $\mathcal{G}$ is integrable, we have
	\begin{align}
	&\left(\frac{1}{\delta}e^{\sigma_{1}\tau}\left\|\nu\right\|_{\X{N}}^{2}(\tau)+\left\|\theta\right\|_{\X{N}}^{2}+e^{-\sigma_{2}\tau}\left\|\theta\right\|_{\Y{N}{\ngrad}}^{2}(\tau)+\frac{1}{\alpha}e^{-\sigma_{2}\tau}\left\|\theta\right\|_{\Y{N}{\ndiv}}^{2}(\tau)\right)\nonumber\\
	&\lesssim \left(\mathscr{C}_{N}(0)+\mathscr{S}_{N}(0)\right)+\sqrt\delta+\mathscr{S}_{N}(\tau_{1})+\sqrt{\delta}\mathscr{S}_{N}(\tau;\tau_{1})+\int_{\tau_{1}}^{\tau}\sqrt{\delta}\mathcal{G}(\beta,\tau)\left(\mathscr{S}_{N}(\tau';\tau_{1})+\mathscr{S}_{N}(\tau_{1})\right)d\tau'\nonumber\\
	&\lesssim \left(\mathscr{C}_{N}(0)+\mathscr{S}_{N}(0)\right)+\sqrt\delta+\mathscr{S}_{N}(\tau_{1})+\sqrt{\delta}\mathscr{S}_{N}(\tau;\tau_{1})+\int_{\tau_{1}}^{\tau}\sqrt{\delta}\mathcal{G}(\beta,\tau)\mathscr{S}_{N}(\tau';\tau_{1})d\tau'.\label{energy-inequality-proof-5}
	\end{align}
	Finally, taking the supremum over $[\tau_{1},\tau^{*}]$ on both sides, and adjusting constants on the right hand side if necessary, we have~\eqref{energy-inequality-statement}, noting the last two terms in $(\ref{energy-inequality-proof-5})$ are increasing in $\tau$.
	
	\subsection{Proof of Theorem \ref{main-theorem}}
	Let $T$ be such that
	\begin{align}
	\sup_{0\leq\tau\leq T}\mathscr{S}_{N}(\tau)=\mathscr{S}_{N}(T)\leq \bar{C}(\mathscr{S}_{N}(0)+\mathscr{C}_{N}(0)+\sqrt\delta),\label{main-proof-inequality-1}
	\end{align}
	for some constant $\bar{C}$ given by local well-posedness theory. Let $C^{*}$ be defined by
	\begin{align}
	C^{*}=3(C_{1}\bar{C}+C_{2}+C_{3}),\label{C-star-definition}
	\end{align}
	with $C_{i}$, $i=1,2,3,4,5$ given in Theorem \ref{energy-inequality-theorem}. Note that $\bar{C}<C^{*}$, and so if we define
	\begin{align}
	T^{*}=\sup_{\tau\geq0}\left\{\text{\ Solution to~\eqref{E:EULERNEW2} exists on $[0,\tau]$\ }|\mathscr{S}_{N}(\tau)\leq C^{*}(\mathscr{S}_{N}(0)+\mathscr{C}_{N}(0)+\sqrt\delta)\right\},\label{T-star-definition}
	\end{align}
	then $T\leq T^{*}$. Now letting $\tau_{1}=T/2$ in Theorem \ref{energy-inequality-theorem}, for any $\tau^{*}\in[\frac{T}{2},T^{*})$, we have
	\begin{align}
	\mathscr{S}_{N}(\tau^{*};\frac{T}{2})&\leq C_{1}\mathscr{S}_{N}(\frac{T}{2})+C_{2}\sqrt\delta+C_{3}\left(\mathscr{C}_{N}(0)+\mathscr{S}_{N}(0)\right)+C_{4}\sqrt{\delta}\mathscr{S}_{N}(\tau^{*};\frac{T}{2})+C_{5}\int_{\frac{T}{2}}^{\tau^{*}}\sqrt{\delta}\mathcal{G}(\beta,\tau')\mathscr{S}_{N}(\tau';\frac{T}{2})d\tau'\nonumber\\
	&\leq C_{1}\mathscr{S}_{N}(\frac{T}{2})+C_{2}\sqrt\delta+C_{3}\left(\mathscr{C}_{N}(0)+\mathscr{S}_{N}(0)\right)+C_{4}\sqrt{\delta}\mathscr{S}_{N}(\tau^{*};\frac{T}{2})+C_{5}\sqrt{\delta}\left(\int_{0}^{\infty}\mathcal{G}(\beta,\tau')d\tau'\right)\mathscr{S}_{N}(\tau^{*};\frac{T}{2}).\label{main-proof-inequality-2}
	\end{align}
	Let $\delta$ be so small that
	\begin{align}
	\delta\leq\min{\left(\left(4C_{4}\right)^{-2},\left(4C_{5}\int_{0}^{\infty}\mathcal{G}(\beta,\tau')d\tau'\right)^{-2}\right)}.\label{delta-smallness-requirements}
	\end{align}
	Then we have 
	\begin{align}
	\mathscr{S}_{N}(\tau^{*};\frac{T}{2})\leq 2C_{1}\mathscr{S}_{N}(\frac{T}{2})+2C_{2}\sqrt\delta+2C_{3}\left(\mathscr{C}_{N}(0)+\mathscr{S}_{N}(0)\right).\label{main-proof-inequality-3}
	\end{align}
	Using $(\ref{main-proof-inequality-1})$ to bound $\mathscr{S}_{N}(T/2)$ by $\bar{C}\left(\mathscr{C}_{N}(0)+\mathscr{S}_{N}(0)+\sqrt\delta\right)$, we obtain
	\begin{align}
	\mathscr{S}_{N}(\tau^{*};\frac{T}{2})&\leq2C_{1}\bar{C}\left(\mathscr{C}_{N}(0)+\mathscr{S}_{N}(0)+\sqrt\delta\right)+2C_{2}\sqrt\delta+2C_{3}\left(\mathscr{C}_{N}(0)+\mathscr{S}_{N}(0)\right)\nonumber\\
	&\leq2(C_{1}\bar{C}+C_{2}+C_{3})\left(\mathscr{C}_{N}(0)+\mathscr{S}_{N}(0)+\sqrt \delta\right)<C^{*}\left(\mathscr{C}_{N}(0)+\mathscr{S}_{N}(0)+\sqrt\delta\right).\label{main-proof-inequality-4}
	\end{align}
	Combining $(\ref{main-proof-inequality-4})$ with $(\ref{main-proof-inequality-1})$, we obtain $\mathscr{S}_{N}(\tau^{*})<C^{*}\left(\mathscr{C}_{N}(0)+\mathscr{S}_{N}(0)+\sqrt\delta\right)$, for all $\tau^{*}\in[0,T^{*})$. Shrinking $\delta$ further if necessary, we also improve our a priori assumptions $(\ref{a-priori-assumptions})$. For $\ninv$ we have, for all $\tau\in[0,T^{*})$,
	\begin{align}
	\left\|\ninv-I\right\|_{L^{\infty}}=\left\|\int_{0}^{\tau}\dell_{\tau}\ninv d\tau'\right\|_{L^{\infty}}\leq C\sqrt{\delta}\int_{0}^{\tau}e^{-\frac{\sigma_{1}}{2}\tau'}\mathscr{S}_{N}(\tau')d\tau'<\frac{1}{3},
	\end{align}
	for small enough $\delta$. Similarly for $\njac$.
	
 Then, by continuity of $\mathscr{S}_{N}(\tau)$ as a function of $\tau$, we must have that $T^{*}=\infty$. Thus the bound $(\ref{global-energy-estimate})$ follows. 
It is left to prove $(\ref{theta-infinity-statement})$. 
	Let $\tau_{1}>\tau_{2}$. For any $m+|\underline{n}|\leq N$, we have
	\begin{align}
	&\spaceI\psi W^{\alpha+m}\left|\Ndell{m}{n}\left(\theta(\tau_{1})-\theta(\tau_{2})\right)\right|^{2}dx=\spaceI\psi W^{\alpha+m}\left|\int_{\tau_{2}}^{\tau_{1}}\Ndell{m}{n}\nu(\tau')d\tau'\right|^{2}dx\nonumber\\
	&\lesssim\left(\int_{\tau_{2}}^{\tau_{1}}e^{-\frac{\sigma_{1}}{2}\tau'}d\tau'\right)\int_{\tau_{2}}^{\tau_{1}}e^{\frac{\sigma_{1}}{2}\tau'}\spaceI\psi W^{\alpha+m}\left|\Ndell{m}{n}\nu\right|^{2}dxd\tau'\nonumber\\
	&\lesssim\delta\mathscr{S}_{N}(\tau_{1})\int_{\tau_{2}}^{\tau_{1}}e^{-\frac{\sigma_{1}}{2}\tau'}d\tau' \notag \\
	& \lesssim \delta \left(e^{-\frac{\sigma_{1}}{2}\tau_{2}}-e^{-\frac{\sigma_{1}}{2}\tau_{1}}\right).
	\end{align}
	The first estimate follows from the Cauchy-Schwarz inequality in $\tau$, the second follows from the integrability of  $e^{-\frac{\sigma_{1}}{2}\tau}$ , and the last one follows from the uniform-in-$\tau$ boundedness of $\delta\mathscr{S}_{N}(\tau)$.
	
	An analogous estimate holds for any $|\underline{k}|\leq N$ on $\supp{\bar{\psi}}$. Thus $\left\|\theta(\tau_{1})-\theta(\tau_{2})\right\|_{\X{N}}\rightarrow0$ as $\tau_{1},\tau_{2}\rightarrow\infty$. In particular this implies that $\theta(\tau_{n})$ is Cauchy, for any strictly increasing sequence $\tau_{n}$. As $\X{N}$ is a Banach space, this implies the existence of $\lim_{\tau\to\infty}\theta(\tau)$ in $\X{N}$, which we call $\theta_{\infty}$. Thus we have $(\ref{theta-infinity-statement})$.
\end{proof}

\section*{Acknowledgments}
S. Parmeshwar acknowledges the support of the EPSRC studentship grant EP/N509498/1.
M. Had\v zi\'c acknowledges the support of the EPSRC grant 
EP/N016777/1. J. Jang acknowledges the support of the NSF grant DMS-1608494.


\begin{appendices}
\section{Commutators}\label{commutators}
We recall that in Section $\ref{derivatives}$, we introduced our angular and radial derivatives, $\angdell$ and $\rdell$ and the various commutator identities $(\ref{commutator-ang-rad})-(\ref{commutator-rectangular-ang})$. In this appendix, we state how these commutators can be written as radial and angular derivatives, following~\cite{HaJa2}.

\begin{lemma}[Higher order commutator identities]\label{higher-order-commutator-identities} Let $\dell_{s}$ be a rectangular derivative. Then for $\underline{n}\in\mathbb{Z}^{3}_{\geq0}$ with $|\underline{n}|>0$, we have
\begin{align}
    \comm{\Ndell{m}{n}}{\dell_{s}}=\sum_{i=1}^{m+|\underline{n}|}\sum_{\substack{a+|\underline{b}|=i\\a\leq m+1}}\mathcal{K}_{s,i,a,\underline{b}}\Ndell{a}{b},\label{expansion-commutator-Ndell-rectangular}
\end{align}
where $\mathcal{K}_{s,i,a,\underline{b}}$ are smooth functions away from the origin.

If $|\underline{n}|=0$, then we instead have
\begin{align}
	\comm{\rdell^{m}}{\dell_{s}}=\sum_{i=1}^{m+|\underline{n}|}\sum_{\substack{a+|\underline{b}|=i\\a\leq m}}\mathcal{F}_{s,i,a,\underline{b}}\Ndell{a}{b},\label{expansion-commutator-rdell-rectangular}
\end{align}
where, again, $\mathcal{F}_{s,i,a,\underline{b}}$ are smooth functions away from the origin.
\end{lemma}

\begin{proof}
	First we prove $(\ref{expansion-commutator-Ndell-rectangular})$. Recall that $(\ref{rectangular-as-ang-rad})$ gives us the identity
	\begin{align*}
		\dell_{q}=\frac{x_{k}}{r^{2}}\angdell_{kq}+\frac{x_{q}}{r^{2}}X_{r}.
	\end{align*}
	Write $\angdell^{\underline{n}}=\angdell^{\underline{\tilde{n}}}\angdell_{jl}$, with $|\underline{\tilde{n}}|=|\underline{n}|-1$. We have
	\begin{align*}
		\comm{\Ndell{m}{n}}{\dell_{s}}&=\Ndell{m}{\tilde{n}}\dell_{s}\angdell_{jl}-\dell_{s}\Ndell{m}{n}+\Ndell{m}{\tilde{n}}\comm{\angdell_{jl}}{\dell_{s}}.
	\end{align*}
	Then we invoke $(\ref{commutator-rectangular-ang})$:
	\begin{align*}
		\Ndell{m}{\tilde{n}}\comm{\angdell_{jl}}{\dell_{s}}&=-\delta_{sj}\Ndell{m}{\tilde{n}}\left(\dell_{l}\right)+\delta_{sl}\Ndell{m}{\tilde{n}}\left(\dell_{j}\right)\\
		&=-\delta_{sj}\Ndell{m}{\tilde{n}}\left(\frac{x_{k}}{r^{2}}\angdell_{kl}+\frac{x_{l}}{r^{2}}X_{r}\right)+\delta_{sl}\Ndell{m}{\tilde{n}}\left(\frac{x_{k}}{r^{2}}\angdell_{kj}+\frac{x_{j}}{r^{2}}X_{r}\right).
	\end{align*}
	Now $x/r^{2}$ is a smooth function away from the origin, so expanding the right hand side gives us
	\begin{align*}
		\Ndell{m}{\tilde{n}}\comm{\angdell_{jl}}{\dell_{s}}=\sum_{i=1}^{m+|\underline{n}|}\sum_{\substack{a+|\underline{b}|=i\\a\leq m+1}}K^{(1)}_{s,i,a,\underline{b}}\Ndell{a}{b},
	\end{align*}
	where $K^{(1)}_{s,i,a,\underline{b}}$ are smooth functions away from the origin.
	
	We can keep repeating this procedure with $\Ndell{m}{\tilde{n}}\dell_{s}\angdell_{jl}$, swapping the $\dell_{s}$ sequentially with the sequence of angular derivatives, at each step gaining a term with the same structure as
	\begin{align*}
		\sum_{i=1}^{m+|\underline{n}|}\sum_{\substack{a+|\underline{b}|=i\\a\leq m+1}}K^{(1)}_{s,i,a,\underline{b}}\Ndell{a}{b}
	\end{align*}
	to end up with 
	\begin{align*}
		\comm{\Ndell{m}{n}}{\dell_{s}}=\rdell^{m}\dell_{s}\angdell^{\underline{n}}-\dell_{s}\Ndell{m}{n}+\sum_{i=1}^{m+|\underline{n}|}\sum_{\substack{a+|\underline{b}|=i\\a\leq m+1}}K^{(2)}_{s,i,a,\underline{b}}\Ndell{a}{b},
	\end{align*}
	where $K^{(2)}_{s,i,a,\underline{b}}$ are smooth functions away from the origin.
	
	If $m>0$ we note that $\rdell^{m}\dell_{s}\angdell^{\underline{n}}=\rdell^{m-1}\dell_{s}\rdell\angdell^{\underline{n}}+\rdell^{m-1}\comm{\rdell}{\dell_{s}}\angdell^{\underline{n}}$, and we invoke $(\ref{commutator-rectangular-rad})$ to get
	\begin{align*}
		\rdell^{m}\dell_{s}\angdell^{\underline{n}}&=\rdell^{m-1}\dell_{s}\rdell\angdell^{\underline{n}}-\rdell^{m-1}\dell_{s}\angdell^{\underline{n}}\\
		&=\rdell^{m-1}\dell_{s}\rdell\angdell^{\underline{n}}-\rdell^{m-1}\left(\frac{x_{k}}{r^{2}}\angdell_{ks}+\frac{x_{s}}{r^{2}}X_{r}\right)\angdell^{\underline{n}}\\
		&=\rdell^{m-1}\dell_{s}\rdell\angdell^{\underline{n}}+\sum_{i=1}^{m+|\underline{n}|}\sum_{\substack{a+|\underline{b}|=i\\a\leq m}}K^{(3)}_{s,i,a,\underline{b}}\Ndell{a}{b},
	\end{align*}
	with $K^{(3)}_{s,i,a,\underline{b}}$ smooth away from the origin. We can also keep repeating this process until we get
	\begin{align*}
		\comm{\Ndell{m}{n}}{\dell_{s}}&=\rdell\dell_{s}\rdell^{m-1}\angdell^{\underline{n}}-\dell_{s}\Ndell{m}{n}+\sum_{i=1}^{m+|\underline{n}|}\sum_{\substack{a+|\underline{b}|=i\\a\leq m+1}}K^{(4)}_{s,i,a,\underline{b}}\Ndell{a}{b},
	\end{align*}
	where $K^{(4)}_{s,i,a,\underline{b}}$ are smooth away from the origin. Finally, note that
	
	\begin{align*}
		\rdell\dell_{s}\rdell^{m-1}\angdell^{\underline{n}}-\dell_{s}\Ndell{m}{n}=\comm{\rdell}{\dell_{s}}\Ndell{m-1}{n}=-\dell_{s}\Ndell{m-1}{n},
	\end{align*}
	and using $(\ref{commutator-rectangular-rad})$ once again, we finally get
	\begin{align*}
		\comm{\Ndell{m}{n}}{\dell_{s}}=\sum_{i=1}^{m+|\underline{n}|}\sum_{\substack{a+|\underline{b}|=i\\a\leq m+1}}\mathcal{K}_{s,i,a,\underline{b}}\Ndell{a}{b},
	\end{align*}
	for some $\mathcal{K}_{s,i,a,\underline{b}}$ smooth away from the origin.
	
	If $m=0$, write $\angdell^{\underline{n}}=\angdell_{j'l'}\angdell^{\underline{n}'}$, with $|\underline{n}'|=|\underline{n}|-1$. We obtain
	\begin{align*}
		\comm{\angdell^{\underline{n}}}{\dell_{s}}&=\angdell_{j'l'}\dell_{s}\angdell^{\underline{n}'}-\dell_{s}\angdell^{\underline{n}}+\sum_{i=1}^{m+|\underline{n}|}\sum_{\substack{a+|\underline{b}|=i\\a\leq 1}}K^{(5)}_{s,i,a,\underline{b}}\Ndell{a}{b}\\
		&=\comm{\angdell_{j'l'}}{\dell_{s}}\angdell^{\underline{n}'}+\sum_{i=1}^{m+|\underline{n}|}\sum_{\substack{a+|\underline{b}|=i\\a\leq 1}}K^{(5)}_{s,i,a,\underline{b}}\Ndell{a}{b},
	\end{align*}
	at which point we once again invoke $(\ref{commutator-rectangular-ang})$ and $(\ref{rectangular-as-ang-rad})$ to obtain the desired statement. The proof for the second part of the lemma is analogous.
\end{proof}

\section{Transformation of Derivatives}\label{transformation-of-derivatives}
In this appendix we record how $\Ndell{m}{n}$ can be written as a sum of $\nabla^{\underline{k}}$ on an appropriate sub-domain of the unit ball $B_{1}$, and vice versa.

\begin{lemma}\label{Ndell-to-Dk-lemma}Let $B_{1}$ be the closed unit ball in $\mathbb{R}^{3}$, and let $\Ndell{m}{n}$ and $\nabla^{\underline{k}}$ be defined as in $(\ref{Ndell-def})$ and $(\ref{Dk-def})$. Then we have
	\begin{align}
	\Ndell{m}{n}=\sum_{|\underline{k}|=0}^{m+|\underline{n}|}\mathcal{Q}_{\underline{k}}\nabla^{\underline{k}},\label{Ndell-to-Dk-statement}
	\end{align}
	for some $\mathcal{Q}_{\underline{k}}$, smooth on $B_{1}$.
\end{lemma}

\begin{proof}
	This is a simple consequence of our definitions of $\rdell$, given in (\ref{radial-derivative-def}) and $\angdell$, given in (\ref{angular-derivative-def}). In rectangular coordinates, they are given as
	\begin{align*}
	\rdell&=x.\nabla\\
	\angdell_{ij}&=x_{i}\dell_{j}-x_{j}\dell_{i}.
	\end{align*}
	Noting that $x$ is a smooth function on the ball, we get the desired statement by using these equalities in $\Ndell{m}{n}$ and expanding the expression.
\end{proof}

There is also a partial converse to Lemma $\ref{Ndell-to-Dk-lemma}$.

\begin{lemma}\label{Dk-to-Ndell-lemma}Let $B_{1}$ be the closed unit ball in $\mathbb{R}^{3}$, and let $\Ndell{m}{n}$ and $\nabla^{\underline{k}}$ be defined as in $(\ref{Ndell-def})$ and $(\ref{Dk-def})$. Then we have
	\begin{align}
	\nabla^{\underline{k}}=\sum_{m+|\underline{n}|=0}^{|\underline{k}|}\mathcal{Z}_{(m,\underline{n})}\Ndell{m}{n},\label{Dk-to-Ndell-statement}
	\end{align}
	for some functions $\mathcal{Z}_{(m,\underline{n})}$, smooth on any region removed from the origin in $B_{1}$.
\end{lemma}

\begin{proof}This is an application of identity (\ref{rectangular-as-ang-rad}), namely
	\begin{align*}
	 \dell_{i}=\frac{x_{j}}{r^{2}}\angdell_{ji}+\frac{x_{i}}{r^{2}}X_{r}.
	\end{align*}
The function $x/r^{2}$ is smooth on any region removed from the origin, so we can use this equality in $\nabla^{\underline{k}}$ and expand, giving us the desired result.
\end{proof}

\section{Hardy-Type Inequality and Sobolev Embeddings}\label{hardy-sobolev-embeddings}
One of the main tools we use is a higher order Hardy-type embedding which tells us that a weighted Sobolev space on a domain can be realised in a Sobolev space of lower regularity, in essence sacrificing regularity to remove degeneracy near the boundary. 

\begin{definition}\label{sobolev-norm}For a bounded domain $\mathcal{O}\subset \mathbb R^3$, and $s\in\mathbb{Z}_{\geq0}$, define the Sobolev space $H^{s}(\mathcal{O})$ by
	\begin{align*}
	H^{s}(\mathcal{O})=\left\{F\in L^{2}(\mathcal{O}): \nabla^{\underline{k}}F \ \text{ is weakly in } \ L^2(\mathcal O)
, \ \underline{k}\in\mathbb Z_{\ge0}^3, \ 0\leq |\underline{k}|\leq b\right\},
	\end{align*}
	with norm given by
	\begin{align*}
	\left\|F\right\|_{H^{s}(\mathcal{O})}^{2}=\sum_{|\underline{k}|=0}^{b}\int_{\mathcal{O}} \left|\nabla^{\underline{k}}F\right|^{2}dx.
	\end{align*}
	The definition of $H^{s}(\mathcal{O})$ can be extended to $s\in\mathbb{R}_{\geq0}$ by interpolation.
\end{definition}

\begin{definition}\label{weighted-sobolev-norm}For a bounded domain $\mathcal{O}\subset \mathbb R^3$, $\alpha>0$ and $b\in\mathbb{Z}_{\geq0}$, define the weighted Sobolev space $H^{\alpha,b}$ by
	\begin{align*}
	H^{\alpha,b}(\mathcal{O})=\left\{d^{\frac{\alpha}{2}}F\in L^{2}(\mathcal{O}): 
	\nabla^{\underline{k}}F \ \text{ is weakly in } \ L^2(\mathcal O, d^\alpha dx),
\ \underline{k}\in\mathbb Z_{\ge0}^3, \ 0\leq |\underline{k}|\leq b\right\},
	\end{align*}
	with norm given by
	\begin{align*}
	\left\|F\right\|_{H^{\alpha,b}(\mathcal{O})}^{2}=\sum_{|\underline{k}|=0}^{b}\int_{\mathcal{O}} d^{\alpha}\left|\nabla^{\underline{k}}F\right|^{2}dx,
	\end{align*}
	where $d=d(x,\dell\mathcal{O})$ is the distance function to the boundary on $\mathcal{O}$.
\end{definition}


Given these definitions, we have the following embedding.

\begin{lemma}{\label{hardy-type-embedding}} Let $\mathcal O$ be as above and let $b\in\mathbb Z_{>0}$ and $0<\alpha\le 2b$. Then the Banach spaces $H^{\alpha,b}(\mathcal{O})$ embeds continuously in $H^{b-\frac{\alpha}{2}}(\mathcal{O})$.
\end{lemma}



Finally, we state the Hardy-Sobolev bounds for $L^\infty$-norm in terms of our energy norms.
Recall the space $\X{b}$ and the set $\Y{b}{\ngrad}$ given in Definition $\ref{energy-function-spaces-definition}$, as well as $\psi$ given in~$\eqref{cutoff-function}$.

\begin{lemma}\label{L-infinity-energy-space-bound-no-weights} Let $a\in\mathbb{Z}_{\geq0}$ and $\underline{b}\in\mathbb{Z}_{\geq0}^{3}$, and let $a+|\underline{b}|=M$. Then
	\begin{align}
	\left\|\Ndell{a}{b}F\right\|_{L^{\infty}(\supp{\psi})}^{2}\lesssim \left\|F\right\|_{\X{\ceil*{\alpha}+2M+4}}^{2}.\label{L-infinity-energy-space-bound-no-weights-statement-1}
	\end{align}
\end{lemma}

\begin{lemma}
	\label{L-infinity-energy-space-bound-weights} Let $a\in\mathbb{Z}_{\geq0}$ and $\underline{b}\in\mathbb{Z}_{\geq0}^{3}$, and let $a+|\underline{b}|=M$. Then
		\begin{align}
		\left\|W^{\frac{a}{2}}\Ndell{a}{b}F\right\|_{L^{\infty}(\supp{\psi})}^{2}\lesssim \left\|F\right\|_{\X{\ceil*{\alpha}+M+6}}^{2}.\label{L-infinity-energy-space-bound-weights-statement-1}
		\end{align}

		\begin{align}
		\left\|W^{\frac{a}{2}}\nabla\Ndell{a}{b}F\right\|_{L^{\infty}(\supp{\psi})}^{2}\lesssim \left\|F\right\|_{\Y{\ceil*{\alpha}+M+6}{\ngrad}}^{2}.\label{L-infinity-energy-space-bound-weights-statement-2}
		\end{align}
\end{lemma}


The proofs of Lemmas~\ref{hardy-type-embedding}--\ref{L-infinity-energy-space-bound-weights} are standard and can be found in~\cite{KMP,Jang2014}.
\end{appendices}


\begin{thebibliography}{99}








\bibitem{CoSh2012}
\textsc{Coutand, D., Shkoller, S.}:
Well-posedness in smooth function spaces for the moving boundary three-dimensional compressible Euler equations in physical vacuum.
{\em Arch. Ration. Mech. Anal.}
{\bf 206}, no. 2, (2012) 515--616 





\bibitem{Dyson1968}
\textsc{Dyson F. J.}: 
Dynamics of a Spinning Gas Cloud. 
{\em J. Math. Mech.},
{\bf 18}, no. 1, (1968) 91--101.




\bibitem{Grassin98}
\textsc{Grassin, M.}:
\newblock Global smooth solutions to {E}uler equations for a perfect gas.
\newblock {\em Indiana Univ. Math. J.} \textbf{47}, (1998) 1397-1432 


\bibitem{HaJa}
\textsc{Had\v zi\'c, M.,  Jang, J.}:
Nonlinear stability of expanding star solutions in the radially-symmetric mass-critical Euler-Poisson system.
{\em Comm. Pure Appl. Math.}, 
{\bf 71}, no. 5,  (2018) 827--891.

\bibitem{HaJa2}
\textsc{Had\v zi\'c, M., Jang, J.}:
Expanding large global solutions of the equations of compressible fluid mechanics.
{\em Inventiones Math.}, {\bf 214}, no. 3, (2018) 1205–1266

\bibitem{HaJa3}
\textsc{Had\v zi\'c, M., Jang, J.}:
A class of global solutions to the Euler-Poisson system.
{\em Available on ArXiv at: https://arxiv.org/abs/1712.00124}.




\bibitem{Jang2014}
\textsc{Jang, J..}:
Nonlinear Instability Theory of Lane-Emden stars.  
{\em Comm. Pure Appl. Math.}, {\bf 67}, no. 9, (2014) 1418--1465.

\bibitem{JaMa2009} \textsc{Jang, J., Masmoudi, N.}:
Well-posedness for compressible Euler equations with physical vacuum singularity, 
{\em Comm. Pure Appl. Math.} {\bf 62}, (2009) 1327--1385
%

 \bibitem{JM1} \textsc{Jang, J., Masmoudi, N.}:
Vacuum in Gas and Fluid dynamics. \textit{Proceedings of the IMA summer school on Nonlinear Conservation Laws and Applications}, Springer, (2011) 315-329

\bibitem{JM2012}\textsc{Jang, J., Masmoudi, N.}: Well and ill-posedness for compressible Euler equations with vacuum. \emph{J. Math. Phys.} \textbf{53} (11), (2012) 115625


\bibitem{JaMa2015}
\textsc{Jang, J., Masmoudi, N.}:
Well-posedness of compressible Euler equations in a physical vacuum.
{\em Communications on Pure and Applied Mathematics}
{\bf 68}, no. 1, (2015) 61--111 


\bibitem{KMP} \textsc{Kufner, A., Malgranda, L.,  Persson, L.-E.}
{\em The Hardy inequality.}
Vydavatelsk\' y Servis, Plzen, 2007. 








\bibitem{LY2} \textsc{Liu, T.-P., Yang, T.}:
Compressible flow with vacuum and physical singularity,
\textit{Methods Appl. Anal.} \textbf{7}, (2000) 495-509 


















\bibitem{Ov1956}
\textsc{Ovsiannikov, L. V.}:
New solution of hydrodynamic equations. 
{\em Dokl. Akad. Nauk SSSR}, Vol lll, N l, (1956) 47--49.


%

\bibitem{Ro} 
\textsc{Rozanova, O.}: 
Solutions with linear profile of velocity to the Euler equations in several dimensions. 
{\em Hyperbolic problems: theory, numerics, applications, Springer, Berlin,} (2003) 861--870

\bibitem{ShSi2017}
\textsc{Shkoller, S., Sideris, T. C.}:
Global existence of near-affine solutions to the compressible Euler equations.
{Preprint, arXiv:1710.08368}



\bibitem{Se1997}
\textsc{Serre, D.}:
Solutions classiques globales des \'equations d'Euler pour un fluide parfait compressible. 
{\em Annales de l'Institut Fourier} 
{\bf 47}, (1997) 139--153 






\bibitem{Sideris2014}
\textsc{Sideris, T. C.}: Spreading of the free boundary of an ideal fluid in a vacuum. 
{\em J. Differential Equations}, {\bf 257}(1), (2014) 1--14 

\bibitem{Sideris}
\textsc{Sideris, T., C.}: Global existence and asymptotic behavior of affine motion of 3D ideal fluids surrounded by vacuum,  
{\em Arch. Ration. Mech. Anal.} \textbf{225}, no. 1, (2017) 141--176 



\end{thebibliography}
\end{document}